\documentclass[11 pt, reqno]{amsart}
\usepackage[T1]{fontenc}
\usepackage[margin=1in]{geometry}
\usepackage{amssymb, amsmath, amsthm}
\usepackage{wasysym}
\usepackage{setspace}
\usepackage{cancel}
\usepackage[all]{xy}
\usepackage{empheq}
\usepackage{extpfeil}
\usepackage{tikz}
\usetikzlibrary{trees}
\usepackage{forest}
\usepackage{pdflscape}
\usepackage[shortlabels]{enumitem}
\usepackage{rotating} 
\usepackage{bbm}

\usepackage{libertine}
\usepackage[libertine]{newtxmath}
\usepackage{xcolor}
\colorlet{darkblue}{blue!80!black}
\usepackage[pagebackref, colorlinks=true, citecolor=darkblue, linkcolor=darkblue, urlcolor=darkblue, linktoc=section]{hyperref}

\makeatletter
\newtheorem*{rep@theorem}{\rep@title}
\newcommand{\newreptheorem}[2]{%
\newenvironment{rep#1}[1]{%
 \def\rep@title{#2 \ref{##1}}%
 \begin{rep@theorem}}%
 {\end{rep@theorem}}}
\makeatother

\numberwithin{equation}{section}

\newtheorem{theorem}{Theorem}[section]
\newreptheorem{theorem}{Theorem}
\newtheorem{lemma}[theorem]{Lemma}
\newreptheorem{lemma}{Lemma}
\newtheorem{proposition}[theorem]{Proposition}
\newtheorem{corollary}[theorem]{Corollary}
\newtheorem{conjecture}[theorem]{Conjecture}

\newtheorem{mainthm}{Theorem}

\theoremstyle{definition}
\newtheorem{remark}[theorem]{Remark}
\newtheorem{definition}[theorem]{Definition}

\newtheorem{example}[theorem]{Example}

\newtheorem{question}[theorem]{Question}

\graphicspath{{new-images/}, {figures/}}

\def\beq{\begin{eqnarray*}}
\def\eeq{\end{eqnarray*}}
\def\Q{\mathbb{Q}}

\def\R{\mathbb{R}}
\def\Z{\mathbb{Z}}

\def\N{\mathbb{N}}

\def\incl{\hookrightarrow}
\def\to{\rightarrow}

\def\id{\mathrm{id}}
\def\x{\times}
\def\d{\partial}
\def\phi{\varphi}

\def\Imm{\mathrm{Imm}}

\def\Homeo{\mathrm{Homeo}}

\def\Tot{\operatorname{Tot}}
\def\P{\mathcal{P}}

\def\K{\mathcal{K}}

\def\T{\mathcal{T}}

\def\Map{\mathrm{Map}}

\def\bx{\mathbf{x}}

\def\bt{\mathbf{t}}
\def\bu{\mathbf{u}}

\def\bL{\mathbf{L}}
\def\bK{\mathbf{K}}
\def\la{\langle}
\def\ra{\rangle}

\def\K{\mathcal{K}}

\def\SS{\mathfrak{S}}
\def\I{\mathbb{I}}

\def\CC{\mathcal{C}}
\def\tC{\widetilde{C}}
\def\tAM{\widetilde{AM}}
\def\PCact{\mathbb{P}\mathcal{C}act}
\def\Cact{\mathcal{C}act}

\def\Ov{\mathcal{O}v}
\def\tb{\widetilde{b}}

\def\quot{q}

\def\barE{\overline{E}}
\def\barK{\overline{\mathcal{K}}}

    \title{Cubes, cacti, and framed long knots}

    \author{Robin Koytcheff}
    \author{Yongheng Zhang}
    \email{koytcheff@louisiana.edu}
    \email{yzhang@amherst.edu}
    \address{Department of Mathematics, University of Louisiana at Lafayette, Lafayette, LA, USA 70504}
    \address{Department of Mathematics \& Statistics, Amherst College, Amherst, MA, USA 01002}

\makeatletter
\@namedef{subjclassname@2020}{%
  \textup{2020} Mathematics Subject Classification}
\makeatother

\begin{document}

\begin{abstract}
We define an action of the operad of projective spineless cacti on each stage of the Taylor tower for the space of framed $1$-dimensional long knots in any Euclidean space.  By mapping a subspace of the overlapping intervals operad to the subspace of normalized cacti, we prove a space-level compatibility of our action with Budney's little $2$-cubes action on the space of framed long knots itself.  
Our result improves upon previous joint work of the first author related to the conjecture that the Taylor tower for classical long knots is a universal Vassiliev invariant over the integers.  As a corollary, we reprove the nontriviality of a certain Browder bracket class first detected by Sakai.
\end{abstract}

\maketitle

\vspace{-1pc}
\tableofcontents

\section{Introduction}

This paper concerns spaces of knots and operads.  Studying a space $\K$ of knots in $\R^d$ generalizes knot theory to higher dimensions and from isotopy classes to the space level.  We study the Taylor tower, which is a sequence of spaces
$T_n\K$ arising from the functor calculus of Goodwillie and Weiss \cite{Weiss:1996, Weiss:1999, Goodwillie-Weiss}.  
It approximates $\K$ much like a Taylor series approximates a smooth function, via maps $\K \to T_n\K$ which fit into the commutative diagram below:
\[
\xymatrix@R=0.5pc{ 
& & & \K \ar[ddlll] \ar[ddl] \ar[dd] \ar@{}[ddrrr]^(.05){}="a"^(.5){}="b" \ar "a";"b" & & &\\
&&\dots  &&\hspace{-1pc} \dots &&\\
\ast \simeq T_0\K & \dots \ar[l] &T_{n-1}\K \ar[l] & T_n\K \ar[l] & &\dots \ar[ll]  &
}
\]
An operad $\P$ consists of a sequence of spaces $\P(m)$ which parametrize operations with $m$ inputs and one output.  Our work mainly concerns $E_2$-operads, which are closely related to $2$-fold loop spaces.  A space with an action of an $E_2$-operad has a multiplication that is homotopy commutative, though two such homotopies may not admit a homotopy between them.  We develop an $E_2$-operad action on each $T_n\K$ where $\K$ is the space of framed $1$-dimensional long knots in $\R^d$.

Budney, Conant, Scannell, and Sinha \cite{BCSS:2005} conjectured that at the level of path components, the evaluation map ${ev}_{n+1}$ from the space $\K$ of long knots in $\R^3$ to $T_{n+1}\K$ is a universal additive abelian-group-valued finite-type (a.k.a.~Vassiliev) invariant of order $n$.
Around the same time, Voli\'c \cite{Volic:FT} established a homological version of this statement over $\Q$.
Later, work of Budney, Conant, Sinha, and the first author \cite{BCKS:2017} established that ${ev}_{n+1}$ is an additive abelian-group-valued finite-type invariant of order $n$, and combining this result with the homotopy spectral sequence for the tower provided more evidence for the conjecture.  Shortly afterwards, Kosanovi\'c \cite{Kosanovic:EmbCalc} showed that each evaluation map is surjective on path components and proved the conjecture rationally, while Boavida de Brito and Horel \cite{Boavida-Horel:Galois} showed that it is true up to $p$-torsion for primes $p$ less than $n-1$.

The main result in that joint work of the first author \cite{BCKS:2017} required two key steps.  One step, known to Conant for some time, was to establish that the evaluation maps respect the equivalence relation on knots induced by the Vassiliev filtration.  The other was to verify that $\pi_0 T_n\K$ is an abelian group and that $\pi_0({ev}_n)$ is a monoid homomorphism.  The latter step was achieved by constructing an $E_1$-operad action on each stage $T_n\K$, compatible with $ev_n$, as well as a somewhat ad hoc homotopy to establish homotopy commutativity.  Technically, that last step involved framed long knots.  Here we improve upon that result by promoting that $E_1$-operad action and homotopy commutativity to a bona fide $E_2$-operad action.  

Our new action on the Taylor tower stages is of an operad of cacti, 
specifically the operad of spineless cacti, due to Kaufmann \cite{Kaufmann:2005}, which is equivalent to the little $2$-disks operad.  
Despite this equivalence, it is unlikely that there is an easily defined operad map directly between the cacti operad and the little $2$-disks operad (or the little $2$-cubes operad).  Operads of cacti and their relatives have appeared in connection with string topology \cite{Voronov:2005} and moduli spaces of Riemann surfaces \cite{KLP}.  They have also been used in or connected to solutions of multiple variants of Deligne's conjecture on the Hochschild complex of an associative algebra \cite{Kontsevich-Soibelman, McClure-Smith:Deligne, McClure-Smith:2004, Kaufmann:2008, Salvatore:2009, Kaufmann-Schwell, Ward:2012}.  

We write $\K_d^{fr}$ for the space of framed $1$-dimensional long knots in $\R^d$.  
See Definition \ref{D:framed-long-knots} for details.  See Section \ref{S:mapping-space-models} 
for a description of the model that we use for $T_n \K_d^{fr}$.  This is our first main result:

\begin{mainthm}
\label{T:A}
For each $n\geq 0$ and each $d\geq 3$, the operad $\PCact$ of projective spineless cacti acts on 
the $n$-th stage $T_n \K_d^{fr}$ of the Taylor tower for $\K_d^{fr}$.  
This action is compatible with the projections $T_n \K_d^{fr} \to T_{n-1} \K_d^{fr}$.
The multiplication, Browder operation, and Dyer--Lashof operations in homology resulting from the action are compatible with those on $\K_d^{fr}$ from Budney's action of the little $2$-cubes operad on $\K_d^{fr}$.
\end{mainthm}

Fully parsing the compatibility in Theorem \ref{T:A} requires the next result, which we view as our second main theorem.  We write $\SS_m$ for the symmetric group on $m$ letters.

\begin{mainthm}
\label{T:B}
There is a symmetric sequence $\Ov^1$ of {\em normalized overlapping intervals} with a map from the little $2$-cubes operad $\CC_2$ and a map to the operad $\PCact$ of projective spineless cacti, both of which are 
$\SS_m$-equivariant homotopy equivalences in each arity $m$:
\[
\CC_2(m) \xrightarrow{\simeq} \Ov^1(m) \xrightarrow{\simeq} \PCact(m).
\]
\end{mainthm}

In more detail, for each $m$ we will have $\SS_m$-equivariant homotopy equivalences
\begin{equation}
\label{eq:equivalences-of-operad-spaces}
\CC_2(m) \xtwoheadrightarrow{} \Ov(m) \xhookleftarrow{} \Ov^1(m) \xtwoheadrightarrow{} \Cact^1(m) \xhookrightarrow{} \PCact(m)
\end{equation}
where $\Ov$ is the operad of overlapping intervals introduced by Budney, and where $\Cact^1$ denotes Kaufmann's symmetric sequence of normalized spineless cacti.  The two outer maps above and the fact that they are equivalences are not new, so the content of Theorem \ref{T:B} is that there is a symmetric sequence $\Ov^1$ with outgoing maps as in diagram \eqref{eq:equivalences-of-operad-spaces} which are $\SS_m$-equivariant homotopy equivalences.  
It is a purely operadic statement and may be of independent interest from this perspective.  
Indeed, Question \ref{Q:infinity-operad} concerns potential higher structures on $\Ov^1$.  In addition, the connection between cacti and overlapping intervals from diagram \eqref{eq:equivalences-of-operad-spaces} allows us to formulate Conjecture \ref{conj:action-on-Omega-G} on two potentially homotopic actions of $E_2$-operads on the based loop space $\Omega G$ of a topological group $G$. 
Beware that neither $\{\Ov^1(m)\}_{m\in \N}$ nor $\{\Cact^1(m)\}_{m\in \N}$ forms an operad.  
The purpose of the former is to be able to compare 
 the two actions in Theorem \ref{T:A} in any fixed arity.

Although the definition of the action in Theorem \ref{T:A} may seem unwieldy, a basic conceptual idea underlies it.  Trees appear in both operad composition and elements of the operad of cacti.  Each Taylor tower stage can be modeled by a space of maps between (suitably compactified) configuration spaces.  These spaces form an operad called the Kontsevich operad.  The map of configuration spaces resulting from the action of a cactus is defined by setting its output to be an iterated operad composition in the Kontsevich operad using the tree underlying the cactus.   
The difficulty in defining an action of $\CC_2$ or $\Ov$ on the Taylor tower is due to the fact that elements of these operads do not give rise to trees.  
The most important idea in proving the compatibility in Theorem \ref{T:A} is a homotopy that shrinks to infinitesimal size, which cannot be done at the level of knots but can be done at the level of compactified configuration spaces.  

As an application of Theorem \ref{T:A}, we recover a result of Sakai \cite{Sakai:Nontrivalent} that the Browder bracket of certain cycles in $\K_d^{fr}$ is nontrivial.  
Though one could possibly circumvent our action of cacti in defining a compatible bracket on $T_n\K_d^{fr}$ and thus in proving Theorem \ref{T:C},
we think the action of cacti provides a natural 
framework for our proof.

\begin{mainthm}
\label{T:C}
Let $d \geq 3$ be odd.  Let $e  \in H_{d-3}\K_d^{fr}$ and $f \in H_{2(d-3)}\K_d^{fr}$ be cycles which map to bracket expressions $[x_1,x_2]$ and $[x_1, x_3][x_2, x_4]$ in the homology spectral sequence for the cosimplicial model for $\K_d^{fr}$.  Then the Browder bracket $[e,f] \in H_{3d-8}\K_d^{fr}$ is nontrivial.
\end{mainthm}

Actually, Sakai's proof applies only to $d>3$, and he treats the case $d=3$ separately \cite{Sakai:Integral}.  In contrast, our proof involves only a small nuance for $d=3$.  On the other hand, for $d=3$, this result is merely the nontriviality of the Gramain loop of any long knot $f$ with nonzero Casson knot invariant, and the Gramain loop has long been known to be nontrivial for any nontrivial long knot \cite{Gramain:1977}.

Finally, we mention potentially related work.
Salvatore \cite{Salvatore:Knots} defined an $E_2$-action on a model of 
$T_\infty\K_d^{fr}$, which for $d\geq 4$ is equivalent to $\K_d^{fr}$.   
He did this by building on work of Sinha \cite{Sinha:Operads} on long knots modulo immersions, which in turn relied on work of McClure and Smith \cite{McClure-Smith:2004}.  
We leave for potential future work the conjecture that Salvatore's action agrees with Budney's; 
Theorem \ref{T:A} gives hope for a proof, since the McClure--Smith constructions are related to cacti.  
Similarly, connections between our work and double deloopings of the space of long knots modulo immersions by Dwyer and Hess \cite{Dwyer-Hess:Knots} and Turchin \cite{Turchin:Delooping} (the latter of which also applies to its Taylor tower stages) are not explored here.


\subsection{Organization of the paper}
%
In Section \ref{S:operads}, we review trees and operads; the operad of overlapping little cubes and its action on spaces of framed long knots; and symmetric sequences and operads of spineless cacti, including a cellular structure on their spaces, which uses the link between cacti and trees.
In Section \ref{S:Taylor-tower}, we review the mapping space models that we use for the $n$-th Taylor tower stage.  
We describe a map from the spatial to the infinitesimal model, which refines the previous joint work of the first author \cite{BCKS:2017}.  We use this map to define the evaluation map to the infinitesimal model.
In Section \ref{S:action-def}, we establish part of Theorem \ref{T:A} by defining an action of cacti on each Taylor tower stage and proving that it satisfies the required conditions.
In Section \ref{S:normalized-overlapping}, we prove Theorem \ref{T:B} by first defining the space of normalized overlapping intervals and a map to the space of normalized cacti.  We then show that this map and the inclusion into the space of overlapping intervals are equivalences.  We pose questions related to this structure.
In Section \ref{S:compatibility}, we complete the proof of Theorem \ref{T:A} by using the action maps and the maps from Theorem \ref{T:B}.  We show that the relevant diagrams commute up to homotopy by constructing such  homotopies.  
In Section \ref{S:applications}, we prove Theorem \ref{T:C} which recovers the result of Sakai that a certain Browder bracket in the space of framed long knots is nontrivial.  
In Section \ref{S:future}, we briefly outline how one might generalize that proof to other examples.

\subsection{Acknowledgments}
The first author was supported by the Louisiana Board of Regents Support Fund, Contract Number LEQSF(2019-22)-RD-A-22 and by the National Science Foundation, Award No.~DMS-2405370.  
He thanks Dev Sinha for directing him to and sharing the PhD thesis of Pelatt.
The second author benefitted from conversations with Michael Ching, David Gepner, Philip Hackney, Ralph Kaufmann, Jason Lucas, James McClure, Jeremy Miller, Alexandru Suciu, Botong Wang, and Benjamin Ward.  Graphics were produced using Inkscape and Mathematica.  We thank an anonymous referee for helpful suggestions that have improved the quality of the paper.

\section{Operads and related structures}
\label{S:operads}

We now review the operads most central to this work.  
In Section \ref{S:trees-and-operads}, we start with trees, which appear in both operad structures and elements of the operad of cacti; we then describe the connection between operads and trees.  The overlapping intervals operad and its action on the space of framed long knots are reviewed in Section \ref{S:ov-int-knots}.   Several variants of spaces of spineless cacti are described in Section \ref{S:spineless-cacti}, and two approaches to them are covered.  Cellular structures on spaces of cacti are reviewed in Section \ref{S:CW-structure-on-cacti}, where a certain tree is seen to underlie any cactus.

\subsection{Planted trees and operads}
\label{S:trees-and-operads}

\subsubsection{Planted trees}
\label{S:trees}

A {\em tree} is a finite contractible graph.  
A {\em rooted tree} is a tree $T$ with a choice of distinguished vertex $r$, called the {\em root}.  
One can view a rooted tree as a directed graph by directing each edge away from $r$.  
We consider rooted trees with a linear ordering of the set of outgoing edges at each vertex.
This data is the same as an isotopy class of embedding the tree into the plane.  
We call such a tree a {\em planted tree}, and we draw such trees in the upper half-plane with the root lying on the horizontal axis.  
Some authors use the terms {\em ordered tree} or {\em planar rooted tree} instead.

Let $v$, $v_1$, and $v_2$ be vertices in a planted tree $T$.  We say $v_2$ lies {\em above} $v_1$ (and $v_1$ lies {\em below} $v_2$) if the path from $r$ to $v_2$ passes through $v_1$.  We say $v_2$ lies {\em directly above} $v_1$ (and $v_1$ lies {\em directly below} $v_2$) if there is an edge directed from $v_1$ to $v_2$.  If $v$ has an outgoing edge to $v_1$ ordered before an outgoing edge to $v_2$, we say $v_1$ lies to the {\em left} of $v_2$ (and $v_2$ lies to the {\em right} of $v_1$).  A {\em branch} starting at $v$ is the subgraph of $T$ on the vertices whose path from the root goes through $v$.  The {\em arity} of $v$, denoted $|v|$, is the number of outgoing edges from $v$.  Planted trees will appear in two contexts below, each with its own set of extra data or requirements.

\subsubsection{Operads}
\label{S:operads-sub}

We assume familiarity with operads and operad actions. 
We view a (topological) operad $\P$ as a collection of spaces $\{\P(m)\}_{m \in \N}$ with a right action of the symmetric group $\SS_m$.  We write $\circ_i$ for the {\em insertion} or {\em partial composition} maps $\P(m) \x \P(k) \to \P(m+k-1)$, where $i \in \{1\dots, m\}$.  
We need to also view an operad as a functor from a category of trees.

\begin{definition}
\label{D:upsilon}
Let $\Upsilon$ be the following category.
An object is a planted tree with a univalent root and a designation of some non-root univalent vertices as {\em leaves}; all other vertices (including any remaining univalent vertices) are called {\em internal vertices}.
A morphism collapses any number of edges not incident to the root or any leaf.
\end{definition}

An operad $\P$ gives rise to a functor $\P: \Upsilon \to \T op$ which encodes all the information of $\P$ except for the symmetric group actions.   
On objects, it is given by
\[
\P: T \mapsto \prod_{i} \P(k_i)
\]
where the product is taken over internal vertices $i$ of $T$, and where $k_i = |i|$.  
A morphism that collapses a single edge is sent to the product of an insertion operation $\circ_j$ with identity maps on the factors associated to vertices not involved in the collapse.  
For any tree $T$ with $m$ leaves, there is a morphism from $T$ to the {\em corolla} $\gamma_m$, which is the tree with  a root, one internal vertex, and $m$ leaves.  It is unique by the associativity condition on an operad.
Define $\bigcirc_T$ as the associated iterated composition associated to this morphism:
\begin{equation}
\label{Eq:big-circ-T}
(T \to \gamma_m) \overset{\P}{\longmapsto}  \left( \bigcirc_T: \prod_i \P(k_i) \to \P(m)\right).
\end{equation}

More details on operads can be found in general references on the subject \cite{goils, MarklShniderStasheff, McClure-Smith:Intro, Fresse:HtpyOfOperads}, as well as in references more specific to our focus \cite{Sinha:Operads, Salvatore:2009, Budney:Cubes, Budney:Splicing}.

\subsection{The overlapping intervals operad and its action on long knots}
\label{S:ov-int-knots}

We assume familiarity with the little $d$-cubes operad $\CC_d$ and the fact that an action of it induces the Browder operation and Dyer--Lashof operations.  We review Budney's \cite{Budney:Splicing} operad $\CC_d'$ of overlapping $d$-cubes for $d=1$, where we use instead the notation $\Ov$, as well as its action on $\K_d^{fr}$.  Let $\I:=[-1,1]$.

\begin{definition}
\label{D:ov-int}
The operad $\Ov$ of {\em overlapping intervals} is defined as follows. 
The space $\Ov(m)$ consists of equivalence classes $[\mathbf{L}, \sigma]=[(L_1, \dots, L_m), \sigma]$ where $\sigma \in \SS_m$ and each $L_i$ is an orientation-preserving affine-linear embedding $\I \incl \I$.
The equivalence relation is that
\[
((L_1, \dots, L_m), \sigma) \sim ((J_1, \dots, J_m), \tau)
\]
\begin{quote}
if for all $i\in \{1,\dots, m\}$, $L_i= J_i$ and for all $i,j\in \{1,\dots,m\}$, $L_i(\mathring{\I}) \cap L_j(\mathring{\I}) \neq \varnothing$ implies that $\sigma^{-1}(i) < \sigma^{-1}(j)$ if and only if $\tau^{-1}(i) < \tau^{-1}(j)$.
\end{quote}
The right action of $\sigma \in \SS_m$  on $\Ov(m)$ is given by $\sigma [(L_1, \dots, L_m), \tau] := [(L_{\sigma(1)}, \dots, L_{\sigma(m)}), \sigma^{-1} \circ \tau]$.
Roughly, composition in $\Ov$ is given by composition of the embeddings of intervals and composition of the block permutation of $k_1 + \dots + k_m$ given by the element of $\SS_m$ with the permutations in $\SS_{k_i}$.
In precise terms, composition in $\Ov$ is defined by 
\begin{align*}
[(L_i)_{i=1}^m, \sigma] \x [(J^1_i)_{i=1}^{k_1}, \tau_1] \x \dots \x [(J^m_i)_{i=1}^{k_m}, \tau_m] \longmapsto 
[(L_i \circ J^i_j)_{(i,j) = (1,1)}^{(m, k_i)}, \  \beta]
\end{align*}
where $\beta$ is the permutation defined for $1\leq i \leq m$ and $1 \leq j \leq k_i$ by 
\[
\beta^{-1} \left( \sum_{\ell < i} k_i  + j \right) := 
\sum_{\ell < \sigma^{-1}(i)} k_{\sigma(\ell)}  + \tau_i^{-1}(j).
\]
\end{definition}

One may imagine the $L_i$ in $[(L_1,\dots,L_m), \sigma]$ as little $2$-cubes with infinitesimal thickness but different heights when they overlap, like playing cards lying on a table.  We think of $\sigma$ as mapping heights to labels of intervals.  Thus for any $[(L_1, \dots, L_m), \sigma] \in \Ov(m)$, we say that $L_i$ and $L_j$ are at the \emph{same height} if $[(L_1, \dots, L_m), \sigma] = [(L_1, \dots, L_m), (i \ j) \circ \sigma]$.  If $L_i$ and $L_j$ are at the same height, then their interiors are disjoint, and if $\sigma^{-1}(k)$ is between $\sigma^{-1}(i)$ and $\sigma^{-1}(j)$, then $L_k$ is at the same height as $L_i$ and $L_j$.  We say that $L_i$ is at the \emph{lowest height} if $L_i$ is at the same height as $L_{\sigma(1)}$.  Budney \cite[Proposition 2.6]{Budney:Splicing} shows that there is an equivalence of operads $\CC_2 \to \Ov$ roughly given by projecting 2-cubes $L_i$ to the first coordinate and using their heights to determine a permutation $\sigma$.

By an abuse of notation, for any $L: \I \to \I$ we will sometimes denote the image of $L$ by the same symbol.  
Similarly, we will sometimes also use $L$ to denote its unique extension to an affine-linear map $\R\to \R$.

\begin{definition}
\label{D:framed-long-knots}
Define the {\em space $\K_d^{fr}$ of  framed $1$-dimensional long knots in $\R^d$} as the space of smooth self-embeddings of $\R \x D^{d-1}$ which are the identity outside of $\I \x D^{d-1}$.
Equip $\K_d^{fr}$ with the $C^\infty$-topology.  We call the first coordinate axis the {\em long axis}.
\end{definition}

Some authors define the space of framed long knots as the subspace of elements $(f, (\alpha_t)_{t\in \I}) \in \K_d \x \Omega SO(d)$ such that the first vector of $\alpha_t$ is the unit derivative of $f$ at $t$.  A linearization argument shows that this space is homotopy equivalent to $\K_d^{fr}$.

For a single little interval $L: \I \to \I$ and a framed long knot $f \in \K_d^{fr}$, define 
\begin{equation}
\label{Eq:single-interval-action}
L \cdot f := (L \x \mathrm{id}_{D^{d-1}}) \circ f \circ (L \x \mathrm{id}_{D^{d-1}})^{-1}.
\end{equation}
Budney's $\Ov$-action on $\K_d^{fr}$ is then given by 
\begin{align}
\label{Eq:ov-action-on-knots}
\begin{split}
\Ov(m) \x (\K_d^{fr})^m & \xrightarrow{A_m} \K_d^{fr} \\
\bigl([(L_1, \dots, L_m), \sigma], (f_1, \dots, f_m)\bigr) & \longmapsto  
\bigl(L_{\sigma(1)} \cdot f_{\sigma(1)}\bigr) \circ \dots \circ \bigl(L_{\sigma(m)} \cdot f_{\sigma(m)}\bigr).
\end{split}
\end{align}
This action is well conveyed by Budney's illustrations \cite[Fig.~7]{Budney:Cubes} \cite[Examples 2.4 and 2.5]{Budney:Splicing}.

\subsection{Spineless cacti}
\label{S:spineless-cacti}

We now review symmetric sequences of spineless cacti, normalized spineless cacti, and projective spineless cacti, which were introduced by Kaufmann \cite{Kaufmann:2005, Kaufmann:2007}, building on work of Voronov \cite{Voronov:2005}. 
In Section \ref{S:symm-seq-cacti}, we discuss them informally and intuitively, fixing some terminology along the way.  
In Section \ref{S:partitions-and-coend}, we give two alternative definitions due to Salvatore \cite{Salvatore:2009}, which facilitates further definitions and proofs.  They are in terms of partitions and coendomorphisms of the circle respectively.  They apply to normalized and projective cacti, which are the only variants we use directly.  


\subsubsection{Symmetric sequences of spineless cacti}  
\label{S:symm-seq-cacti}

A {\em (spineless) cactus} $C$ is a tree-like arrangement of $m$ labeled topological circles in the plane, called {\em lobes}.
It has a basepoint $b$ called the {\em global root} which we draw as a square vertex.  We call a point where lobes intersect an {\em intersection point} and draw these as round vertices.  Each lobe has a length partitioned among the arcs between either two intersection points or an intersection point and $b$.  The image of a loop traversing each circle in a fixed direction is called the {\em perimeter} of $C$.  An operadic insertion map $\circ_i$ is roughly given by inserting a cactus into a lobe, as in Figure \ref{F:cact-comp}.

\begin{figure}[h!]
	\includegraphics[scale=0.35]{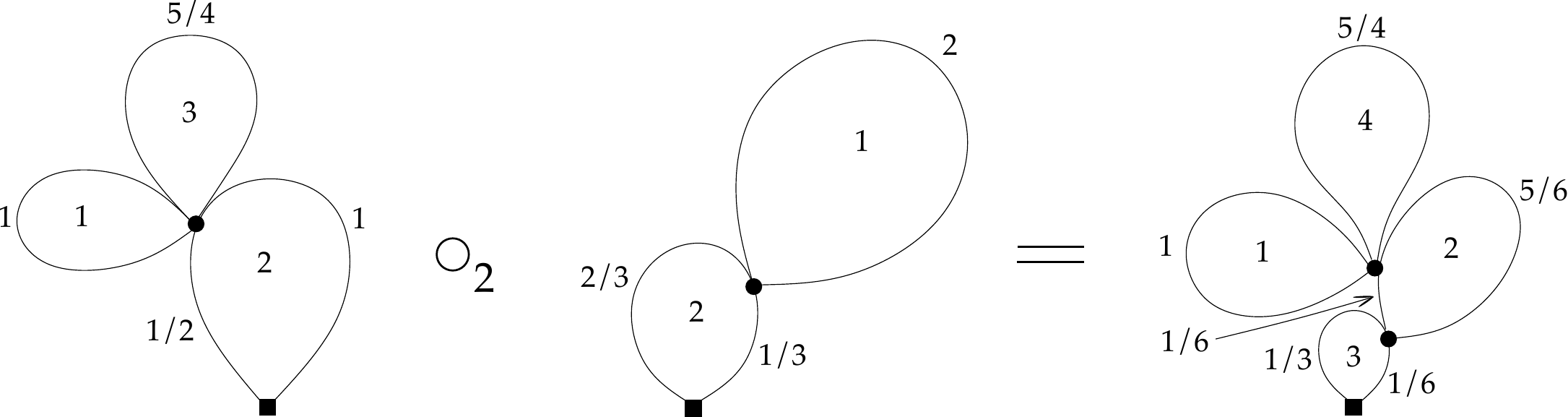}
	\caption{Shown above is an example of operadic insertion $\circ_2: \Cact(3)\times \Cact(2)\rightarrow \Cact(4)$.  Working modulo the total length of the perimeter of each cactus, we obtain the operation $\circ_2: \PCact(3)\times \PCact(2)\rightarrow \PCact(4)$. }
	\label{F:cact-comp}
\end{figure}

We consider three symmetric sequences $\{\Cact^1(m)\}_{m\in \N}$, $\{\Cact(m)\}_{m\in \N}$, and $\{\PCact(m)\}_{m \in \N}$ of spineless cacti, where the first and third are more central to our results, while the second serves as an intermediary.  The differences lie in the restrictions on the lengths of the lobes:
\begin{itemize}  
\item 
In an element of $\Cact^1(m)$, called a {\em normalized (spineless) cactus}, each lobe has length 1.  
Here composition is not associative, 
though $\Cact^1$ can be equipped with the structure of an infinity operad \cite{BCLRRW}, i.e., a structure where the operad relations hold only up to coherent homotopy.
\item 
In an element of $\Cact(m)$, called a {\em (spineless) cactus}, the lengths of the lobes are arbitrary positive real numbers.  
Here composition is associative, but any element in $\Cact(1)$ acts as an identity only on the right.  
\item 
The space $\PCact(m)$, whose elements are called {\em projective (spineless) cacti}, is the quotient of $\Cact(m)$ by the action of $(0,\infty)$ that rescales the lengths of all $n$ lobes.  Equivalently, one can define $\PCact(m)$ by requiring the sum of the $m$ lengths to be a fixed positive number.  
These spaces form an operad.
\end{itemize}
The spaces $\Cact(m)$ and $\PCact(m)$ are topologized to allow continuous variation of the length of each lobe and continuous variations of the partition of that length among the arcs.  In particular, with $\R_+:=(0,\infty)$, we have 
$\Cact(m) \cong \Cact^1(m) \x \R_+^m$, where a point in $\R_+^m$ specifies the lengths of the lobes.
We also have $\PCact(m) \cong \Cact^1(m) \x {\mathring\Delta}^{m-1}$, where a point in ${\mathring\Delta}^{m-1} = \{(t_1,\dots,t_m): t_i > 0 \text{ and } t_1 + \dots + t_m=1\}$ specifies the lengths of the lobes.  
For any of these three spaces, a permutation $\sigma \in \SS_m$ acts on the right by changing the labels $1,\dots, m$ on the lobes to $\sigma^{-1}(1), \dots, \sigma^{-1}(m)$.  
It can be viewed as pre-composing a bijection from $\{1,\dots,m\}$ to the lobes by $\sigma$, much like for the actions on cubes and overlapping intervals.

\subsubsection{Cacti via partitions of $S^1$ and $\mathrm{CoEnd}_{S^1}(n)$}
\label{S:partitions-and-coend}
Another approach is to view these spaces inside the coendomorphism operad $\mathrm{CoEnd}_{S^1}$.  Salvatore \cite{Salvatore:2009} described $\Cact^1(n)$ as a space of partitions of $S^1$ into $n$ compact $1$-manifolds (with boundary) of equal length, with pairwise disjoint interiors.  
A normalized cactus gives rise to such a partition by first rescaling it to have total length 1 and then traversing the perimeter in a fixed direction, labeling each point by the label of the lobe at the corresponding time.  Conversely, to get a normalized cactus from such a partition of $S^1$, one scales the circle to have circumference $n$ and then traces out arcs of lobes according to the labels of the $1$-manifolds encountered as one traverses the circle.

Salvatore then describes an embedding $\Cact^1(m) \incl \mathrm{CoEnd}_{S^1}(m):= \Map(S^1, (S^1)^{\x m})$.  Viewing an element of  $\Cact^1(m)$ as a partition of $S^1$ into $1$-manifolds labeled $1, \dots, m$, we send it to the map $S^1 \to (S^1)^{\x m}$ which on part $i$ is constant on all factors except the $i$-th one; on that factor it is a degree-1 map of constant speed.  
Following the idea of \cite[Proposition 4.5]{Salvatore:2009}, we can define $\Cact^1(m)$ as the subspace of maps $g:S^1 \to (S^1)^{\x m}$ such that if $\pi_i:(S^1)^{\x m}\to S^1$ is the projection onto the $i$-th factor, then 
\begin{enumerate}
\item 
for each $i$, $\pi_i \circ g$ is a piecewise-linear, degree-1 map where each piece is either constant or has speed $m$;
\item 
there is a partition of $S^1$ into compact connected $1$-manifolds intersecting only at boundary points such that on each part,  
$\pi_i \circ g$ is constant for all indices $i$ except one, called the {\em special index}; and 
\item 
the list of special indices encountered when traversing $S^1$ contains no sublist of the form $(i,j,i,j)$ with $i\neq j$.
\end{enumerate}

Similarly, $\PCact(m)$ can be defined first as a space of partitions of $S^1$ into compact $1$-manifolds with pairwise disjoint interiors, but now with possibly different lengths.  
This space embeds into $\mathrm{CoEnd}_{S^1}(m)$.  
We can equivalently define $\PCact(m)$ as the subspace of $g:S^1 \to (S^1)^{\x m}$ satisfying the conditions (2) and (3) that we set for $\Cact^1(m)$, but with condition (1) replaced by the following weaker condition:
\begin{itemize}
\item[($1'$)]
for each $i$, there exists $v_i\in (0,\infty)$ such that $\pi_i \circ g$ is a piecewise-linear, degree-1 map where each piece is either constant or has speed $v_i$.
\end{itemize}
A key advantage of defining $\PCact$ in terms of $\mathrm{CoEnd}_{S^1}$ is that the operad structure on $\mathrm{CoEnd}_{S^1}$ can be used to define one on $\PCact$.  
Salvatore writes $\mathcal{F}$ for $\Cact^1$ and $\mathcal{C}'$ for an operad homeomorphic to $\PCact$.

\subsection{B/w planted trees and cellular structures on spaces of spineless cacti}
\label{S:CW-structure-on-cacti}

We now describe cell structures on these operads using certain planted trees, as in the work of Kaufmann and Schwell \cite{Kaufmann:2005, Kaufmann:2007, Kaufmann-Schwell}; see also the work of Kaufmann and the second author \cite{Kaufmann-Zhang:2017}.   

\begin{definition}[Category of b/w trees] 
\label{D:cat-of-bw-trees}
A {\em b/w planted tree} (or more briefly a {\em b/w tree}) is a planted tree $T$ with a bipartition of its vertex set $V(T)$ into a set $B(T)$ of black vertex and a set $W(T)$ of white vertices such that the root is a black vertex and every other univalent vertex is white.  
In such a tree, a {\em leaf} is any non-root univalent vertex.

Let $\T_m$  be the following category:
\begin{itemize}
\item 
The set of objects, denoted by abuse of notation as $\T_m$, consists of b/w trees $T$ with white vertices $w_1, \dots, w_m$.  For brevity, we may sometimes refer to a white vertex $w_i$ by just its index $i$.
\item 
The morphisms are defined as follows.  Let $T \in \T_m$, and let $w\in W(T)$.  Let $e_0$ be the edge entering $w$, and let $e_1, \dots, e_{|w|}$ be the outgoing edges from $w$, arranged in the order that $T$ is equipped with.  We obtain a new tree $T'$ by gluing cyclically consecutive edges $e_i$ and $e_{i+1}$ (and the black vertices that they join $w$ to), where subscripts are read modulo $|w|+1$.  If $i=0$ or $i=|w|$, then the resulting edge enters $w$, while otherwise it leaves $w$.  We call such a gluing an {\em angle collapse}, and we write $T' \angle T$.  More generally, for any $T,T' \in \T_n$ a morphism $T' \prec T$ is a composition of angle collapses.  
\end{itemize}
\end{definition}

Examples of b/w trees are shown in Figures \ref{F:TCt} and \ref{F:19tree}.  In keeping with our conventions for drawing cacti, we draw the root as a square vertex, while the remaining black (and all white) vertices are round.

\begin{definition}
For any $j\in \{0,\dots, m-1\}$, define 
\[
\T_m^j:=\{ T \in \T_n : |w_1| + \dots + |w_m|=j\}.
\]
We call an element of $\T_m^j$ a {\em $j$-dimensional tree}.
\end{definition}

Thus $\T_m = \coprod_{j=0}^{m-1} \T_m^j$.  An angle collapse lowers the dimension of a tree by 1, so $\prec$ is a partial order on the set of objects in $\T_m$.   A $0$-dimensional b/w tree is called a {\em corolla}. \\

We now describe a regular CW complex structure on $\Cact^1(m)$ by presenting it as the colimit of a functor $\mathcal{E}_m$ from the poset category $(\mathcal{T}_m,\prec)$ to the category of topological spaces.  The $j$-dimensional cells of $\Cact^1(m)$ are indexed by elements of $\T^j_m$, and the gluing data is encoded by angle collapses $T' \angle T$.

\begin{definition}[Cellular structure on $\Cact^1(n)$]
\label{D:cell-functor}
Define a functor $\mathcal{E}_n$ as follows.
\begin{enumerate}
\item On objects $T \in \mathcal{T}^j_n$, set $\mathcal{E}_m(T):=\Delta^{|w_1|} \times \cdots \times \Delta^{|w_m|}$.  
\item On morphisms, if $T' \angle T$, and $T'$ is obtained from $T$ by the angle collapse at the white vertex $w_k$ of the edge pair $(e_i,e_{i+1})$ (where $i \in \{0,1,\cdots,|w_k|\}$ and indices are read modulo $|w_k|+1$), then  
\[
\mathcal{E}_n(T' \angle T)=\mathrm{id}_{\Delta^{|w_1|}}\times \cdots\times \mathrm{id}_{\Delta^{|w_{k-1}|}}\times d^i\times \mathrm{id}_{\Delta^{|w_{k+1}|}} \times\cdots\times \mathrm{id}_{\Delta^{|w_m|}},
\]
where $d^i$ is the $i$-th coface map:
\begin{equation*}
\begin{array}{ccrcl}
d^i& : & \Delta^{|w_k|-1} & \longrightarrow & \Delta^{|w_k|}\\
     &  & (t_1,\cdots,t_{|w_k|}) & \longmapsto & (t_1,\cdots,t_{i},0,t_{i+1},\cdots,t_{|w_k|}).\\
\end{array}
\end{equation*}
\end{enumerate}
\end{definition}

\begin{proposition}[Cellular structure on $\Cact^1(m)$]
\label{P:cell-structure}
For each $m\geq 1$, there is a homeomorphism 
\[
\Cact^1(m) \cong \mathrm{colim}_{\T_m}\mathcal{E}_m.
\]
\end{proposition}

\begin{proof}[Sketch of proof]
To see that the right-hand side agrees with the definition of $\Cact^1(m)$ in terms of lobes as in Section \ref{S:symm-seq-cacti}, one identifies white vertices with lobes and black vertices with the basepoint and any intersection points.  If one views a point in a simplex as an ordered tuple of numbers $t_i$ that sum to 1, then these $t_i$ specify the lengths of the arcs along each lobe, traveling in a fixed direction from the local root of the lobe.  
Alternatively, in terms of partitions of $S^1$ as in Section \ref{S:partitions-and-coend}, the coordinates $t_1, \dots, t_{|w_k|}$ specify the lengths of the components of the compact $1$-manifold labeled by $k$.  
\end{proof}

Because $\PCact(m) \cong \Cact^1(m) \x \mathring\Delta^{m-1}$, we obtain a cellular structure on $\PCact(m)$ from the one on $\Cact^1(m)$.
In particular, we can use the following cell structure on 
$
\mathring{\Delta}^{m-1}=\{(t_1,\dots,t_m) : t_i> 0, \ t_1+\dots+t_m=1\}.
$
For each $i=1,\dots,m$, and $j=m, m+1, \dots$, consider the hyperplanes $t_i=\frac{1}{j}$. They divide $\mathring{\Delta}^{m-1}$ into regions which are the top dimensional cells.  
Each such hyperplane or an iterated intersection of them is divided by the other hyperplanes into regions which are the lower-dimensional cells.

\begin{definition}
\label{D:underlying-tree}
A cactus $C$ in $\Cact^1(m)$ (respectively $\PCact(m)$) has $T \in \T_m$ as its {\em underlying} b/w tree if $C$ corresponds to a point in the interior of $\mathcal{E}(T)$ (respectively the product of the interior of $\mathcal{E}(T)$ with $\mathring{\Delta}^{m-1}$).
\end{definition}

In other words, $C$ may belong to closed cells $\mathcal{E}(T)$ for various $T$, and the underlying b/w tree of $C$ corresponds to the closed cell of lowest dimension.  It is also the initial tree among these $T$, meaning that it has the fewest edges.
We may sometimes adopt terminology for trees when discussing cacti.

\begin{remark}
\label{R:eqv-of-cacti-as-K(B,1)-spaces}
%
There is a $\SS_m$-equivariant inclusion $\iota_m:\mathcal{C}act^1(m)\hookrightarrow \PCact(m)$ which sends a normalized cactus to a cactus in which the length of each lobe is $\frac{1}{m}$. 
The $\SS_m$-action on $\PCact(m)$ corresponds to the diagonal action on $\Cact^1(m) \x \mathring\Delta^{m-1}$. 
Via this inclusion, $\Cact^1(m)$ is a deformation retract of $\PCact(m)$. 
\end{remark}

\section{Models for the Taylor tower for spaces of long knots}
\label{S:Taylor-tower}

In Section \ref{S:configs-and-models}, we review the models we use for the Taylor tower stages.  First we recall the simplicial compactifications of (framed) configuration spaces and their operadic insertion maps.  These are used to build cosimplicial models for spaces of knots, out of which the mapping space models are constructed.  We consider both spatial and infinitesimal variants of these models.  In Section \ref{S:spatial-to-infinitesimal}, we carefully define a map from the spatial to the infinitesimal mapping space model, refining previous joint work of the first author \cite{BCKS:2017} to better suit our purposes and provide more conceptual clarity.  This map is needed because the spatial mapping space model admits a map from the space of knots, described in Section \ref{S:eval-maps}, whereas the infinitesimal mapping space model will admit an action of $\PCact$.

\subsection{Mapping space models for Taylor tower stages}
\label{S:configs-and-models}

\subsubsection{Compactified configuartion spaces and the Kontsevich operad}
\label{S:cpt-config} 
We recall various configuration spaces and maps between them, primarily following Sinha's work \cite{Sinha:Operads, Sinha:Cptfn}.
Our notations agree with previously used ones \cite{BCKS:2017} except where noted:
\begin{itemize}
\item $C_n(X)$ is the configuration space of distinct ordered $n$-tuples in a space $X$.  
\item $C_n\la M \ra$ is the simplicial compactification of $C_n(M)$, due to Sinha \cite{Sinha:Cptfn}, where $M$ is a manifold $M$, possibly with boundary and embedded in $\R^d$.  This space records the directions of collision when multiple points are at the same location in $M$.
\item $C_n^{fr}\la \R^d \ra := C_n\la \R^d \ra \x GL(d)^n$, and $C_n^{fr}\la \I^d \ra := C_n\la \R^d \ra \x GL(d)^n$.  
We primarily use $GL(d)$ instead of $O(d)$ \cite{BCKS:2017} because it facilitates the comparison of our cactus action to the cubes action on $\K_d^{fr}$.  In Section \ref{S:applications}, we use $O(d)$ instead to avail ourselves of its action on $S^{d-1}$.
\item $C_n\la \I^d, \d \ra$ is the subspace of $C_{n+2}\la \I^d, \d \ra$ where the first and last point are at the boundary points along the long axis.  Using a similar definition, $C_n\la \I, \d \ra \cong C_n \la \I \ra \cong \Delta^n$.  
\item
$C_n^{fr} \la \I^d, \d \ra$ is the subspace of $C_{n+2}^{fr}\la \I^d, \d \ra$ where the first and last point are at the boundary points along the long axis and the first and last frames are the identity element in $GL(d)$.
\item 
$\tC_n^{fr} \la \R^d \ra$ is the result of forgetting locations of points from $C_n^{fr} \la \R^d \ra$; alternatively, it is the result of taking the quotient by translation and scaling.  Defining $\tC_n^{fr} \la \I^d \ra$ similarly, we have 
$\tC_n^{fr} \la \I^d \ra \cong \tC_n^{fr} \la \R^d \ra$.  We call elements of these space {\em infinitesimal} framed configurations, whereas elements of the configuration spaces above are called {\em spatial} (framed) configurations.
\item 
There are maps 
$\circ_i : C_m^{fr} \la \I^d, \d \ra \x C_n^{fr} \la \I^d, \d \ra \longrightarrow C_{n+m-1}^{fr} \la \I^d, \d \ra$
and
$\circ_i : \tC_m^{fr} \la \R^d \ra \x \tC_n^{fr} \la \R^d \ra \longrightarrow \tC_{n+m-1}^{fr} \la \R^d \ra$ roughly given by inserting one configuration into the $i$-point of the other at infinitesimal scale and using the framing at the $i$-th point.  These maps make $\{\tC_m^{fr} \la \R^d \ra\}_{m \in \N}$ into an operad, called the {\em framed Kontsevich operad}.  
In either the spatial or infinitesimal setting, they are well illustrated by Sinha's pictures \cite[Figure 3.7]{Sinha:Cptfn} \cite[Figure 4.6]{Sinha:Operads} \cite[Figure 4.3]{Sinha:Top}.
\end{itemize}

For an element of a space above,
we use the term $i${\em-th configuration point} to mean all the required data involving the index $i$, that is, a location, a collection of direction vectors, and if applicable, a frame.

\subsubsection{Cosimplicial models for spaces of long knots}
\label{S:cosimplicial-models}

The mapping space models which will use for the Taylor tower stages come from cosimplicial models for spaces of long knots.  Cosimplicial models for spaces of long knots were first defined in work of Sinha \cite{Sinha:Top},
and the framed version was first defined by Salvatore \cite{Salvatore:Knots}. 
They are most easily described in the context of $\overline{\K}_d$, the space of long knots modulo immersions, meaning the homotopy fiber of the inclusion $\K_d \to \Imm_d$ which sends a long embedding to a long immersion.  

The cosimplicial space  $C_\bullet \la \I^d, \d \ra$ for $\overline{\K}_d$ has  $C_n \la \I^d, \d \ra$ as its $n$-th entry.  Each coface map $d^i: C_{n} \la \I^d, \d \ra \to C_{n+1} \la \I^d, \d \ra$, $i=0,\dots, n+1$, is given by doubling the $i$-th configuration point in the positive direction of the first coordinate axis.  Each codegeneracy map $s^i: C_n \la \I^d, \d \ra \to C_{n-1} \la \I^d, \d \ra$, $i=1,\dots, n$, is given by forgetting the $i$-th point.  
There are cosimplicial models $C'_\bullet\la \I^d, \d \ra$ and $C^{fr}_\bullet\la \I^d, \d \ra$ for $\K_d$ and $\K^{fr}_d$ respectively.  They are defined similarly, but with the additional data of a tangent vector in $S^{d-1}$ (respectively a frame in $GL(d)$ or $O(d)$) at each configuration point.  Each coface map adds a new point in the direction of the tangent vector (respectively the first vector in the frame).  We refer to these cosimplicial spaces as {\em spatial cosimplicial models}.

The reason for that terminology is that $\overline{\K}_d$ and $\K_d^{fr}$ can also be modeled by cosimplicial spaces $\tC_\bullet \la \R^d \ra$ and $\tC^{fr}_\bullet \la \R^d \ra$, respectively, which we call {\em infinitesimal cosimplicial models}.  The spatial and infinitesimal models for $\overline{\K}_d$ and $\K_d^{fr}$ admit an operadic description of the coface maps.  In the spatial model for $\overline{\K}_d$, $d^i$ is given by inserting $b$ into the $i$-th point, where $b$ is the image in $C_2\la \I^d \ra$, respectively, of the single point in $C_0\la \I^d, \d \ra$.  In the infinitesimal model for $\overline{\K}_d$, $d^1, \dots, d^{n}$ are defined using a similar element $\tb$ in $\tC_2\la \R^d \ra$, while $d^0$ and $d^{n+1}$ are given by inserting the $n$-point configuration into the first and last point of $b$ respectively.  For $\K_d^{fr}$, a similar description applies.

\begin{remark}
\label{R:AM-behavior-at-pm-1}
The coface maps $d^0$ and $d^n$ in the infinitesimal model can be viewed as insertions into $(\mp \infty, 0,\dots,0)$ respectively.  That is, each is the limit of insertion maps into an added basepoint which approaches one of these two points.  The equalities between $d^0$ and $d^n$ and these limits are crucial for the continuity of our action and the compatibility with the evaluation map on knots.
\end{remark}

\subsubsection{Mapping space models for spaces of long knots}
\label{S:mapping-space-models}

The mapping space models arise from Sinha's cosimplicial models \cite{Sinha:Top} and 
have been used to study finite-type knot invariants \cite{BCSS:2005, BCKS:2017}.  
Let $[n]:=\{0,1,\dots, n\}$, and let $\P_\nu[n]$ be the category whose objects are the nonempty subsets of $[n]$ and whose morphisms are inclusions.  The  {\em spatial mapping space model} ${AM}_n^{fr}$ (respectively the {\em infinitesimal mapping space model} $\widetilde{AM}_n^{fr}$) is the homotopy limit of the functor $\P_\nu[n] \to \mathcal{T}op$ obtained by pre-composing $\ C^{fr}_\bullet \la \I^d, \d \ra$ (respectively $\ \tC^{fr}_\bullet \la \R^d \ra$) by a canonical functor $\P_\nu[n] \to \Delta$.

The key point for us is the concrete descriptions of these models that follow from this definition.
Namely, $AM_n^{fr}$ is the subspace
\begin{equation}
\label{eq:AM-as-subspace-of-product}
AM_n^{fr} \subset \prod_{S: \, \varnothing \neq S \subseteq [n]} \Map(\Delta^{|S|-1}, C^{fr}_{|S|-1} \la \R^d \ra)
\end{equation}
of maps $(\phi_S)_{\varnothing \neq S\subseteq [n]}$ such that for every $T\subset S$ with $|T|=|S|-1$, we have $\phi_S \circ d^i = d^i \circ \phi_T$.

Since each $d^i$ is injective, an element in $AM_n^{fr}$ is determined by  $\phi := \phi_{[n]}: \Delta^n \to C_n^{fr} \la \R^d \ra$.  Thus we will sometimes view $AM_n^{fr}$ as a subspace of $\Map(\Delta^n, C_n^{fr}\la \R^d \ra)$, equipped with the compact-open topology.  

For projection maps in the tower, we can use the inclusion $\delta^n: [n-1] \incl [n]$.
By considering only those $\phi_S$ such that $S \subset [n-1]$, we obtain a map 
\[
AM_n^{fr} \to AM_{n-1}^{fr}
\]
which we view as a projection.  Using any other inclusion $\delta^i: [n-1]\to [n]$ in this definition yields a homotopic map because all the $d^i: \Delta^{n-1} \to \Delta^n$ are homotopic. 

A discussion analogous to that in the previous two paragraphs apply just as well to $\tAM_n^{fr}$.

We have both $AM_n^{fr} \simeq T_n \K_d^{fr}$ and $\tAM_n^{fr} \simeq T_n \K_d^{fr}$.  A proof for $AM_n^{fr}$ can be given by adapting Sinha's work \cite[Theorem 5.4]{Sinha:Top} to the framed setting, by replacing tangent vectors by frames.  For $\tAM_n^{fr}$, one must also replace spatial configurations by infinitesimal ones; see Proposition \ref{P:quot-is-equiv} below.  We need $\widetilde{AM}_n^{fr}$ for our cactus operad action, while ${AM}_n^{fr}$ is more closely related to long knots via the evaluation map.

\subsection{From the spatial to the infinitesimal mapping space model}
\label{S:spatial-to-infinitesimal}

The quotient map from $C_n \la \I^d, \d \ra$ to $\tC_n \la \I^d \ra = \tC_n \la \R^d\ra$ does not induce a map $AM_n^{fr} \to \widetilde{AM}_n^{fr}$.  
Indeed, consider the restriction of any $\phi \in AM_n^{fr}$ to the face $t_1=-1$: here the unit vector direction from the first configuration point $\phi(\bt)_1$ to any configuration point $\phi(\bt)_j$ with $1< j < n$ might not be $(1,0,\dots,0)$.  The same is true of the restriction of $\phi$ to the face $t_n=1$ and the unit vector  to $\phi(\bt)_n$ from any point $\phi(\bt)_j$ with $1 < j <n$.

Nonetheless, there is a map $\quot: AM_n^{fr} \to \widetilde{AM}_n^{fr}.$
We will define it by first enlarging the output configurations of a map $\phi\in AM_n^{fr}$ to include points on the long axis in $[-2,2]$, and then shrinking the original configuration to infinitesimal size so that the above mentioned vectors become horizontal.  

A map similar to $q$ was defined previously \cite{BCKS:2017} by including configuration points on the long axis all the way to $\pm \infty$ rather than shrinking the configuration.  
Although our exposition here is more elaborate, it provides conceptual consistency with a homotopy that we will use in Section \ref{S:compatibility} to establish the compatibility of our cactus action with the  $\CC_2$-action on $\K^{fr}_d$.
It also allows us to define our cactus action on aligned maps without going back and forth between the spatial and infinitesimal variants; see Remark \ref{R:conf-it}.

We start by defining a smooth map $H: \R^d \x [0,1) \to \R^d$ that encodes the shrinking process.  
To ensure that points can continuously exit or enter the infinitesimal configuration (by becoming arbitrarily far away within it), we interpolate between quadratic and linear rates of approach towards the origin.
More precisely, define $H$ for $\bx = (x_1,\dots,x_d) \in \R^d$ and $t\in [0,1)$ by 
\begin{equation}
\label{Eq:shrinking-H}
H(\bx,t) := 
\dfrac{(1-t)(1-\nu(\bx)t)}{1 - \mu(\bx)t} \bx 
\end{equation}
where $\mu$ and $\nu$ are smooth cutoff functions such that 
\begin{align*}
\mu(\bx)=
\left\{
\begin{array}{ll}
0
& \text{ if }  \lVert \bx \rVert \leq \sqrt{2} \\
1
 & \text{ if } 2 \leq \lVert \bx \rVert \\
\end{array}
\right.
\qquad \text{ and } \qquad
\nu(\bx)= 
\left\{
\begin{array}{ll}
1  & \text{ if }    \lVert \bx \rVert \leq 2 \\
0 & \text{ if }  4 \leq \lVert \bx \rVert .
\end{array}
\right.
\end{align*}
For example, we can use 
$\lambda(x) = \left\{ \begin{array}{ll} e^{-1/x^2} \text{ if } x>0 \\ 0 \text{ if } x \leq 0\end{array} \right.$ and
\begin{equation*}
\begin{split}
\mu(\bx) = \dfrac{\lambda(\lVert \bx \rVert - \sqrt{2})}{\lambda(\lVert \bx \rVert - \sqrt{2}) + \lambda(2 - \lVert \bx \rVert)}
\end{split}
\qquad \qquad
\begin{split}
\nu(\bx) = \dfrac{\lambda(4 - \lVert \bx \rVert)}{\lambda(4 - \lVert \bx \rVert) + \lambda(\lVert \bx \rVert - 2)}.
\end{split}
\end{equation*}
The bound $\sqrt{2}$ corresponds to the maximum distance from the origin to a point in $\I \x D^{d-1}$, which will be important when we combine this map with framed long knots in Section \ref{S:eval-maps} and especially in Section \ref{S:compatibility}.
Trajectories of various points under $H$ are illustrated in Figure \ref{F:H-trajectories}.

\begin{figure}[h!]
\includegraphics[width=\textwidth, height=3.5cm]{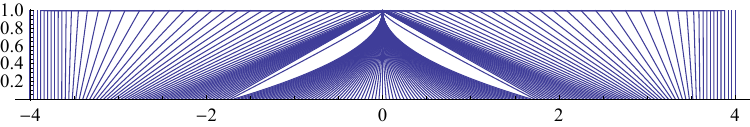}
\caption{
Above are trajectories of various points $x \in [-4, 4]$ under the map $H(x,t)$ for $t\in [0,1)$, where $d=1$.  The $x$-axis is the horizontal axis, the $t$-axis is the vertical axis.
}
\label{F:H-trajectories}
\end{figure}

For each $t\in [0,1)$, $H(-,t)$ is an embedding.  
Let  $H_n(-,t): C_n^{fr}\la \R^d\ra \to C_n^{fr}\la \R^{d}\ra$ be the map which applies $H(-,t)$ to each configuration point and applies the derivative $DH(-,t)$ followed by the Gram--Schmidt retraction $GL(d) \to O(d)$ to each frame.
Because each $H(-,t)$ is a scaling map, the effect on a frame is independent of $t$.  (The Gram--Schmidt process ensures that the frames do not approach the zero matrix.)
Although the limit as $t\to 1$ of $H(-,t)$ is not an embedding, $\lim_{t\to 1}H_n(-,t)$ exists, and it will be our shrinking map.  

We sketch the proof of continuity by describing the effect on a framed $n$-point configuration.  The independence of the effect on frames from $t$ means that we just have to consider the configuration points.
A point $\bx$ with $\lVert \bx \rVert>2$ is simply sent to $\lim_{t\to 1}H(\bx,t)$, which lies outside the origin.  
Points $\bx$ with $\lVert \bx \rVert < 2$ are sent to points in an infinitesimal configuration at the origin.  
The data needed for an infinitesimal configuration is a configuration modulo translation and scaling, but we describe a representative before taking this quotient.  
In this representative configuration, a point $\bx$ with $\lVert \bx \rVert \leq \sqrt{2}$ is sent to $\bx/\sqrt{2}$, while more generally $\bx$ with $\sqrt{2} \leq \lVert \bx \rVert < 2$ is sent to a point at distance $\lVert \bx \rVert/(1-\mu(\bx))$ from the origin along the ray in the direction of $\bx$.
The key point is that we ensure continuity by mapping the disk of radius $\sqrt{2}$ to the unit disk and the open disk of radius 2 onto all of $\R^d$ in this infinitesimal configuration.  A point $\bx$ with $\lVert \bx \rVert =2$ is mapped to a point at infinity in the direction of $\bx$ in the infinitesimal configuration.  
Thus a point can travel into or out of the infinitesimal configuration in a continuous path.

\begin{remark}
One can check that for example the map 
\[(\bx, t) \mapsto \bigl((1-\mu(\bx))(1-t)^2 + \mu(\bx)\nu(\bx)(1-t) + (1-\nu(\bx)) \bigr) \bx
\]
with $p$ and $q$ as above does {\em not} induce a continuous map on $C_n^{fr} \la\R^d\ra$.  This is because it fails the required continuity property noted for configuration points $\bx$ with $\sqrt{2} \leq \lVert \bx \rVert \leq 2$, even though it defines a smooth embedding for each $t\in [0,1)$ and interpolates between quadratic, linear, and constant functions of $t$ as $\lVert \bx \rVert$ varies.
\end{remark}

Next, any $\phi \in AM_n^{fr}$ extends to a map 
\[
\overline{\phi}: C_n\la \R \ra  \to C_n^{fr}\la{\R^d} \ra
\]
as follows. For $t\in \R$, define 
\[
\overline{t}=
\left\{
\begin{array}{ll}
-1 & \text{ if } t\leq -1 \\
t & \text{ if } t\in \I \\
1 & \text{ if } 1 \leq t
\end{array}
\right.
\]  
Then $\overline{\phi}(-\infty < t_1 \leq \dots \leq t_n < \infty) \in C_n^{fr}\la {\R^d} \ra$ is the class obtained from $\phi(\overline{t_1}, \dots, \overline{t_n})$ by replacing the $i$-th configuration point by $(t_i, 0, \dots,0)$ for all $i$ such that $t_i  \notin \I$.  
If $t_i=t_j \in \R\setminus\I$ and $i<j$, 
then the vector $v_{ij}$ specifying the direction of collision from point $i$ to point $j$ is $(1,0,\dots, 0)$.  

We are now ready to map from the spatial to the infinitesimal model.  Let $\tau: C_n \la \I \ra \to C_n \la \R \ra$ be the map induced by the embedding $\I \to \R$ given by $t \mapsto 2t$.

\begin{definition}
\label{D:AM-quot}
Define 
\begin{align*}
\quot: AM_n^{fr} &\to \widetilde{AM}_n^{fr}\\
\phi &\mapsto \quot(\phi)
\end{align*}
by setting $\quot(\phi)$ to be the composite 
\[
C_n \la \I \ra \overset{\tau}{\longrightarrow} C_n \la \R \ra \overset{\overline{\phi}}{\longrightarrow} C_n^{fr}\la {\R^d} \ra
\xrightarrow{\lim_{t\to 1} H_n(-,t)}
C_n^{fr}\la {\R^d} \ra 
\twoheadrightarrow
\tC_n^{fr} \la \R^d \ra.
\]
\end{definition}

A collection of times in $[-1/2, 1/2]$ is sent by $q(\phi)$ to a configuration in the image of $\phi$.  Times $t_i$ with $1/2 \leq |t_i| \leq 1$ are sent to configuration points on the long axis, and for $t_i=\pm 1$, the corresponding configuration points appear infinitely far from configuration points coming from the image of $\phi$.

\begin{proposition}
\label{P:quot-is-equiv}
The map $\quot$ is a homotopy equivalence.
\end{proposition}
\begin{proof}[Sketch of proof]
The map $q$ is homotopic to the map $\iota$ in previous joint work of the first author \cite[Definition 3.7]{BCKS:2017} because $\quot$ and $\iota$ differ by a homeomorphism of $C_n\la \R\ra$ induced by a reparametrization $\R \to \R$.
The map $\iota$ is a homotopy equivalence \cite[Proposition 3.8]{BCKS:2017}.  
A key ingredient in the proof is that the quotient $C_n \la\I^d, \d\ra \to \tC_n\la \I^d\ra$ is an equivalence,
as implied by work of Sinha \cite[Theorem 4.2]{Sinha:Operads}.
Since $\iota$ is a homotopy equivalence, so is $\quot$.
\end{proof}

\subsection{The evaluation map}
\label{S:eval-maps}

Recall from Definition \ref{D:framed-long-knots} that $\K_d^{fr}$ is the space of framed long knots in $\R^d$.
The derivative $Df(x)$ of $f\in \K_d^{fr}$ at any $x$ in the domain of $f$ is an element of $GL(d)$, using standard trivializations of the tangent bundles of the domain and codomain.  
Let $GS: GL(d) \to GL(d)$ be the composition of the Gram--Schmidt retraction followed by the inclusion $O(d) \to GL(d)$.

\begin{definition}
\label{D:ev-map}
We define the {\em evaluation map} 
\begin{align*}
 \K_d^{fr} &\to AM_n^{fr}\\
  f &\mapsto ev_n(f)
\end{align*}
by letting $ev_n(f)$ be the unique continuous extension of the map 
\begin{align*}
\mathring{\Delta}^n &\to C_n(\I^d) \x GL(d)^n\\
(-1 < t_1 < \dots < t_n < 1) &\mapsto 
\bigl( f(t_1,0^{d-1}), \dots, f(t_n, 0^{d-1}); \, GS(Df(t_1, 0^{d-1})), \dots, GS(Df(t_n, 0^{d-1}))  \bigr)
\end{align*}
to a map $\Delta^n \to C_n^{fr} \la \I^d, \d \ra$.
\end{definition}  

We may now compare knots to infinitesimal aligned maps via the composite $\quot \circ ev_n: \K_d^{fr} \to \widetilde{AM}_n^{fr}$.

\section{The cactus operad action on the Taylor tower for the space of framed long knots}
\label{S:action-def}

In this Section, we will prove part of Theorem \ref{T:A}. We define an action of the operad of spineless projective cacti on each Taylor tower stage in Definition \ref{D:action}.  Theorem \ref{T:action} is that this is a well defined operad action.  
In Proposition \ref{P:proj-compatible}, we prove compatibility with the projection maps in the tower.
The proof of the remainder of Theorem \ref{T:A} is given in Section \ref{S:compatibility} and requires Theorem \ref{T:B}, which will be proven in Section \ref{S:normalized-overlapping}.

\subsection{The definition of the action}

More precisely, we will now construct for each $m\geq 1$ continuous maps
\[
\alpha_m:\PCact(m)\times \left(\widetilde{AM}^{fr}_n\right)^m\to \widetilde{AM}^{fr}_n
\]
and show that they satisfy the defining conditions of an operad action.

Let $C\in \PCact(m)$ and $\phi^1, \dots, \phi^m \in \widetilde{AM}^{fr}_n$, where each $\phi^i=\{\phi^i_S\}_{\varnothing \neq S\subseteq [n]}$.   Let $\bt=(t_1, \dots, t_n)\in \Delta^n$, where each $t_i \in \I$ and $t_1 \leq \dots \leq t_n$.  Our present task is to produce out of these three pieces of data an equivalence class of configurations in $\tC_n^{fr}\la \R^d \ra$.

To do so, we need to set a handful of auxiliary definitions.  Specifically, for each $\ell \in \{1,\dots, m\}$ we will define the following objects in the order listed: 
\begin{itemize}
\item a nonempty subset $S_\ell(C,\bt) \subseteq [n]$, 
\item an element $\bt^{C,\ell} \in \Delta^{|S_\ell(C,\bt)|-1}$,
\item a continuous, monotone non-decreasing surjection $\rho_\ell(C):\I \to \I$, and 
\item an object $T(C,\bt)$ of the category of trees $\Upsilon$ from which operads induce functors.
\end{itemize}

In turn, some preliminary notation and discussion facilitate these definitions.
Fix a representative of $C$ by requiring the total lengths of its lobes to be the length of $\I$.  Then identify the coordinates $t_1, \dots, t_n \in \I$ with points on the perimeter of $C$, by starting at the root and traversing it clockwise (so as to coincide with the positive direction in $\I$ along the horizontal axis).

For a lobe $\ell \in \{1, \dots, m\}$, let $w$ be the number of lobes directly above $\ell$.
We start by defining positive integers $k_1 \leq \dots \leq k_{w+1}$ and nonnegative integers $j_1 \leq \dots \leq j_{w+1}$ depending on $\ell$.
If $w=0$, let $k_1$ be the index such that $t_{k_1}$ is the last $t_i$ encountered before lobe $\ell$, and let $j_1=0$.
Otherwise, let $k_p$ and $j_p$ be the indices such that for each $q\in \{1,\dots, w+1\}$, 
$t_{k_p+1}\leq \dots\leq  t_{k_p+j_p}$ are the $t_i$ lying on lobe $\ell$ and between the $(p-1)$-th and $p$-th local roots on $\ell$, where we consider the first and second occurrences of the local root of $\ell$ itself as the $0$-th and $(w+1)$-th local root.
That is, the $t_i$ on lobe $\ell$ are
\begin{equation}
\label{Eq:t_i-on-lobe-ell}
t_{k_1+1}\leq \dots\leq  t_{k_1+j_1}\bigg| 
t_{k_2+1}\leq \dots\leq t_{k_2+j_2}\bigg| 
\dots\bigg| 
t_{k_{w+1}+1}\leq \dots\leq t_{k_{w+1}+j_{w+1}}
\end{equation}
where each vertical bar corresponds to a lobe directly above $\ell$.  
Any $t_i$ at local roots pose some ambiguity, so we fix the following convention: 
we do not include any $t_i$ at the local roots of lobes directly above $\ell$ or to the right of $\ell$, but we include any $t_i$ at the first appearance of the local root of $\ell$ and, if there are no lobes to the right of $\ell$, any $t_i$ at its last appearance.  Informally, we put $t_i$ on lobes as high up and as far to the right as possible.

Beware that a ``part'' in the ``partitioned'' list \eqref{Eq:t_i-on-lobe-ell} can be empty; this happens when $j_p=0$, meaning that there are $t_i$ above $\ell$ on adjacent branches with no $t_i$ on $\ell$ between them.
The scenario where $k_{q} = k_p+j_p$ for a consecutive list of indices $q$ starting at $p+1$ happens when there are consecutive lobes above $\ell$ with no $t_i$ on them.

For the last preliminary notation, let $s_1,\dots,s_w$ be the coordinates of the first instances of the local roots of lobes above $\ell$.

Intuitively, with respect to the aligned map $\phi^\ell$ corresponding to lobe $\ell$, the $t_i$ in $\bt^{C,\ell}$ together with the local roots will be the actual input to $\phi^\ell$; those $t_i$ before them are at left infinity in the domain; those $t_i$ after them will be at right infinity in the domain; and each group of points of the form $t_{k_p+j_p+1}\leq \dots \leq t_{k_{p+1}}$ will have collided at the $p$-th local root that lies on $\ell$.

\subsubsection{The nonempty subset $S_\ell(C,\bt) \subset [n]$}
Define
\[
S_\ell(C,\bt):= \{k_1,k_1+1,\dots,k_1+j_1, \ k_2,k_2+1,\dots,k_2+j_2, \  \dots, \ k_{w+1},k_{w+1}+1,\cdots,k_{w+1}+j_{w+1}\}
\]
where we view $S_\ell(C,\bt)$ as a set, not a list, so some indices above may be repeated.
Since $S_\ell(C,\bt)$ contains at least $k_1$, it is a nonempty subset of $[n]$.
Conceptually, $S_\ell(C,\bt)$ is the subset of $[n]$ indexing the face of $\Delta^n$ where the $t_i$ on a branch starting at lobe $\ell$ have collided at a single point and where the $t_i$ to the left (respectively right) below lobe $\ell$ have collided with the left (respectively right) endpoint of $\I$.  See Figure \ref{F:mickey} for an example where $S_\ell(C,\bt)$ is listed for a fixed $C$ and $\bt$ and all values of $\ell$.

\begin{figure}[h]
	\raisebox{-8pc}{\includegraphics[scale=0.45]{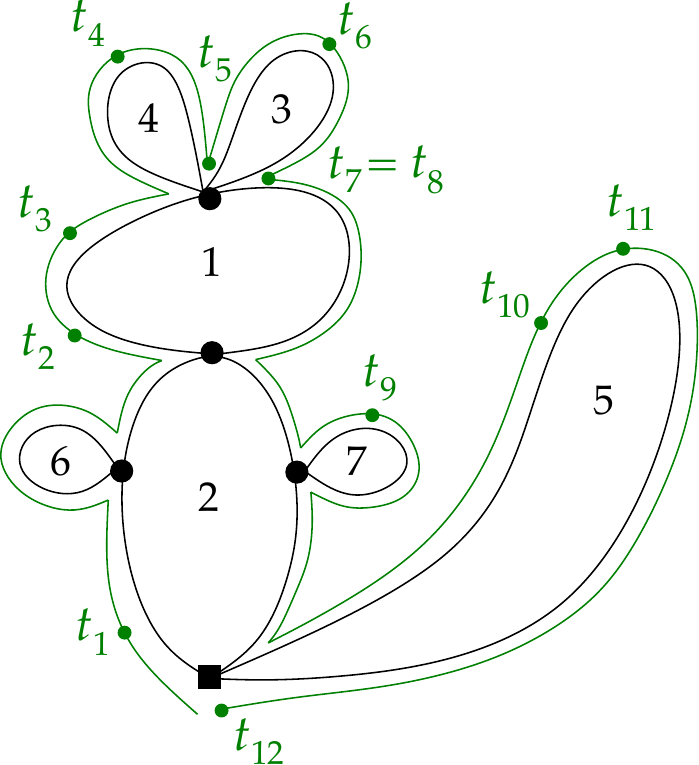}}
\qquad \qquad		
\begin{tabular}{l | l}
	$\ell$ & $S_\ell(C,\bt)$ \\
	\hline
	1 & \{1,2,3,4,8\} \\
	2 & \{0,1,8\} \\	
	3 & \{4,5,6,7,8\} \\
	4 & \{3,4\} \\
	5 & \{9,10,11,12\} \\
	6 & \{1\} \\
	7 & \{8,9\} 
\end{tabular}
	\caption{Above is an example of the calculation of $S_\ell(C,\bt)$ for the cactus $C\in \PCact(5)$ and $\bt \in \Delta^{12}$ pictured.  Recall that our convention is to put $t_i$ at local roots on the lobe ``as high up and far to the right as possible.''}
	\label{F:mickey}
\end{figure}

\subsubsection{The list of times $\bt^{C,\ell} \in \Delta^{|S_\ell(C, \bt)|-1}$}
Let $\bt^{C,\ell}$ be the list of $t_i$ shown in formula \eqref{Eq:t_i-on-lobe-ell} together with the coordinates of the local roots $s_1, \dots, s_w$, in order.
Explicitly, 
\[
\mathbf{t}^{C,\ell} := (t_{k_1+1}, \dots, t_{k_1+j_1}, \, s_1, \, 
t_{k_2+1}, \dots, t_{k_2+j_2}, \, s_2, \, \dots, \,
s_w, \, t_{k_{w+1}+1}, \dots, t_{k_{w+1}+j_{w+1}}).
\]
Thus the number of entries in $\bt^{C, \ell}$ is the cardinality of $S_\ell(C, \bt)$.

\subsubsection{The lobe parametrizations $\rho_\ell(C): \I \to \I$}
\label{S:lobe-param}
Keeping in mind the parametrization of the perimeter of $C$ by $\I$, we define $\rho_\ell(C): \I \to \I$ as the map that fixes the endpoints of $\I$, is constant on every lobe other than $\ell$, and is affine-linear with the same slope on the remaining subintervals of $\I$ (which correspond to the arcs along $\ell$).  In terms of the coendomorphism $g:S^1 \to (S^1)^{\x m}$ associated to $C$ as in Section \ref{S:partitions-and-coend}, the map $S^1 \to S^1$ induced by $\rho_\ell(C)$ is $\pi_\ell \circ g$.  
By abuse of notation, we also write $\rho_\ell(C)$ for the induced map $\Delta^n \to \Delta^n$ on non-decreasing $n$-tuples of points in $\I$.

\begin{figure}[h!]
\includegraphics[scale=0.6]{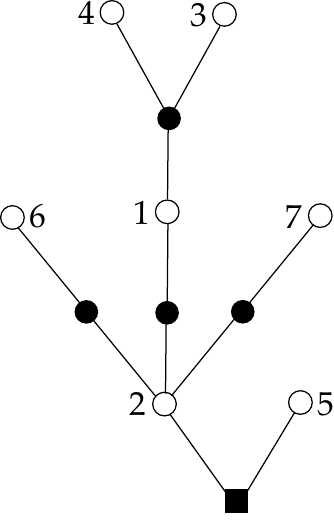} \qquad \qquad
\includegraphics[scale=0.6]{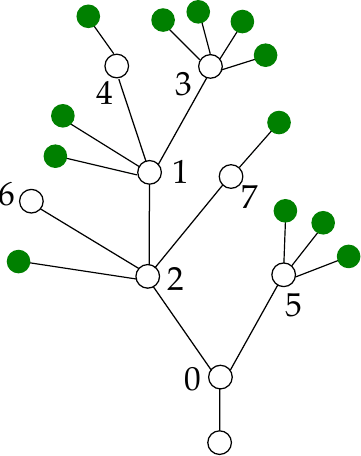}
\caption{On the left is the underlying b/w tree of the cactus $C$ from Figure \ref{F:mickey}, and on the right is the associated tree $T(C,\bt)$ in $\Upsilon$, where $\bt$ is the element of $\Delta^{12}$ shown in  Figure \ref{F:mickey}.}
\label{F:TCt}
\end{figure}

\subsubsection{The tree $T(C,\bt)$}
Define a tree $T(C,\bt) \in \mathrm{Ob}(\Upsilon)$, where $\Upsilon$ is the category from Definition \ref{D:upsilon}, by the following procedure.
Start with the b/w tree underlying the cactus $C$, as in Definition \ref{D:underlying-tree}.  Collapse every edge directed from a white vertex (upwards) to a black vertex so that the only remaining black vertex is the root.  Label each white vertex by the element of $\{1,\dots,m\}$ that labels the corresponding lobe.  Label the root by 0, and add a univalent vertex below it.  Finally, for each white vertex $\ell$, add an outgoing edge from $\ell$ to a new leaf for every $t_i$ lying on lobe $\ell$, in the order according to the appearance of $t_i$ with respect to local roots of higher lobes.  
See Figure \ref{F:TCt} for an example of $T(C,\bt)$ for a fixed $C$ and $\bt$.

The result $T(C,\bt)$ is a planted tree with internal vertices labeled by $0,\dots,m$ and $n$ leaves.  Therefore for any operad $\P$, one can associate to $T(C,\bt)$ the $(m+1)$-fold product of spaces $\prod_{\ell=0}^m\P(q_\ell)$, where $q_\ell$ is the number of outgoing edges from the vertex $\ell$ in $T(C,\bt)$.  Furthermore, the map $T(C,\bt)\to \gamma_n$ gives rise to a map $\bigcirc_{T(C,\bt)}:\prod_{i=0}^m\P(q_i) \to \P(n)$, as defined in formula \eqref{Eq:big-circ-T}.

\medskip

For any integer $p\geq1$, let $\mathbf{e}^p := (d^0)^p(\varnothing) \in \tC_p^{fr} \la \R^d \ra$, the configuration where the direction between any two points is one of $(\pm 1, 0, \dots, 0)$, and the frame at each point is $I_d$.  It can be obtained by the $p$-fold application of any coface maps to the empty configuration $\varnothing$.

For the global basepoint $b(C)$ of $C$, write $|b(C)|$ for the number of outgoing edges from the corresponding black root vertex.
 
\begin{definition}[Action of projective cacti on framed infinitesimal aligned maps]
\label{D:action}
For $m\geq 1$, $C\in \PCact(m)$, $\phi^1, \dots, \phi^m \in \tAM_n^{fr}$, and $\bt \in \Delta^n$, define
\begin{empheq}[box=\fbox]{align}
\label{Eq:Action}
\begin{split}
\alpha_m(C; \, \phi^1,\dots,\phi^m) (\bt)
:=& 
\bigcirc_{T} \left(
\mathbf{e}^{|b|}, 
\phi^1_{S_1}\left(\rho_1 \left(\bt^1\right)\right),\dots,
\phi^m_{S_m}\left(\rho_m \left(\bt^m\right)\right)
\right)\\
:=&
\bigcirc_{T(C, \bt)} \left(
\mathbf{e}^{|b(C)|}, 
\phi^1_{S_1(C,\bt)}\left(\rho_1(C) \left(\bt^{C,1}\right)\right),\dots,
\phi^m_{S_m(C,\bt)}\left(\rho_m(C) \left(\bt^{C,m}\right)\right)
\right).
\end{split}
\end{empheq}
In the first formula, we omit the dependence of various quantities on $C$ and $\bt$ for the sake of readibility, while in the second one, we include them for the sake of transparency.
\end{definition}

Applying the definition of $\bigcirc_{T(C,\bt)}$  to the right-hand side of formula \eqref{Eq:Action} yields an expression of the form
\begin{equation}
\label{Eq:ellth-term}
\cdots \left[ (\cdots ((\phi^\ell_{S_\ell}(\rho_\ell (\bt^\ell))\circ_{j_1+1} \cdots )\circ_{j_1+j_2+2} \cdots) \cdots) \circ_{j_1+\dots+j_{w}+w}  \cdots \right] \ \cdots
\end{equation}
where each ellipsis directly after a composition symbol $\circ_j$ stands for a configuration resulting from some of the remaining $\phi^1, \dots, \phi^m$ or an iterated composition of such configurations.  
More succinctly, we rewrite equation \eqref{Eq:Action} as
\begin{equation}
\label{Eq:ellth-term-xes}
\bx = \cdots \ \left[(\cdots ((\bx^\ell\circ_{j_1+1} \cdots )\circ_{j_1+j_2+2} \cdots) \cdots) \circ_{j_1+\dots +j_{w}+w}  \cdots \right] \ \cdots
\end{equation}
where each $\bx^\ell$ is a configuration of as many points as there are entries in $\bt^\ell$.  

When $m=2$, which is the setting of the multiplication and Browder bracket in homology studied in Section \ref{S:compatibility},  the formulas are more tractable.  
If both lobes of $C$ have their local root at the global basepoint $b$ and lobe 1 is to the left of lobe 2, then 
\begin{equation}
\label{Eq:arity-2-a}
\alpha_2(C; \phi^1, \phi^2)(\bt) = \Bigl(\mathbf{e}^2 \circ_2  \left(\phi^2_{S_2}(\rho_2 (\bt^2)\right) \Bigr) \circ_1 \left(\phi^1_{S_1}(\rho_1 (\bt^1)\right).
\end{equation}
If only lobe 1 has its local root at $b$, then  
\begin{equation}
\label{Eq:arity-2-a}
\alpha_2(C; \phi^1, \phi^2)(\bt) = 
\mathbf{e}^1 \circ_1 \Bigl(\phi^1_{S_1}(\rho_1 (\bt^1)) \circ_{j_1+1} \left(\phi^2_{S_2}(\rho_2 (\bt^2)\right) \Bigr).
\end{equation}
The other two possibilities for $m=2$ are accounted for by permuting the indices 1 and 2 on the right-hand sides above.

\begin{remark}
\label{R:conf-it}
The action $\alpha_2$ of 2-lobed cacti is related to the action $\alpha_2': \CC_1(1)^2 \x (\tAM_n^{fr})^2 \to \tAM_n^{fr}$ of pairs of intervals in previous joint work of the first author \cite[Definition 4.8]{BCKS:2017}.  
To explain the relationship, we use the spaces $\Ov^1(m)$ defined in Section \ref{S:Ov^1-defn} or rather just $\Ov^1(2)$, which is homeomorphic to a circle, as shown in Figure \ref{F:arity2}.  There is an inclusion $[\pi, 2\pi] \incl \CC_1(1)^2$, which first maps $[\pi, 2\pi]$ to $\Ov^1(2)$ as the bottom half of that circle.
Via the projection $\Ov^1(2) \to \Cact^1(2)$ illustrated in Figure \ref{F:arity2},
we can compare the actions induced by $\alpha_2$ and $\alpha_2'$ on the interval $[\pi, 2\pi]$.

These actions will differ by a homotopy on neighborhoods of the boundary points of the intervals $L_1, L_2 \in \CC_1(1)$. 
This is because $\alpha_2'$ involves the ``core'' $L_i^\circ$ of $L_i$ and a map $e_{L_i}$ that stretches $\I$ to fill in a gap resulting from shrinking $L_i$ to a point.  These elements appear because $\alpha_2'$ first maps infinitesimal configurations back to spatial ones, using \emph{infinitesimally thickened} configurations in $\R^3$ whose spatial coordinates all lie on the long axis.  In contrast, $\alpha_2$ stays purely within the infinitesimal realm, avoiding spatial and infinitesimally thickened configurations, which we find conceptually cleaner.  
The homotopy $H$ given by formula \eqref{Eq:shrinking-H} is similar to $e_{L_i}$, but it appears only in the map $q \circ ev_n : \K_d^{fr} \to \tAM_n^{fr}$ and the compatibility of the cactus and cubes actions (where we also use the cores $L_i^\circ$).
\end{remark}

\subsection{Proof of the action's required properties}

\begin{theorem}
\label{T:action}
Formula \eqref{Eq:Action} defines an action of the operad $\PCact$ on the space $\widetilde{AM}_n^{fr}$.  
\end{theorem}

\begin{proof}
We must check the following:
\begin{enumerate}
\item 
$\alpha_m$ is well defined, i.e., the map $\phi$ defined by \eqref{Eq:Action} is indeed an element of $\tAM_n^{fr}$, 
\item $\alpha_m$ is continuous, and
\item $\alpha_m$ satisfies the three conditions for an operad action.
\end{enumerate}


{\bf (1)  Proof that the action is well defined}:  We fix $C$ and $\phi^1, \dots, \phi^m$ and view $\phi$ as a map $\Delta^n\to \tC_n^{fr}\la\R^d\ra$.  We need to check that it satisfies the appropriate conditions on the boundary faces.  Since the condition for an intersection of faces is the logical conjunction of the conditions for those faces, it suffices to check the behavior on codimension-1 faces.  On the face $d^i\Delta^{n-1}$ where $t_i=t_{i+1}$, the condition is that $\phi(\bt)$ is the result of applying the coface $d^i$ to a configuration in $\tC_{n-1}^{fr}\la \R^d\ra$.  At each point $\bt$ in this face, $t_i$ lies in some lobe $\ell$ of $C$.  Now $\phi^\ell$ satisfies the required boundary condition on the corresponding face of $\Delta^{|S_\ell(C, \bt)|-1}$; suppose that this face is given by $t_j=t_{j+1}$.  The proof of this step is completed by noting that the insertion of a configuration in the image of $d^j$ is in the image of $d^i$ for the larger configuration.  Indeed, a coface map is a special case of the insertion maps, which are structure maps for the framed Kontsevich operad and in particular satisfy the operad associativity condition.  
Thus we have shown that $\phi\in \tAM_n^{fr}$. 


{\bf (2) Proof of the continuity of the action}:
It suffices to show that the adjoint 
\[
\widehat{\alpha}_m: \PCact(m) \x (\tAM_n^{fr})^m \x \Delta^n \longrightarrow \tC_n^{fr}\la\R^d\ra
\]
of $\alpha_m$ is continuous.
Let $T$ be a b/w tree with vertices labeled by $\{1,\dots, m\}$, and let $P$ be an ordered partition of $\{1,\dots, n\}$ into $m$ parts.
We partition the domain of $\widehat{\alpha}_m$ into subsets $\mathcal{S}(T,P)$ of elements $(C, \phi^1, \dots, \phi^m, \bt)$ such that $C$ has corresponding b/w tree $T$ and such that $C$ and $\bt=(t_1,\dots,t_n)$ determine the partition $P$, using the parametrization of the perimeter of $C$ by $\I$ and the labels of the lobes of $C$ by $1,\dots,m$.  On each $\mathcal{S}(T,P)$, the continuity of $\widehat{\alpha}_m$ follows from the continuity of each $\phi^\ell$ as a function of $\bt$, each $\rho_\ell$ as a function of $C$ and $\bt$, and each insertion operation $\circ_i$ as a function of the two input configurations.

The sets $\mathcal{S}(T,P)$ are neither open nor closed, so we consider their closures $\overline{\mathcal{S}(T,P)}$ and extend $\widehat{\alpha}_m$ to them by continuity.   We just need to check the agreement of these extensions on each intersection $\overline{\mathcal{S}(T,P)} \cap \overline{\mathcal{S}(T',P')}$.  
One of the two extensions is given by $\widehat{\alpha}_m$ itself, while the other is determined by a limit of its values.
An intersection point falls into one or both of the two cases below.
The main point of the proof of continuity at a boundary point will be the behavior of an aligned map in $\tAM_n^{fr}$ when input configuration points approach the boundary of $\I$.  The required behavior in turn comes from the property of the first and last coface maps in $\tC_\bullet^{fr}\la\R^d\ra$ stated in Remark \ref{R:AM-behavior-at-pm-1}.  
Thus for continuity, it is crucial that we are using the infinitesimal rather than spatial mapping space model.  
We provide a more detailed discussion below.

{\bf Case (a)}: Suppose some $t_i$ on a lobe $\ell$ approaches the local root of a lobe $k$ directly above $\ell$.  This corresponds to a change of the partition $P$.
In the limit, the corresponding output configuration point $x_i$ has entered the infinitesimal configuration $\bx^k$ that is 
plugged into $x^\ell_j$ for some $j$, but $x_i$ lies infinitely far from all the other points in $\bx^k$.
More precisely, with reference to the defining data of an element in $\tC_n^{fr}\la\R^d\ra$, all the points in $\bx^k$ have the same location as $x_i$, they share the same unit vector direction to $x_i$, and the frame at $x_i$ is the frame at $x^\ell_j$.
Because of the boundary behavior of infinitesimal aligned maps, 
the same is true of the configuration $\widehat{\alpha}_m(C, \phi^1, \dots, \phi^m,\bt)$ when $t_i$ is at that local root, and the values of all other direction vectors and frames are also their limits along this approach.  

More generally, at any point in an intersection $\overline{\mathcal{S}(T,P)} \cap \overline{\mathcal{S}(T,P')}$, some nonzero number of $t_i$ lie at local roots of the cactus $C$.  Similar arguments show that the value of $\widehat{\alpha}_m$ at such a point agrees with the limit as any number of those $t_i$ approach those local roots along the lobes directly below them.  Not much modification is needed even if multiple $t_i$ approach the same local root from the same direction. 
Indeed, relative rates of approach are not recorded in $\tC_n^{fr}\la\R^d\ra$, so collinear configurations of a given number of points in a given direction can differ only by the framings on the points, and the framings corresponding to points $t_i$ which have collided coincide under an aligned map.  This completes the proof of case (a).

{\bf Case (b)}: Suppose that the b/w tree $T$ changes.  There are two subcases according to the two types of angle collapses described in Definition \ref{D:cat-of-bw-trees}, i.e., according to whether $e_0$ is not or is among the two edges that are glued.  The first subcase corresponds to two conglomerations of points colliding, while the second corresponds to a conglomeration of points going to infinity relative to the other ones on the same lobe.   

For the subcase (i), suppose lobes $k$ and $k'$ approach each other on lobe $\ell$.  
In this case the formula for the action is unchanged.  Indeed, in constructing $T(C,\bt)$ from the b/w tree underlying $C$, both black vertices involved are collapsed to the same white vertex.
Thus continuity holds in this case.

For the subcase (ii), suppose the local root $r_k$ of a lobe $k$ approaches the local root $r_\ell$ of the lobe $\ell$ directly below it.
Along such an approach, $\bx^k$ remains infinitely far from points $x_i$ coming from those $t_i$ on lobe $\ell$.  That is, for any such point $x_i$, a single unit vector specifies the direction from it to any point in $\bx^k$.  The same is true of $\widehat{\alpha}_m(C;\, \phi^1, \dots, \phi^m)$ at the cactus $C$ where $r_k=r_\ell$, by the  properties of infinitesimal aligned maps on $\d \I$.  
In other words, there is no distinction in $\tC_n^{fr}\la\R^d\ra$ between the insertion of $\bx^k$ at infinity in $\bx^\ell$ and the insertion of $\bx^k$ and $\bx^\ell$ into an infinitesimal two-point configuration.  
Thus, the limit equals the value along such an approach, as desired.

In general, multiple angle collapses in $T$ may happen simultaneously, but similar considerations show that in these cases too, the limit of $\widehat{\alpha}_m$ as $C$ approaches such a point is the value of $\widehat{\alpha}_m$ at that point.  
Much like in case (a), little modification of the arguments is needed even if multiple local roots approach the same point.  
This completes the proof of case (b).

Approaches as in cases (a) and (b) may happen simultaneously.  
As before, the same type of argument covers the approach of multiple points or configurations to points at infinity in an infinitesimal configuration.
(For example, suppose two lobes $k$ and $k'$ collide along $\ell$ at a point where some $t_i$ lies, with $k$ approaching from the left and $k'$ from the right. Then the limit along such a path and the value at the endpoint are both the configuration where $\bx^k$, $x_i$, and $\bx^{k'}$ are located at the same point; the direction vector from any point in $\bx^k$ to $x_i$ is the same as that from $x_i$ to any point in $\bx^{k'}$; and the frame at $x_i$ is the frame of the points $x^\ell_j$ and $x^\ell_{j'}$ into which $\bx^k$ and $\bx^{k'}$ are inserted.)
This completes the proof of continuity.


{\bf (3) Proof that the action is an operad action}:  
We need to check the (a) identity, (b) equivariance, and (c) associativity conditions that define an operad action.

First, for the single class of $1$-lobe cactus $C \in \PCact(1)$,  we have $S_1(C;\bt)=[n]$, $\rho_1(C) = \mathrm{id}_{\I}$, and $\bt^1=\bt$, so $\alpha_1(C; \phi^1) = \phi^1$, thus verifying the identity condition (a) of an operad action.

Next, for equivariance, let $\sigma \in \mathfrak{S}_m$.  
Write 
\[
\alpha_m(C; \, \phi^1, \dots, \phi^m)(\bt) 
= \cdots \left[ (\cdots (\phi^\ell_{S_\ell}(\rho_\ell (\bt^\ell))\circ_{j_1(\ell)+1} \cdots ) \cdots) \circ_{j_1(\ell)+\dots + j_{w(\ell)}(\ell) +w(\ell)}  \cdots \right] \cdots
\]
It suffices to consider the action by $\sigma^{-1}$, which we do only for notational brevity below.
Indeed, we have 
\begin{align*}
&\alpha_m(C \sigma^{-1}; \, \phi^1, \dots, \phi^m)(\bt)= \\
  &\cdots \left[ (\cdots (\phi^{\sigma(\ell)}_{S_{\sigma(\ell)}}(\rho_{\sigma(\ell)} (\bt^{\sigma(\ell)}))\circ_{j_1(\sigma(\ell))+1} \cdots ) \cdots) \circ_{j_1(\sigma(\ell))+\dots+j_{w(\sigma(\ell))}(\sigma(\ell))+w(\sigma(\ell))}  \cdots \right] \cdots =\\
&\alpha_m(C; \, \phi^{\sigma(1)}, \dots, \phi^{\sigma(m)})(\bt) 
=\alpha_m(C; \, \sigma^{-1}(\phi^{1}, \dots, \phi^{m}))(\bt)  
\end{align*}
as required.  

Finally, for the associativity condition, suppose that $C \in \PCact(m)$, $C' \in \PCact(m')$, and $\ell \in \{1,\dots, m\}$.
We want to show that
\begin{align}
\label{Eq:AssocAction}
\begin{split}
&\alpha_{m+m'-1}(C \circ_\ell C' ; \, \phi^1, \dots, \phi^{m+m'-1})(\bt)
= \\
&\alpha_m(C; \, \phi^1, \dots, \phi^{\ell-1}, \alpha_{m'}(C'; \, \phi^\ell, \dots, \phi^{\ell+m'-1}), \phi^{\ell+m'}, \dots, \phi^{m+m'-1}).
\end{split}
\end{align}
Equivalently, we must verify that 
\begin{align*}
\begin{split}
&\bigcirc_{T(C\circ_\ell C', \bt)} \left(
\mathbf{e}^{|b|},
\phi^1_{S_1}\left(\rho_1 \left(\bt^1\right)\right),\dots,
\phi^{m+m'-1}_{S_{m+m'-1}}\left(\rho_{m+m'-1} \left(\bt^{m+m'-1}\right)\right)
\right)
= \\
&\bigcirc_{T(C, \bt)} \left(
\mathbf{e}^{|b|},
\phi^1_{S_1}\left(\rho_1 \left(\bt^1\right)\right),\dots,
\alpha_{m'}(C'; \, \phi^\ell, \dots, \phi^{\ell+m'-1}), \dots, 
\phi^{m+m'-1}_{S_{m+m'-1}}\left(\rho_{m+m'-1} \left(\bt^{m+m'-1}\right)\right)
\right).
\end{split}
\end{align*}
Consider the functor from the category $\Upsilon$ determined by the framed Kontsevich operad.
The left-hand side above is an evaluation of the map of spaces arising from the collapse $\chi$ of $T(C \circ_\ell C')$ to the corolla $\gamma_{m+m'-1}$.  The right-hand side above is an evaluation at the same point of the map of spaces arising from the composition of two morphisms: $\chi'$ which collapses the subtree corresponding to $C'$ to a corolla and $\chi''$ which collapses the resulting tree to $\gamma_{m+m'-1}$.
Since $\chi = \chi'' \circ \chi'$, the corresponding maps of spaces determined by the Kontsevich operad are equal too.
In other words, the key point is that the spaces $\tC_n^{fr} \la\R^d\ra$ form an operad.  
Thus associativity holds.
\end{proof}

\begin{proposition}
\label{P:proj-compatible}
The actions of $\PCact$ on the spaces $\tAM_n^{fr}$, $n\in \N$, are compatible with the projections 
$\tAM_n^{fr} \to \tAM_{n-1}^{fr}$.  That is, the following square commutes:
\[
\xymatrix{
\PCact(m) \x \left(\tAM_n^{fr}\right)^m \ar[r] \ar[d] & \tAM_n^{fr}  \ar[d] \\
\PCact(m) \x \left(\tAM_{n-1}^{fr}\right)^m \ar[r] & \tAM_{n-1}^{fr}
}
\]
\end{proposition}
\begin{proof}
The proof is similar to that of \cite[Proposition 5.2]{BCKS:2017}.
Let $\bu:=(t_1, \dots, t_{n-1})$ and $\bt:=(t_1,\dots,t_{n-1},1)$.  
Each composition is given by a formula of the form \eqref{Eq:Action}.  
The one through the upper-right corner has terms of the form 
$\phi^\ell_{S_\ell(C,\bt)} (\rho_\ell(\bt^\ell))$, 
while the one through the lower-left corner has terms of the form 
$(\phi^\ell|_{d^n \Delta^{n-1}})_{S_\ell(C,\bu)} (\rho_\ell(\bu^\ell))$, where we have suppressed the dependence of $\rho_\ell$, $\bt^\ell$, and $\bu^\ell$ on $C$.  
For all $\ell$ except the one corresponding to the rightmost lobe of $C$, we have 
$\phi^\ell_{S_\ell(C,\bt)} = (\phi^\ell|_{d^n \Delta^{n-1}})_{S_\ell(C,\bu)}$ and 
$\bt^\ell = \bu^\ell$.  Conceptually, this is because in both cases, the $t_i$ on lobe $\ell$ coincide, as do those to the left of $\ell$, those to the right of $\ell$, and those on any branch emanating from $\ell$.
If $\ell$ corresponds to the rightmost lobe of $C$, then $\bt^\ell=(\bu^\ell, 1)$, and 
$(\phi^\ell|_{d^n \Delta^{n-1}})_{S_\ell(C,\bu)}$ is the restriction of $\phi^\ell_{S_\ell(C,\bt)}$ to the face where the last coordinate is 1.
\end{proof}

The statements in Theorem \ref{T:A} about the existence of an action and its compatibility with the projections in the Taylor tower are now proven.  The remaining parts of Theorem \ref{T:A} are proven in Section \ref{S:compatibility}.

\section{The space of normalized overlapping intervals}
\label{S:normalized-overlapping}

Completing the proof of Theorem \ref{T:A} requires the equivalences stated in Theorem \ref{T:B}.
In this Section, we prove Theorem \ref{T:B}.  We construct a space $\Ov^1(m)$ whose inclusion into $\Ov(m)$ is a homotopy equivalence and which projects to $\Cact^1(m)$ by a homotopy equivalence.  In Section \ref{S:Ov^1-defn}, in particular in Definition \ref{D:normalized-overlapping} (and Remark \ref{R:alt-def-Ov1}), we define this space, whose elements we call normalized overlapping intervals.  In Section \ref{S:projection-onto-Cact^1}, we define the projection $\Ov^1(m) \to \Cact^1(m)$ and show that this map and the inclusion $\Ov^1(m) \incl \Ov(m)$ are homotopy equivalences.

\subsection{Definitions of normalized overlapping intervals}
\label{S:Ov^1-defn}

Recall the definitions of {\em above, below, left, and right} in a planted tree  from Section \ref{S:trees}, as well as the category $\T_m$ of b/w trees from Definition \ref{D:cat-of-bw-trees}.  

\begin{definition}
\label{D:T_sigma}
Let $\sigma \in \SS_m$ and $T \in \T_m$.  We say that $T$ and $\sigma$ are {\em compatible} if for every directed path in $T$ from the root to a leaf, the sequence of labels of the white vertices is a subsequence of $(\sigma(1), \dots, \sigma(m))$.  
Let $\T_\sigma$ be the full subcategory of $\T_m$ of trees $T$ compatible with $\sigma$.  For any $j\in \{1,\dots, m\}$, we call $\sigma^{-1}(j)$ the {\em $\sigma$-height} of $w_j$.
\end{definition}

In general, a permutation $\sigma$ is compatible with multiple trees $T$, and a tree $T$ may be compatible with multiple permutations $\sigma$.
If $T' \prec T$ is a morphism in $\T_m$, then the set of $\sigma$ compatible with $T$ are a subset of those compatible with $T'$.  Thus if a cactus $C$ belongs to closed cells $\mathcal{E}(T)$ for multiple trees $T$, its underlying tree (as in Definition \ref{D:underlying-tree}) is compatible with the largest set of permutations $\sigma$.

\begin{definition}
\label{D:lambda-and-rho}
Let $\sigma \in \mathfrak{S}_m$, and let $T \in \T_\sigma$.
 For a vertex $v$ in $T$, let $\alpha(v)$ be the number of white vertices above $v$ if $v$ is a black vertex and the number of white vertices above $v$ plus one if $v$ is a white vertex.
Suppose $i$ is a white vertex in $T$.
Let $\lambda(i)$ be the set of white vertices in $T$ to the left of $i$
and $\rho(i)$ the set of white vertices in $T$ to the right of $i$.
Define $\lambda^+(i)$ as the subset of white vertices $w$ in $\lambda(i)$ such that 
\begin{enumerate}
\item 
$\sigma^{-1}(w) > \sigma^{-1}(i)$ and 
\item
if $w'$ is to the right of $w$ and to the left of $i$, then $\sigma^{-1}(w') > \sigma^{-1}(i)$.
\end{enumerate}
Define  $\rho^+(i) \subset \rho(i)$ similarly, with ``right'' and ``left'' swapped but inequalities in the same directions.
\end{definition}

\begin{example}
\label{Ex:19tree}
We give a few examples of the numbers and sets defined in Definition \ref{D:lambda-and-rho} using the tree $T$ in Figure \ref{F:19tree}, where $m=19$.  We first list them for the white vertices $3$, $4$, $6$, and $11$:
\begin{align*}
i&=3:
& \alpha(3)&=1,
& \lambda(3)&=\{2,1,4\}, & \lambda^+(3)&=\{4\}, & \rho(3)&=\varnothing, & 
\rho^+(3)&=\varnothing.  \\
i&=4:
& \alpha(4)&=4,
& \lambda(4)&=\{2,1\}, & \lambda^+(4)&=\varnothing, & \rho(4)&=\{3\}, & 
\rho^+(4)&=\varnothing.  \\
i&=6:
& \alpha(6)&=6,
& \lambda(6)&=\{10,9\}, & \lambda^+(6)&=\{10,9\}, & \rho(6)&=\{11,13,12\}, & 
\rho^+(6)&=\{11,13,12\}.  \\
i&=11:
& \alpha(11)&=1,
& \lambda(11)&=\{10,9,6\}, & \lambda^+(11)&=\varnothing, & \rho(11)&=\{13,12\}, & 
\rho^+(11)&=\{13,12\}.  
\end{align*}
Considering $\alpha$ for some black vertices, we have $\alpha(v_r)=19$, $\alpha(v_1)=11$, $\alpha(v_2)=1$, $\alpha(v_3)=3$, and $\alpha(v_6)=3$.
\end{example}

\begin{figure}[!h]
		\centering
		\includegraphics[scale=0.58]{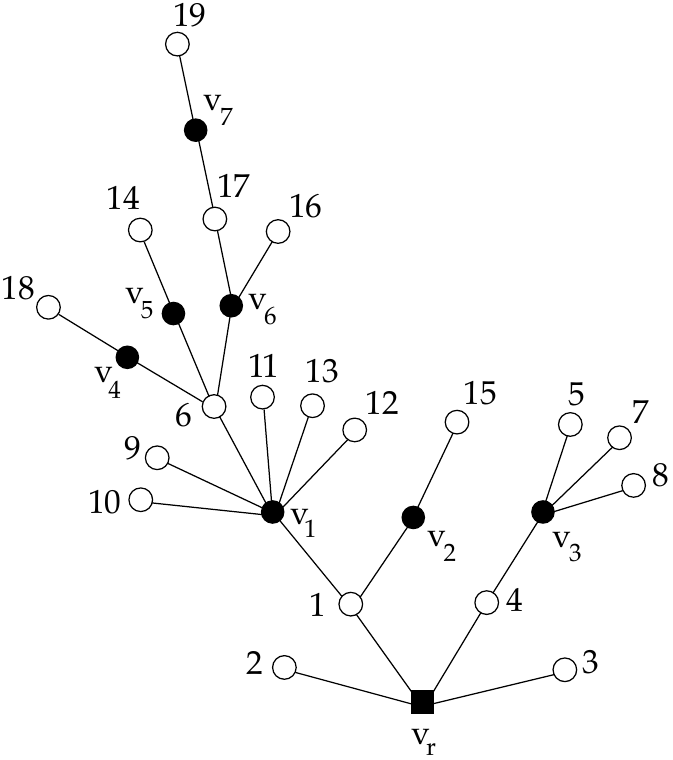}
		\caption{An element $T$ in $\mathcal{T}_{\sigma}$ where $\sigma=\mathrm{id} \in \SS_{19}$. The black vertices are labelled only for easy reference in Example \ref{Ex:19tree}.
		}
		\label{F:19tree}
\end{figure}

We will now define the subspace $\Ov^1(m) \subset \Ov(m)$ of normalized overlapping intervals.  Roughly, the key property of $\Ov^1(m)$ is that for an interval $L_i$ to start sliding past an interval $L_j$, $L_j$ must first grow so that every point in $L_i$ lies above a point in $L_j$.  More intuition may be gained from Figure \ref{F:19-ovin} and a shorter definition in Remark \ref{R:alt-def-Ov1}, but we will rely on the details in Definition \ref{D:normalized-overlapping} to verify that $\Ov^1(m) \simeq \Ov(m)$.

\begin{definition}[Normalized overlapping intervals]
\label{D:normalized-overlapping}
Let $\sigma \in \mathfrak{S}_m$, and let $T\in \T_\sigma$.
Define $\Ov^1(m,\sigma,T)$ as the set of equivalence classes of elements 
$(\bL, \sigma, \bK) = \bigl( ([x_i,y_i])_{i=1}^m, \, \sigma, \, ([x_v,y_v])_{v\in B(T)} \bigr)$ 
such that the endpoints $x_i$, $y_i$, $x_v$, and $y_v$ satisfy conditions defined below inductively in terms of the tree structure of $T$, starting at the root and going upwards:  
\begin{itemize}
\item
If $v$ is the root, we require that $[x_v, y_v]=\I$.
\end{itemize}
\begin{itemize}
\item
Suppose $i$ is a white vertex directly above a black vertex $v$ such that $[x_v, y_v]$ has been chosen.  Set 
\begin{align*}
\begin{split}
a_i&:= x_v + \frac{|\I|}{m} \sum_{j \in \lambda(i) \setminus \lambda^+(i)} \alpha(j), \\
c_i&:= y_v - \frac{|\I|}{m} \sum_{j \in \rho(i)} \alpha(j), \qquad \qquad \text{ and } 
\end{split}
\begin{split}
b_i&:= x_v + \frac{|\I|}{m} \sum_{j \in \lambda(i)} \alpha(j), \\
d_i&:= y_v - \frac{|\I|}{m} \sum_{j \in \rho(i) \setminus \rho^+(i)} \alpha(j), 
\end{split}
\end{align*}
where $\lambda(i)$, $\rho(i)$, $\lambda^+(i)$, and $\rho^+(i)$ are as in Definition \ref{D:lambda-and-rho}.
The conditions on $x_i$ and $y_i$  are that 
\begin{align*}
a_i \leq x_i \leq b_i 
&&\text{and} && 
c_i \leq y_i \leq d_i.
\end{align*}

\item 
Suppose $v_1, \dots, v_n$ are the black vertices directly above a white vertex $i$, listed in the order determined by the planar embedding of $T$ in the upper-half plane.  
Let $v$ be the black vertex directly below $i$, and suppose that $[x_v, y_v]$ has already been chosen.
The conditions on the intervals $[x_{v_j}, y_{v_j}]$ with $j=1,\dots,n$ are that 
\begin{enumerate}
\item each $[x_{v_j}, y_{v_j}]$ has length $\alpha(v_j) {|\I|} / {m}$, 
\item each $[x_{v_j}, y_{v_j}]$ is contained in $[b_i, c_i]$, and 
\item for each $j\in \{1,\dots, n-1\}$, we have $y_{v_j}\leq x_{v_{j+1}}$, i.e., $[x_{v_j},y_{v_j}]$ is to the left of $[x_{v_{j+1}},y_{v_{j+1}}]$ and these two intervals intersect in at most at one point.
\end{enumerate}
Here $b_i$ and $c_i$ are determined by $x_v$ and $y_v$ as above.
\end{itemize}
Finally, define the equivalence relation on these elements by setting 
\[
(\bL, \sigma, \bK) \sim (\bL', \sigma', \bK')
\]
if $(\bL, \sigma)$ and $(\bL', \sigma')$ represent the same element of $\Ov(m)$.  
\end{definition}

For example, Figure \ref{F:19-ovin} shows an element of $\Ov^1(19,\mathrm{id}, T)$ for the tree $T$ shown in Figure \ref{F:19tree}.
By the equivalence relation in Definition \ref{D:normalized-overlapping}, the map $\Ov^1(m,\sigma,T) \to \Ov(m)$ given by $[\bL, \sigma, \bK] \mapsto [\bL, \sigma]$ is injective, so we view $\Ov^1(m,\sigma,T)$ as a subspace of $\Ov(m)$.
We next define
\[
\Ov^1(m):= 
\bigcup_{\sigma\in \mathfrak{S}_m} \ 
\bigcup_{T\in \T_\sigma} 
\Ov^1(m,\sigma,T)
=
\bigcup_{T\in \T_m} \ 
\bigcup_{\sigma: T \in \T_\sigma} 
\Ov^1(m,\sigma,T).
\]
For example, $\Ov^1(2)$ is illustrated in Figure \ref{F:arity2}.
This union is not disjoint.  
Indeed, if $T' \prec T$, 
then $\Ov^1(m,\sigma,T) \cap \Ov^1(m,\sigma,T')\neq \varnothing$, 
much like the cells $\mathcal{E}(T)$ and $\mathcal{E}(T')$ in $\Cact^1(m)$ are not disjoint.  
On the other hand, those cells satisfy the condition $\mathcal{E}(T') \subset \mathcal{E}(T)$, whereas $\Ov^1(m,\sigma,T') \not\subset \Ov^1(m,\sigma,T)$.  For example, for $m=2$, two such spaces are edges which intersect in a vertex.
Similarly, if $\sigma$ and $\sigma'$ are both compatible with $T$, then $\Ov^1(m,\sigma,T) \cap \Ov^1(m,\sigma',T)$ may be nonempty.

\begin{figure}[!h]
	\includegraphics[scale=0.6]{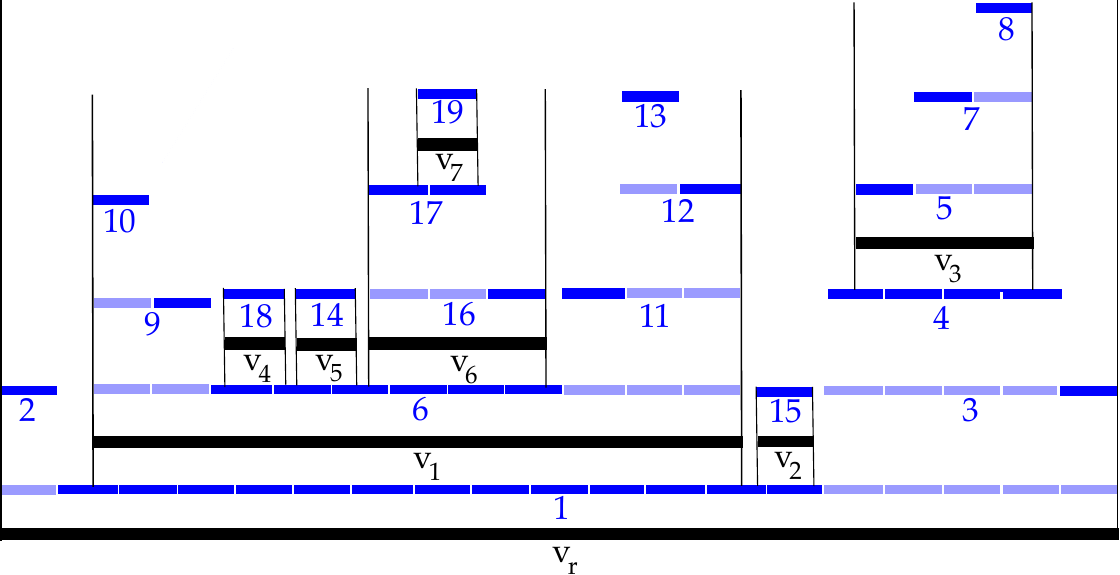}
	\caption{This picture indicates all possible elements of the form $[\bL,\mathrm{id}, \bK] = \left[([x_i,y_i])_{i=1}^m, \, \mathrm{id}, \, ([x_v,y_v])_{v\in B(T)}\right]$ in $\Ov^1(19, \mathrm{id}, T)$ for the tree $T$ in Figure \ref{F:19tree} and a fixed $\bK$.  
	The interval $[x_{v_j},y_{v_j}]$ for each black vertex $v_j$ is colored in black.  
	The interval $[b_i,c_i]$ for each white vertex $w_i$ is colored in dark blue.  
	The intervals $[a_i, b_i]$ and $[c_i, d_i]$ are colored in light blue.  
	The blue intervals are segmented only so that their lengths can be easily read: each segment has length $|\I|/19$. 
	Thus $[x_i,y_i]$ can be any interval containing the dark blue interval $[b_i,c_i]$ and with endpoints in the adjacent light blue intervals. 
	To vary $\bK$, we allow any $[x_{v_j},y_{v_j}]$ to slide within $[b_i,c_i]$ (but not on the light blue part of $[x_i,y_i]$), where $w_i$ is the white vertex directly below $v_j$.  All intervals above $[x_{v_j}, y_{v_j}]$ would move with it. 
	}
	\label{F:19-ovin}
 \end{figure}

The symmetric sequence $\{\Ov^1(m)\}_{m\in \N}$  does {\em not} form an operad.   For example, 
\[
\bigl( \bigl([-1,1], \, [-{1}/{2}, {1}/{2}]\bigr), \, \mathrm{id} \bigr) \circ_1 
\bigl( \bigl([-1,0], \, [0,1]\bigr), \, \mathrm{id} \bigr) \notin \Ov^1(3)
\]
even though both elements on the left-hand side are in $\Ov^1(2)$.

In a representative $(\bL, \sigma, \bK) \in \Ov^1(m, \sigma, T)$, the intervals $[x_v,y_v]$ that constitute $\bK$ are completely determined by the tree $T$ and the intervals $[x_i, y_i]$ that constitute $\bL$.
Indeed, from the definitions of $a_i$, $b_i$, $c_i$, and $d_i$, $x_v$ and $y_v$ are respectively the leftmost $x_i$ and rightmost $y_i$ over all white vertices $i$ directly above $v$.
Some elements of $\Ov^1(m)$ can be represented by multiple trees $T$, and for different choices of $T$, we get different numbers of intervals $[x_v,y_v]$.
We will see in Definition \ref{D:proj-to-cacti} that this parallels the situation of a cactus lying in the closures of multiple cells $\mathcal{E}(T)$ of various dimensions.

With a slight overloading of notation, we define the following two spaces:
\begin{align*}
\Ov^1(m,\sigma) := \bigcup_{T\in \T_\sigma} \Ov^1(m,\sigma,T); \qquad
\Ov^1(m,T) := \bigcup_{\sigma : T \in \T_\sigma} \Ov^1(m,\sigma,T).
\end{align*}
The space $\Ov^1(m)$ is the union of the former over all $\sigma \in \SS_m$ or the union of the latter over all $T \in \T_m$.

\begin{remark}[Alternative definition of $\Ov^1(m)$]
\label{R:alt-def-Ov1}
Here is a way to view $\Ov^1(m)$, without reference to trees.
An element $[(L_1,\dots,L_m), \sigma] \in \Ov(m)$  lies in $\Ov^1(m)$ if and only if for each $i\in \{1,\dots,m\}$, the set of points in $L_i$ for which no $L_j$ lies above $L_i$ has total length $|\I|/m$.  Informally, the defining condition is that ``each little interval receives an equal amount of sunlight.''  To prove that it is a necessary condition for being in $\Ov^1(m)$, one can first check that it holds if each $L_i=[x_i,y_i]$ is $[b_i,c_i]$, whose width is $ \alpha(i) |\I|/m$; then note that if some $L_i$ is any wider, the part outside $[b_i,c_i]$ must lie under another $L_j$.  
We will not need the sufficiency of this condition, but it can be proven as follows.  If $[\bL, \sigma]\in \Ov(m)$ satisfies it,  first construct an associated tree $T$ (and hence the black intervals $[x_v,y_v]$ and the points $a_i$, $b_i$, $c_i$, and $d_i$), from the leaves downward, by roughly considering the lowest $L_j$ directly below each $L_i$; then one can check that since $L_i$ is the highest interval for length $|\I|/m$, it must be at least as wide as $[b_i,c_i]$ and no wider than $[a_i,d_i]$.
\end{remark}

 \begin{figure}[!h]
	\includegraphics[scale=0.5]{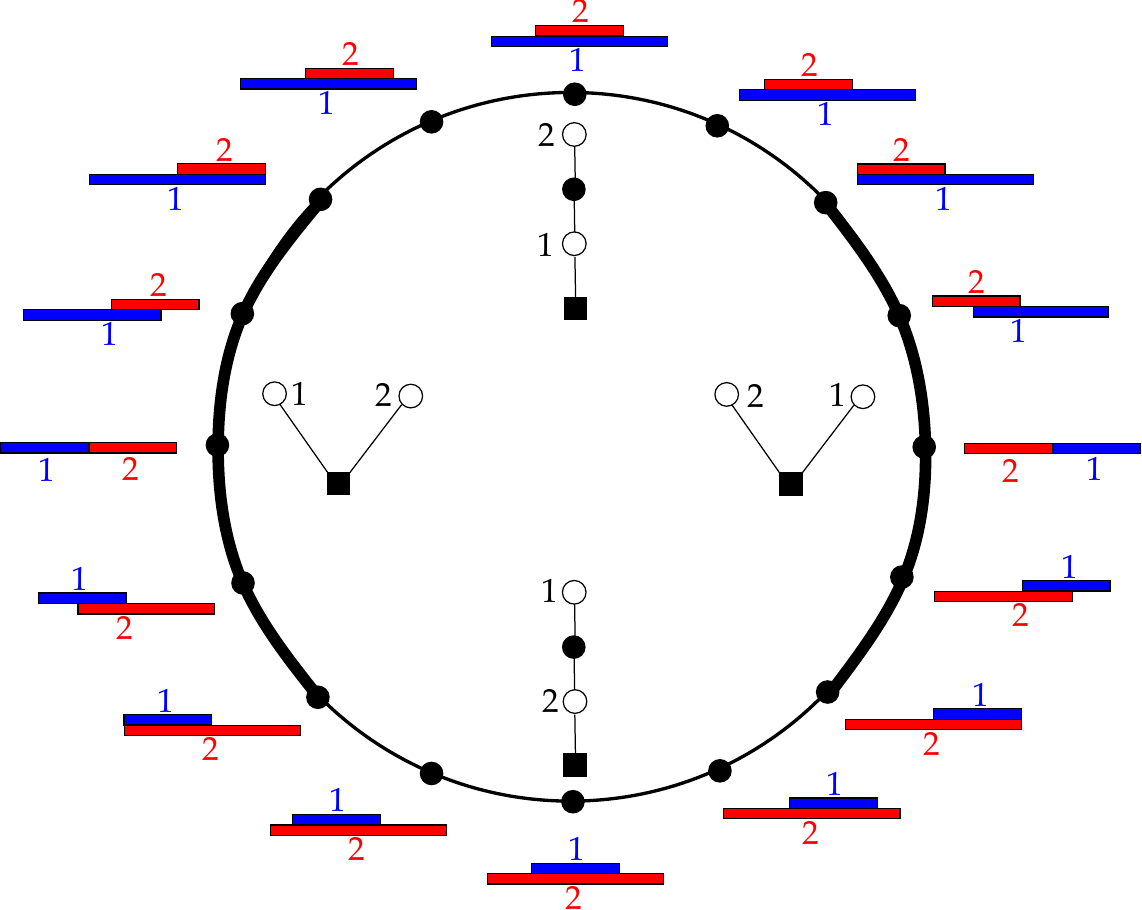} \qquad
	\raisebox{2pc}{\includegraphics[scale=0.5]{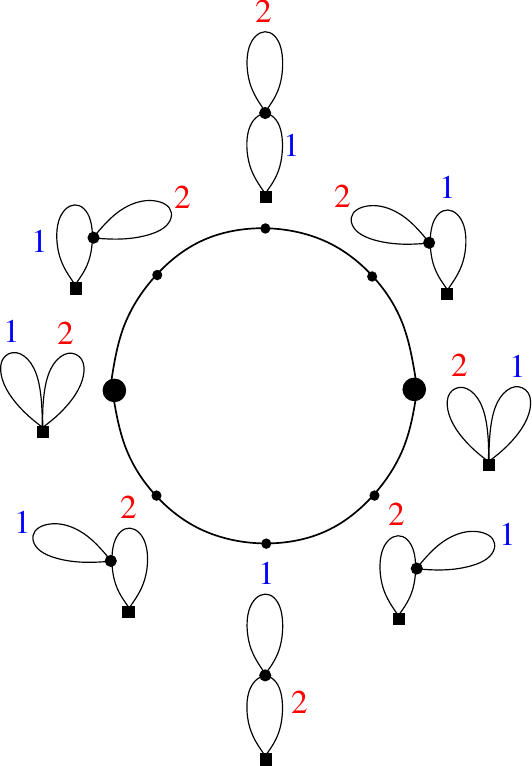}}
	\caption{On the left is the space $\Ov^1(2)$, the corresponding four trees $T$ in $\mathcal{T}_2$, and a sampling of sixteen normalized overlapping intervals elements $[(L_1, L_2),\sigma]$ in $\Ov^1(2)$.  
	The subspaces $\Ov^1(2,\sigma,T)$ are the six closed arcs between the points at angles $0$, $\pi/4$, $3\pi/4$, $\pi$, $5\pi/4$, and $7\pi/4$.
	On the right is the space $\Cact^1(2)$.  The map $\Ov^1(2) \to \Cact^1(2)$ collapses each thick arc in the left-hand circle to one of the 2 large points in the right-hand circle.
	}
	\label{F:arity2}
\end{figure}

\subsection{Maps from the space of normalized overlapping intervals}
\label{S:projection-onto-Cact^1}

We will first define a map $p_m: \Ov^1(m) \to \Cact^1(m)$.  We will then show that it is a homotopy equivalence by showing that it is a cellular map with contractible preimages.  We will finally show that the inclusion $\Ov^1(m) \to \Ov(m)$ is also a homotopy equivalence.

Recall from Definition \ref{D:cell-functor} and Proposition \ref{P:cell-structure} that $\Cact^1(m)$ is a union of cells $\mathcal{E}_m(T):=\Delta^{|w_1|} \x \dots \x \Delta^{|w_m|}$ where $T$ is a b/w tree with white vertices $w_1, \dots, w_m$.

\begin{definition}
\label{D:proj-to-cacti}
For each $m\geq 1$ and each $T \in \T_m$, we define a map $p_m^T: \Ov^1(m,T) \to \Cact^1(m)$.   
Let $[\bL, \sigma, \bK] \in \Ov^1(m,T)$, let $i\in\{1,\dots,m\}$, and suppose that $v_1, \dots, v_n$ are the black vertices directly above $w_i$ in $T$.
Let $t_0, \dots, t_n$ be the gaps between the intervals $[x_{v_j}, y_{v_j}]$; that is, $t_0:=x_{v_1}-b_i$, $t_j:=x_{v_{j+1}}-y_{v_j}$ for each $j\in \{1,\dots, n-1\}$, and $t_n:= c_i - y_{v_n}$.
Then define $p_m^T[\bL, \sigma, \bK]$ as the point in the cell $\mathcal{E}_m(T)$ whose coordinates in the factor $\Delta^{|w_i|}$ are $(mt_0/|\I|, \dots, mt_n/|\I|)$.
\end{definition}

Notice that $p_m^T[\bL, \sigma, \bK]$ is determined completely by $T$ and $\bK$.  Conversely, given $T\in \T_m$ and  $C\in \mathcal{E}_m(T)$, there is a unique list $\bK$ of intervals $[x_v,y_v]$ corresponding to the black vertices of $T$ such that $p_m^T[\bL, \sigma, \bK]=C$; this is because $T$ determines the bounds $b_i$ and $c_i$ in Definition \ref{D:normalized-overlapping}.

One can prove that the maps $p_m^T$ for all $T \in \T_m$ glue together to a well defined function $p_m: \Ov^1(m) \to \Cact^1(m)$ by checking that on $\Ov^1(m,T) \cap \Ov^1(m,T')$, we have $p_m^T=p_m^{T'}$.
We instead reinterpret the definition of $p_m^T$, following Remark \ref{R:alt-def-Ov1}.
The cactus $p_m^T[\bL, \sigma, \bK] \in \Cact^1(m)$ can be described by using the clockwise parametrization of its perimeter and specifying a lobe at each time $t\in [0,m]$. 
First apply the orientation-preserving affine-linear homeomorphism $\I \to [0,m]$.
Then $p_m^T[\bL, \sigma, \bK]$ is the cactus where the lobe $w_i$ at time $t$ corresponds to the interval $L_i$ with the highest $\sigma$-height at $t$.  

It is perhaps even clearer to use the space of partitions of $S^1$, as in the work of Salvatore \cite{Salvatore:2009}.
In that space, $p_m^T[\bL, \sigma, \bK]$ is the partition of 
$S^1 
\cong \I / \d \I$ where a time $t$ lies in the $1$-manifold labeled by the index of the interval with the highest $\sigma$-height at $t$ (see Definition \ref{D:T_sigma}).
From this description, we see that $p_m^T$ is independent of the b/w tree $T$ used in choosing a representative of a domain element.  
In other words, $p_m^T[\bL, \sigma, \bK]$ is completely determined by $\bL$ and $\sigma$.
Thus we obtain a continuous surjective function 
\[
p_m: \Ov^1(m) \to \Cact^1(m).
\]  

\begin{theorem}
\label{T:proj-is-cellular}
The space $\Ov^1(m)$ admits the structure of a finite regular CW complex such that $p_m$ is a cellular map.
\end{theorem}

\begin{proof}

Recall that the structure of a regular CW complex on $\Cact^1(m)$ was described in Definition \ref{D:cell-functor} and Proposition \ref{P:cell-structure}.
For any $C\in \Cact^1(m)$, we will essentially obtain CW structures on the preimages $p_m^{-1}\{C\}$ which fit together to give one on $\Ov^1(m)$.  We do this by restricting $p_m$ first to subspaces $\Ov^1(m,\sigma,T)$ and then to subspaces $\Ov^1(m,T)$.  Figure \ref{F:cells-m=3} gives some indication of the sub-preimages that we will define.

Let $Q:=\I^{2m}$.  First give it the structure of a cubical complex by subdividing each factor of $\I$ into $m$ subintervals of equal length.  Then give $Q$ the structure of a simplicial complex by subdividing each resulting cube (which has side length $|\I|/m$) into $(2m)!$ simplices of dimension $2m$.  If $t_1, \dots, t_{2m}$ are coordinates on $Q$, this subdivision is obtained via all the hyperplanes  $t_j = t_i + k |\I|/m$ where $k$ is an integer, as well as all intersections of them.  

For any $C$ in $\Cact^1(m)$, 
define $F_{\sigma, T}(C):=p_m^{-1}\{C\} \cap \Ov^1(m,\sigma,T)$, and let $[\bL, \sigma, \bK] \in F_{\sigma,T}(C)$. 
Here $\sigma$ is fixed, and as noted just after Definition \ref{D:proj-to-cacti}, 
$\bK=([x_v,y_v])_{v\in B(T)}$
 is completely determined by $C$.
So $F_{\sigma,T}(C)$ can be identified with the set of all possible values of $\bL=([x_i,y_i])_{i=1}^m$.
In $Q$, we identify for each $i \in \{1,\dots,m\}$ the coordinate $t_{2i-1}$ with $x_i-x_v$ and the coordinate $t_{2i}$ with $y_i-x_v$, where $v$ is the black vertex directly below the white vertex $i$.  
From Definition \ref{D:normalized-overlapping}, we see that $F_{\sigma,T}(C)$ is a union of closed cubical cells in $Q$.

Next, with $T$ still fixed, let $F_T(C):=\bigcup_{\sigma} F_{\sigma, T}(C)$ where the union is over those $\sigma\in \SS_m$ compatible with $T$.  
Consider the intersection $\bigcap_{\sigma \in S} F_{\sigma, T}(C)$ where $S$ is any subset of $\SS_m$.  
In any single $F_{\sigma,T}(C)$, this intersection is 
a simplicial subcomplex determined by 
weak inequalities of the 
form
$y_i \leq x_j$.  
Thus $F_T(C)$ has the structure of a simplicial complex.

\begin{figure}[h!]
\includegraphics[scale=0.8]{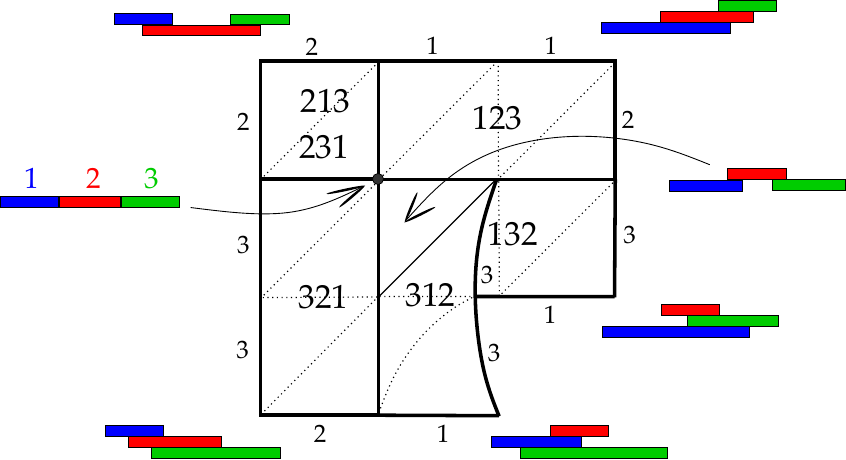}
\caption{Above is the space $p_3^{-1}(C)$ for the corolla cactus $C$ in $\Cact^1(3)$ such that the lobes are labeled 1, 2, and 3 from left to right.  It is also $\Ov^1(3,T)$ for the b/w tree $T$ with just a root and white vertices labeled $1$, $2$, and $3$ from left to right.
The thick solid lines delineate the pieces $F_{\sigma, T}(C)$ for the six possible $\sigma$.  
Each $F_{\sigma, T}(C)$ is labeled by $\sigma(1) \sigma(2) \sigma(3)$.  The dotted lines show the simplices from subdividing the cube $Q$.  A label $i$ on an edge indicates that interval $i$ varies along this edge, with the upward and rightward directions corresponding to the rightward motion of an endpoint.
The upper-left square is equal to $F_{\sigma, T}(C)$ for two different $\sigma$.  
The triangle with a thin solid edge is the intersection of spaces $F_{\sigma, T}(C)$ for two different $\sigma$.  
Any other $F_{\sigma, T}(C) \cap F_{\sigma', T}(C)$ is a vertex or an edge.  
Each picture of intervals is an interior point of $F_{\sigma, T}(C)$, except for the one for the vertex common to all of them.  
For any other tree $T' \in \T_3$, we have $T\prec T'$, and $F_{\sigma, T'}(C)$ is a subcomplex of $F_{\sigma, T}(C)$.
}
\label{F:cells-m=3}
\end{figure}

We now allow the cactus $C$ to vary, still fixing $T$.  
Each $T$ determines a closed cell in $\Cact^1(m)$ of the form 
$\mathcal{E}(T)=\Delta^{|w_1|} \x \dots \x \Delta^{|w_m|}$.  Suppose $C,C' \in \mathcal{E}(T)$.
Let $\bK=([x_v, y_v])_{v\in B(T)}$ and $\bK'=([x_v', y_v'])_{v\in B(T)}$ be the lists of intervals determined by $C$ and $C'$ respectively. 
There is a canonical homeomorphism $F_T(C) \to F_T(C')$ induced by the $|B(T)|$ translations of $\R$ which send $x_v$ to $x'_v$ and thus $y_v$ to $y'_v$.    
Thus $p_m^{-1}(\mathcal{E}(T)) = \Ov^1(m,T)$ is homeomorphic to $\mathcal{E}(T) \x F_T(C)$ for any $C\in \mathcal{E}(T)$.  We give $p_m^{-1}(\mathcal{E}(T))$ the induced product cell structure.  For any morphism $T'\prec T$, the space $(p_m^T)^{-1}(\mathcal{E}(T'))=\Ov^1(m,T) \cap \Ov^1(m,T')$ is a subcomplex of $(p_m^T)^{-1}(\mathcal{E}(T)) = \Ov^1(m,T)$ because $\mathcal{E}(T')$ is a face of $\mathcal{E}(T)$.  


Finally, we just need to show that the cellular structures on the subspaces $p_m^{-1}(\mathcal{E}(T))$ yield one on all of $\Ov^1(m)$.  It suffices to consider an angle collapse $T' \angle T$ where black vertices $v$ and $w$ in $T$ are identified.  We need only consider the case where one of these vertices lies directly to the left of the other and the case where one lies directly above the other.  In each case, one can check that $F_{\sigma, T}(C)$ is a cubical subcomplex of $F_{\sigma, T'}(C)$ where certain $[x_i,y_i]$ lie in $[x_v,y_v]$ and/or certain $[x_i,y_i]$ lie in $[x_w,y_w]$.  Thus $F_T(C)$ is a subcomplex of $F_{T'}(C)$, and hence $(p_m^T)^{-1}(\mathcal{E}(T'))=\Ov^1(m,T) \cap \Ov^1(m,T')$ is a subcomplex of $(p_m^{T'})^{-1}(\mathcal{E}(T'))=\Ov^1(m,T')$.


Putting this all together, we get the structure of a CW complex on $\Ov^1(m)=\bigcup_{T \in \T_m} \Ov^1(m,T)$ as follows.  Start with those $\Ov^1(m,T)$ for which $\mathcal{E}(T)$ is a $0$-cell.  Then attach those $\Ov^1(m,T')$ for which $\mathcal{E}(T')$ is a $1$-cell, one by one in any order.  
Continue in this way until all $T\in \T_m$ have been used.  Indeed, we have shown in the paragraphs above that $\Ov^1(m,T) \cap \Ov^1(m,T')$ is a subcomplex of both $\Ov^1(m,T)$ and $\Ov^1(m,T')$, so at each step we attach one cell complex to another along a subcomplex of each of them.  
For example, Figure \ref{F:cells-m=3} shows $\Ov^1(3,T)$ for the corolla $T$, i.e., the tree $T$ such that $\mathcal{E}(T)$ is 0-dimensional.  
This CW structure is regular, as a union of products of finitely many simplices with a finite simplicial complex.

The projection $p_m:\Ov^1(m) \to \Cact^1(m)$ is cellular because over each cell $\mathcal{E}(T)$, it is a projection onto one of two factors such that the cellular structure on $\Ov^1(m)$ is the product of the cellular structures on those factors.
\end{proof}

\begin{lemma}
\label{L:fibers-contractible}
The preimage of any point under $p_m$ is contractible.
\end{lemma}

\begin{proof}
Let $C\in \Cact^1(m)$, and let $T$ be the b/w tree associated to $C$.
The space $p_m^{-1}(C)$ ($= F_T(C)$) can be decomposed as a product of spaces indexed by the black vertices in $T$, where the space indexed by $v$ is the space of possible intervals $[x_i,y_i]$ for all white vertices $i$ directly above $v$.  Thus we may assume that $T$ is a corolla $\gamma$ whose only vertices are its root and $m$ white vertices.  We may also assume that the white vertices in $T$ are labeled $1,\dots, m$ from left to right, since a permutation $\sigma\in \SS_m$ of these labels induces a homeomorphism $p_m^{-1}(\sigma \cdot C) \cong p_m^{-1}(C)$.  

Let $F(m):=p_m^{-1}(\star_m)$, where $\star_m$ is the element of $\Cact^1(m)$ with underlying tree a corolla with white leaves labeled by $1,\dots, m$ from left to right.  
By Theorem \ref{T:proj-is-cellular}, $F(m)$ has the structure of a cell complex.
We will prove that it is contractible by induction on $m$.  One can take either $m=0$ or $m=1$ as the basis case, since each of $F(0)$ and $F(1)$ is a point.

Suppose that $F(n)$ is contractible for all $n<m$.  Let $(i_1, \dots, i_k)$ be a sublist of $(1,\dots, m)$.  Define $F_{i_1,\dots,i_k}$ as the subspace of $F(m)$ where all the intervals $L_{i_1}, \dots, L_{i_k}$ are at lowest height.  
For brevity, we will also write $I:=(i_1,\dots,i_k)$ and $F_I:=F_{i_1,\dots,i_k}$.  
Any $F_I$ is a subcomplex of $F(m)$.
Indeed, $F_m = \bigcup_{\sigma \in \SS_m} F_{\sigma, \gamma}(C)$ 
and $F_I$ is the union of subspaces of $F_{\sigma, \gamma}(C)$ determined by inequalities of the form $y_i \leq x_j$ where $i$ and $j$ are indices which either are in $I$ or have lower $\sigma$-height than some index in $I$.
Notice also that $F(m) = F_1 \cup \dots \cup F_m$.  We will now show by induction on $q$ that a $q$-fold union of $F_I$ is contractible for any $q\geq 1$, thus completing the proof.

First, consider a single $F_I$.
Let $n_0:=i_1-1$, let $n_j:=i_{j+1} - i_j - 1$ for $j=1,\dots,k-1$, and let $n_k:=m-i_k$.  
Thus $n_0$ is the number of $L_i$ to the left of $L_{i_1}$; for $j \in \{1,\dots,n-1\}$, $n_j$ is the number of $L_i$ strictly between $L_{i_j}$ and $L_{i_{j+1}}$; and $n_k$ is the number of $L_i$ to the right of $L_{i_k}$.  
Then $F_I$ is homeomorphic to 
\[
F(n_0) \x \Delta^{\chi_+(n_0)} \x 
F(n_1) \x \Delta^{2 \chi_+(n_1)} \x 
\dots \x 
F(n_{k-1}) \x \Delta^{2 \chi_+(n_{k-1})} \x 
F(n_k) \x \Delta^{\chi_+(n_k)}
\]
where $\chi_+(x)$ is $1$ if $x>0$ and $0$ otherwise.  Since $n_0, \dots, n_k<m$, the above space is a product of contractible spaces by induction on $m$, and hence $F_I$ is contractible.

Now suppose that any nonempty union of fewer than $q$ spaces $F_I$ is contractible.  
Consider $F_{I_1} \cup \dots \cup F_{I_q}$.  
We know that $F_{I_q}$ is contractible, and by induction on $q$, so is $F_{I_1} \cup \dots \cup F_{I_{q-1}}$.  
For any two sublists $I$ and $J$ of $(1,\dots,m)$, we have $F_I \cap F_J = F_{I \cup J}$.  
Thus $F_{I_q} \cap (F_{I_1} \cup \dots \cup F_{I_{q-1}}) = F_{I_1 \cup I_q} \cup \dots \cup F_{I_{q-1} \cup I_q}$, which is contractible by induction on $q$.  Hence $F_{I_1} \cup \dots \cup F_{I_q}$ is contractible, as the union of two contractible complexes whose intersection is a contractible subcomplex.
Therefore $F(m)$ is contractible, and we conclude that any preimage $p_m^{-1}(C)$ is contractible.
\end{proof}

\begin{theorem}
\label{T:proj-is-htpy-eqv}
The surjection $p_m: \Ov^1(m) \to \Cact^1(m)$ is a simple, $\SS_m$-equivariant homotopy equivalence.
\end{theorem}
\begin{proof}
By Theorem \ref{T:proj-is-cellular}, $p_m$ is a cellular map.  Via barycentric subdivision, we turn it into a simplicial map between finite simplicial complexes.  By Lemma \ref{L:fibers-contractible}, the preimage of every point is contractible.  Then by work of M.~Cohen \cite[Theorem 11.1]{MCohen:1967}, $p_m$ is a simple homotopy equivalence.  The $\SS_m$-equivariance is immediate from the definition of $p_m$.
\end{proof}

The previous Theorem connects to our cactus action, while the next one connects to the action of little 2-cubes on long knots.


\begin{theorem}
\label{T:Ov^1-eqv-to-Ov}
The inclusion $\Ov^1(m) \incl \Ov(m)$ is an $\SS_m$-equivariant homotopy equivalence.  
\end{theorem}

\begin{proof}
The fact that the inclusion is $\SS_m$-equivariant is clear, so we just need to check that it is a homotopy equivalence.  Let $PB_m$ denote the group of $m$-component pure braids.  By \cite[Proposition 3.3.19]{Kaufmann:2005}, $\mathcal{C}act^1(m)$ is a $K(PB_m,1)$ space.  We take as its basepoint the corolla cactus $\star_m\in \Cact^1$ with white vertices labelled $1,2,\dots,m$ from left to right.  Given $1\leq i<j\leq m$, let $\beta_{ij}$ be the loop where lobe travels over lobes $i+1, \dots, j$, then lobe $j$ travels over lobe $i$, and finally lobe $i$ travels over lobes $j-1, \dots, i+1$ to its original position.  Then $\{\beta_{ij}\}_{1\leq i<j\leq m}$ generates $\pi_1(\Cact^1(m), \star_m)$.

By Theorem ~\ref{T:proj-is-htpy-eqv}, $\Ov^1(m)$ is also a $K(PB_m, 1)$ space.  We take the basepoint $\ast_m \in \Ov^1(m)$ to be the increasing list of points $(x_1,y_1,\dots,x_m,y_m)$ in $\I$ where each consecutive pair of points spaced $|\I|/m$ apart.  Given $1\leq i<j\leq m$, let $\alpha_{ij}$ be the loop in $\Ov^1(m)$ defined analogously to $\beta_{ij}$, but with intervals instead of lobes.  By definition of $\Ov^1$, this requires the interval of lower height at each stage of the loop to grow to width $2|\I|/m$ and then shrink back to its original width.  Then $p_m(\alpha_{ij}) = \beta_{ij}$.
Since $p_m$ is 
a homotopy equivalence, $\{\alpha_{ij}\}_{1 \leq i < j \leq m}$ generates $\pi_1(\Ov^1(m), \ast_m)$.

Finally, the image of $\{\alpha_{ij}\}_{1 \leq i,j \leq m}$ under the inclusion $\Ov^1(m) \incl \Ov(m)$ generates $\pi_1(\Ov(m), \ast_m)$.  In fact, its image under a homotopy inverse to the projection $\CC_2(m) \xtwoheadrightarrow{\simeq} \Ov(m)$ generates the fundamental group of the $K(PB_m, 1)$ space $\CC_2(m)$.
Since the inclusion $\Ov^1(m) \incl \Ov(m)$ is a map of aspherical spaces which sends a generating set to a generating set, it is a homotopy equivalence.
\end{proof}

\begin{remark}
We can further show that $\Ov^1(m)$ is a subcomplex and a (strong) deformation retract of $\Ov(m)$, when $\Ov(m)$ is given a suitable cellular structure.  We omit the proofs here.
\end{remark}

Each of the maps below is a pointed 
$\SS_m$-equivariant homotopy equivalence.  The first map is briefly described in Section \ref{S:ov-int-knots}, and $\iota_m$ is described in Remark \ref{R:eqv-of-cacti-as-K(B,1)-spaces}.  Therefore so is their composition:
\[
(\CC_2(m), \bullet_m) \xtwoheadrightarrow{}
(\Ov(m) , \ast_m) \xtwoheadrightarrow{r_m} (\Ov^1(m), \ast_m)
\xtwoheadrightarrow{p_m} (\Cact^1(m), \star_m) 
\xhookrightarrow{\iota_m} (\mathbb{P}\Cact(m), \iota_m(\star_m))
\]
Thus the proof of Theorem \ref{T:B} is complete.  


We conclude this Section with some questions.  If $G$ is a topological group, then the overlapping intervals operad (and hence the little $2$-cubes operad) acts on $\Omega G$, roughly by  concatenating loops according to the interval coordinates and multiplying elements of $G$ in orders given by their heights.  Salvatore \cite{Salvatore:2009} constructed an action on $\Omega G$ of a certain operad of cacti with spines.  Roughly, that operad is the same as $\Cact$ except that each lobe has a basepoint that need not be the local root.  One can map $\Cact$ to this operad by taking each basepoint to be the local root.  Hence $\Cact$ acts on $\Omega G$.

\begin{conjecture}
\label{conj:action-on-Omega-G}
Let $G$ be a topological group.
Via the maps of  symmetric sequences $\Ov \to \Ov^1 \to \Cact^1 \to \Cact$, the actions of $\Ov$ and $\Cact$ on $\Omega G$ are compatible up to homotopy at the space level.
\end{conjecture}

This conjecture is related to work of Hepworth \cite{Hepworth:TAMS} who established a compatibility between Salvatore's action on $\Omega G$ and the action of the framed little $2$-disks operad on $\Omega^2 BG$.  His compatibility involved an intermediate operad and operad maps.  Our conjecture involves only maps of spaces, but the intermediary $\Ov^1$ is 
perhaps more geometrically tractable than the intermediate operad in Hepworth's work.  

The following question is motivated by recent work on infinity operads of cacti \cite{BCLRRW}:

\begin{question}
\label{Q:infinity-operad}
Does the symmetric sequence $\{ \Ov^1(m) \}_{m\in \N}$ of normalized overlapping intervals admit the structure of an infinity operad?  If so, does the projection $p_m: \Ov^1(m) \to \Cact^1(m)$ give rise to a compatibility of such structures on these spaces?
\end{question}

If Conjecture \ref{conj:action-on-Omega-G} holds and the answer to Question \ref{Q:infinity-operad} is affirmative, one could also ask if the compatibility of actions of overlapping intervals and cacti extends to such an infinity operad structure.

\section{The cactus action on the tower and the cubes action on framed long knots}
\label{S:compatibility}

We will now finish the proof of Theorem \ref{T:A}.
That is, we will show compatibility at the level of the multiplication, Browder operation, and Dyer--Lashof operations in homology of the action of $\PCact$ on $T_n \K_d^{fr}$ with the action $\CC_2$ on $\K_d^{fr}$.  

Essentially, it suffices to compare the two actions using diagram \eqref{Eq:compatibility-diagram} below, which we now explain.  For brevity of notation, we write $\K^{fr}$ to mean $\K_d^{fr}$.  Recall that the map $\K^{fr} \to \tAM_n^{fr}$ is the composition of two maps.  The first is  $ev_n: \K^{fr}\to AM_n^{fr}$ (see Definition \ref{D:ev-map}).  The second is the map $\quot: AM_n^{fr} \to \tAM_n^{fr}$, which sends $\phi \in AM_n^{fr}$ to the composition $(\lim_{t\to 1} H_n(t,-)) \circ \overline{\phi} \circ \tau$ of the map $\tau: C_n\la \I \ra \to C_n \la \R\ra$ induced by multiplication by 2, the map $\overline{\phi}: C_n\la \R \ra \to C_n^{fr} \la \R^d\ra$ induced by $\phi : C_n\la \I \ra \to C_n^{fr}\la \I^d, \d \ra$, and the limit of a homotopy $H_n(-,t)$ that shrinks $n$-point configurations (see Section \ref{S:spatial-to-infinitesimal}).

The $\CC_2$-action on $\K^{fr}$ comes from the projection $\CC_2 \to \Ov$ and the $\Ov$-action on $\K^{fr}$ described in Section \ref{S:ov-int-knots}.  The map $r_m: \Ov(m) \to \Ov^1(m)$ is the homotopy inverse to the inclusion from Theorem \ref{T:Ov^1-eqv-to-Ov}.  The map $p_m:\Ov^1(m) \to \Cact^1(m)$ (see Definition \ref{D:proj-to-cacti} or Figure \ref{F:arity2}) is a homotopy equivalence by Theorem \ref{T:proj-is-htpy-eqv}.  The inclusion $\iota_m: \Cact^1(m) \incl \PCact(m)$ is a homotopy equivalence, as mentioned in Remark \ref{R:eqv-of-cacti-as-K(B,1)-spaces}.
\begin{equation}
\label{Eq:compatibility-diagram}
\xymatrix@R1.0pc{
\CC_2(m) \x (\K^{fr})^m \ar@{->>}[d]^-{\simeq} \ar[r]^-{A_m} & \K^{fr} \ar@{=}[d] \\
\Ov(m) \x (\K^{fr})^m  \ar@{->>}[d]_-{r_m \x \mathrm{id}}^-{\simeq} \ar[r]^-{A_m} & \K^{fr} \ar@{=}[d] \\
\Ov^1(m) \x (\K^{fr})^m  \ar[r]^-{A_2} \ar_-{p_m \x (\quot \circ ev_n)^m}[d]  
& \K^{fr} \ar^-{\quot \circ ev_n}[d] \\
\Cact^1(m) \x (\tAM_n^{fr})^m \ar[r]^-{\alpha_m} \ar@{^(->}_-{\iota_m \x \mathrm{id}}^-{\simeq}[d] & \tAM_n^{fr} \ar@{=}[d] \\
\PCact(m)  \x (\tAM_n^{fr})^m \ar[r]^-{\alpha_m} & \tAM_n^{fr} 
}
\end{equation}

The top and bottom squares commute because in each case, the action of the upper-left space of $m$-ary operations is pulled back from the action of the lower-left space of $m$-ary operations.  If we replaced $r_m$ above by its homotopy inverse, namely the inclusion $i_m:\Ov^1(m) \incl \Ov(m)$, then the resulting upward-flowing square would commute for the same reason.   Thus the square immediately below the top one commutes up to homotopy.  
Therefore establishing the homotopy commutativity of the third square will yield the desired compatibility.
We begin with the simpler case $m=2$ in Theorem \ref{T:browder-compatible-space-level}.  
The reader who is interested in a shorter exposition may skip ahead to Theorem \ref{T:dyer-lashof-compatible-space-level}, which is essentially the case of arbitrary $m$ and which we prove in a similar manner to Theorem \ref{T:browder-compatible-space-level} but with fewer details.

\subsection{The multiplication and Browder operation}

\begin{theorem}
\label{T:browder-compatible-space-level}
The square 
\begin{equation}
\label{Eq:browder-square}
\xymatrix@R1.5pc{
\Ov^1(2) \x \K^{fr} \x \K^{fr}  \ar[r]^-{A_2} \ar_-{p_2 \x (\quot \circ ev_n) \x (\quot \circ ev_n)}[d]  & \K^{fr} \ar^-{\quot \circ ev_n}[d] \\
\Cact^1(2) \x \tAM_n^{fr} \x \tAM_n^{fr} \ar[r]^-{\alpha_2}  & \tAM_n^{fr} 
}
\end{equation}
commutes up to homotopy.  
\end{theorem}

\begin{figure}[h!]
\raisebox{2.5pc}{\includegraphics[scale=0.5]{cactus-2-circle.pdf}}
\qquad
\includegraphics[scale=0.18]{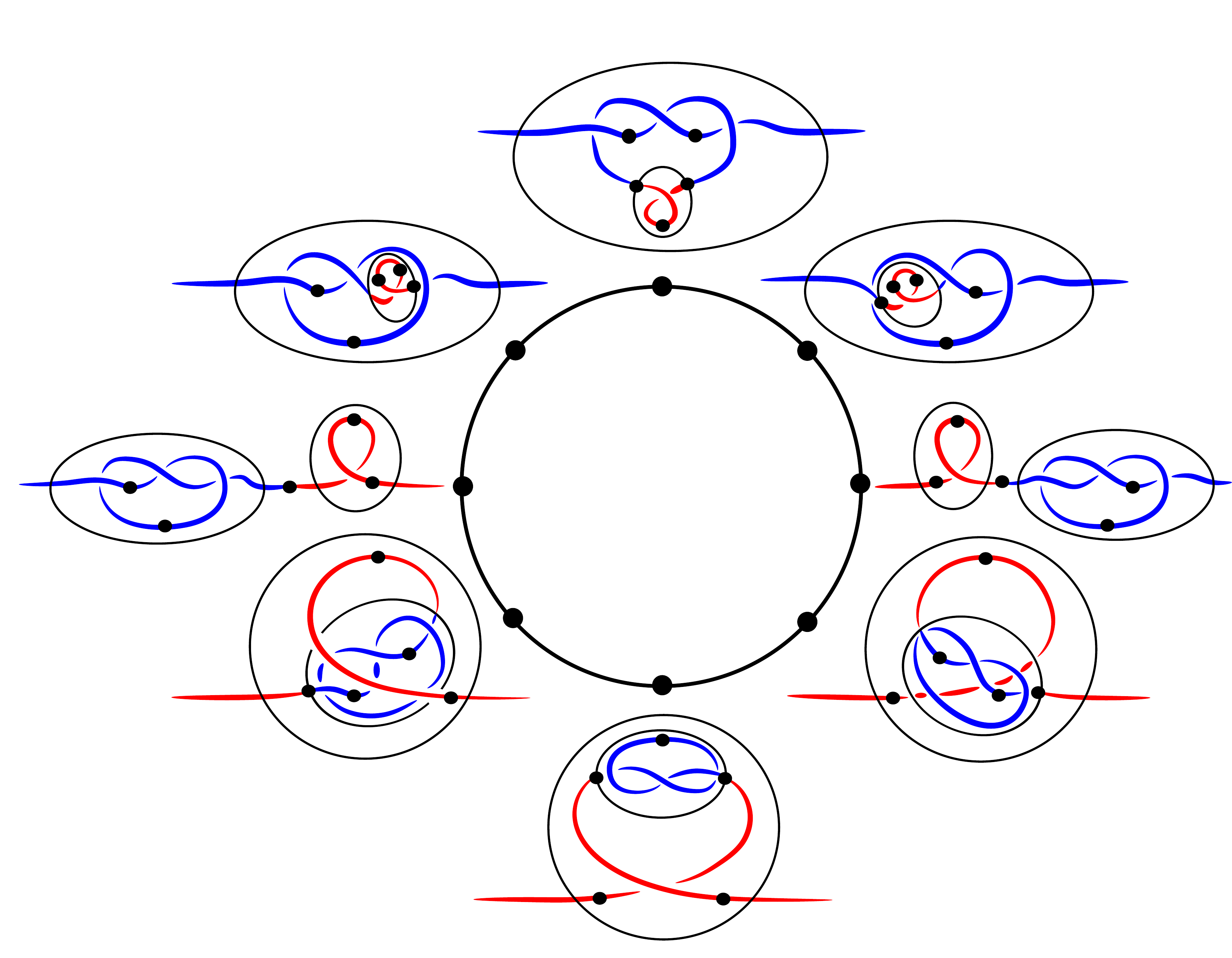}
\caption{On the left are eight sample points in $\Cact^1(2)$, and on the right is shown the action of each one on the images of two (framed) long knots under $q \circ ev_5$, i.e., the composition through the lower-left corner of the square \eqref{Eq:browder-square} for $n=5$.  Because of the map $q$, both the images of both knots under $ev_5$ are shrunk to infinitesimal size.  A circle around a knot represents a magnifying glass that makes this infinitesimal knot visible.  Without it, each knot would occupy just a point.  Near that point, where one interpolates between finite and infinitesimal scales, the embedding is either standard or given by the the larger-scale knot.  Each knot is colored by the color of the corresponding interval of top height; see also Figure \ref{F:arity2}.  The points along each knot are the images of the following points in $\I$: 
\[t_1=-0.75, \quad t_2=-0.5, \quad t_3=0, \quad t_4=0.5, \quad t_5=0.75.\]
The composition through the upper-right corner of the square \eqref{Eq:browder-square} yields configuration points along knots obtained by Budney's overlapping intervals action.
It gives pictures similar to the ones shown, but with no magnifying glasses (see Step 3 in the proof of Theorem \ref{T:browder-compatible-space-level}), some rescaling by factors less than 1 along the long axis (see Step 2), different framings (not shown, see Step 4), and a reparametrization of $\I$ whose configuration space is the domain of an aligned map (see Step 5).  Since the map $\Ov^1(2) \to \Cact^1(2)$ collapses two arcs onto the large points on the horizontal diameter, this composition through the upper-right corner gives an interval's worth of outputs for each of the these points. 
}
\label{F:cactus-knots-action}
\end{figure}

\begin{proof}
At a conceptual and informal level, the two compositions in the square \eqref{Eq:browder-square} differ in a handful of ways.  To describe the most significant difference, let $\phi\in \tAM_n^{fr}$ be an element in the image of one of the two compositions.  We loosely view a configuration in the output of $\phi$ as made up of two configurations, each coming from one of the two framed knots.  For the composition through the lower-left corner, one is at infinitesimal scale relative to the other, as indicated by the magnifying glasses in Figure \ref{F:cactus-knots-action}.  At the two vertices of $\Cact^1(2)$ on the horizontal diameter in Figure \ref{F:cactus-knots-action}, one can view either configuration as infinitesimal relative to the other one because relative to their scales they are infinitely far apart here.  For the composition through the upper-right corner however, the relative scale of the two configurations is finite, i.e., the ratio of the distance between their centers of mass to the diameter of either one is finite.  Along the two arcs in $\Ov^1(2)$ that are collapsed to two points in $\Cact^1(2)$ (as shown in Figure \ref{F:arity2}), there are two 1-parameter families of output configurations which vary in the composition through the upper-right corner but are constant families when we go through the lower-left corner.

There are a few more minor differences.  The action of an interval on a framed knot scales it only in the first coordinate, whereas the action of cacti on an aligned map scales it by the same factor in each coordinate.  The frames at each point differ between the two compositions.  Finally, the two compositions may differ by a reparametrization of $\I$, whose configuration space is the domain of an aligned map.

The homotopy that we will construct is similar to the one in the proof of compatibility of little intervals actions on $AM_n^{fr}$ and $\tAM_n^{fr}$ \cite[Proposition 4.12]{BCKS:2017}.
It will consist of a handful of steps, i.e., we will concatenate multiple homotopies, and these will mostly correspond to the handful of differences we outlined between the two compositions.  
The boxed formulas below show starting points of the homotopies we will concatenate.  
All but one of the steps below will factor through $K^{fr}$, i.e., they will arise from continuous maps
 $\Ov(2) \x (\K^{fr})^2 \x [0,1] \to \K^{fr}$.
The starting point of the homotopy is the map 
\begin{equation}
\boxed{
\bigl( [(L_1,L_2),\sigma] , (f_1, f_2) \bigr) \mapsto (\quot \circ ev_n) \bigl( L_{\sigma(1)} \cdot f_{\sigma(1)} \circ 
L_{\sigma(2)} \cdot f_{\sigma(2)} \bigr)
}
\end{equation}
where $L \cdot f$ is defined as in formula \eqref{Eq:single-interval-action}.

{\bf Step 1:}
We start by shrinking the two intervals $L_i$ to certain sub-intervals $L_i^\circ$, which we view as their cores.  
Ultimately, we will shrink configurations associated to the $L_i$, and this step will allow us to interpolate between a shrinking map and the identity map outside of $L_i$.  A key feature of the core will be that for any pair of intervals mapping to a corolla cactus and $i\neq j \in \{1,2\}$, the center of $L_i$ is sent to a point at infinity by $L_j^\circ \cdot H(-,1)$, where $H$ is the shrinking homotopy defined in formula \eqref{Eq:shrinking-H}.  Indeed the factor of $4$ below comes from the fact that $H(-,t)$ is the identity only outside $[-4,4]$.

For a little interval $L:\I \to \I$, define $L^\circ: \I \to \I$ by 
\[
L^\circ(t):=L(t/4).
\]
Thus $L^\circ(\I)$ is the inner fourth of $L(\I)$.
Perform a straight-line homotopy $\Ov^1(2) \x [0,1] \to \Ov(2)$ from $[(L_1, L_2), \sigma]$ to $[(L_1^\circ, L_2^\circ), \sigma]$.  Since each $L_i^\circ$ is a subinterval of $L_i$, $L_1(\I)^\circ \cap L_2^\circ (\I) \neq \varnothing$ implies $L_1(\I) \cap L_2(\I) \neq \varnothing$, so this homotopy is well defined, no matter which permutation $\sigma$ is chosen in a representative.  
Applying the maps $\Ov^1(2) \x (\K^{fr})^2 \to \K^{fr}$ and then $\quot \circ ev_n$ gives a homotopy of maps $\Ov^1(2) \x (\K^{fr})^2 \to \tAM_n^{fr}$ starting at the composition through the upper-right corner in square \eqref{Eq:browder-square}.
Its ending point is the map 
\begin{equation}
\label{Eq:end-step-1}
\boxed{\bigl( [(L_1,L_2),\sigma] , (f_1, f_2) \bigr) \mapsto
(\quot \circ ev_n) \bigl( L_{\sigma(1)}^\circ \cdot f_{\sigma(1)} \circ 
L_{\sigma(2)}^\circ \cdot f_{\sigma(2)} \bigr)
}.
\end{equation}

{\bf Step 2:}  
We will now correct for the fact that the little intervals action scales only the first coordinate rather than all $d$ coordinates by the same factor.
Let $S_{\sigma(i)}^\circ$ be the linear map that scales by the same factor as $L_{\sigma(i)}^\circ$ but fixes 0.
Apply a path of maps parametrized by $t\in [0,1]$ which ends at a map where each of the two factors
$L_{\sigma(i)}^\circ \cdot f_{\sigma(i)}=
(L_{\sigma(i)}^\circ \x \mathrm{id}_{D^{d-1}}) \circ f_{\sigma(i)} \circ (L_{\sigma(i)}^\circ \x \mathrm{id}_{D^{d-1}})^{-1}$
is replaced by an element of $\K^{fr}$ satisfying 
\begin{equation}
\label{Eq:scale-by-same-factor}
\bx \mapsto \left\{
\begin{array}{ll}
(L_{\sigma(i)}^\circ \x  S_{\sigma(i)}^\circ \dots \x S_{\sigma(i)}^\circ) \circ f_{\sigma(i)} \circ (L_{\sigma(i)}^\circ \x \mathrm{id}_{D^{d-1}})^{-1}(\bx) &  \text{ if } x_1 \in L_{\sigma(i)}^\circ(\I) \\
(L_{\sigma(i)}^\circ \x \mathrm{id}_{D^{d-1}}) \circ f_{\sigma(i)} \circ (L_{\sigma(i)}^\circ \x \mathrm{id}_{D^{d-1}})^{-1}(\bx) & \text{ if } x_1 \notin L_{\sigma(i)}(\I).
\end{array}
\right.
\end{equation}
and which at any point in $L_{\sigma(i)}(\I) \setminus L^\circ_{\sigma(i)}(\I)$ is a weighted average of the two maps above.  
Since all the functions $\R \to \R$ in \eqref{Eq:scale-by-same-factor} are affine-linear and increasing, the resulting map is a smooth embedding, for any choice of smooth partition of unity that interpolates between the two formulas shown.
(Notice also that the long axis is mapped to itself outside of $L^\circ_{\sigma(i)}$.)
A path to this map can be obtained by a straight-line homotopy.  
We compose the result with $\quot \circ ev_n$ and write the ending point as
\begin{equation}
\label{Eq:end-step-2}
\boxed{\bigl( [(L_1,L_2),\sigma] , (f_1, f_2) \bigr) \mapsto
(\quot\circ ev_n) \bigl(
L_{\sigma(1)}^\circ \, \widehat{\cdot} \, f_{\sigma(1)} \circ 
L_{\sigma(2)}^\circ \, \widehat{\cdot} \, f_{\sigma(2)}
\bigr)
} 
\end{equation}
where $L_{\sigma(i)}^\circ \, \widehat{\cdot} \, f_{\sigma(i)}$ is shorthand for the map satisfying the properties in \eqref{Eq:scale-by-same-factor}.
That is, $L \, \widehat{\cdot}\, f$ denotes an action of a single interval $L$ on a framed knot $f$ which essentially differs from the action $\cdot$ in that it scales all $d$ coordinates by the same factor throughout $L$.

{\bf Step 3:}  
This is the key step in the homotopy.
Roughly, we will shrink to infinitesimal size the configuration associated to each interval $L_i$.  
We will describe the shrinking as a map to $\K^{fr}$, which will be defined just up to, but not including, the time when the size becomes infinitesimal.
The shrinking for each $L_i$ will ensure first that for a pair of intervals mapping to a corolla cactus, the two knots appear infinitely far apart, roughly because the $L_i^\circ$ are sufficiently far apart.  Second, for any other pair of intervals, it will ensure that the knot $f_{\sigma(2)}$ for the smaller interval $L_{\sigma(2)}$ appears infinitesimal relative to $f_{\sigma(1)}$ because it incurs an extra application of the shrinking homotopy.
The map $\quot: AM_n^{fr} \to \tAM_n^{fr}$ already introduces a shrinking of the whole configuration, but this poses no problems because configurations in $\tC_n^{fr}\la \R^d\ra$ are considered modulo scaling and translation.

Recall the family $H$ of smooth embeddings defined by formula \eqref{Eq:shrinking-H}.
By abuse of notation, let $H$ now denote the restriction of the latter family to a map $(\R \x D^{d-1}) \x [0,1] \to \R \x D^{d-1}$.  We apply the path $t \mapsto H(-,t)$ within each of the two factors $L_{\sigma(i)} \, \widehat{\cdot} \ f_{\sigma(i)}$ in \eqref{Eq:end-step-2}.  
More specifically, we consider the path of maps with $t \in[0,1]$ given for $t \in [0,1)$ by 
\begin{equation}
\label{Eq:step-3-part-1}
\bigl( [(L_1,L_2),\sigma] , (f_1, f_2) \bigr) \longmapsto 
(\quot \circ ev_n) 
\Bigl(
\bigl(L_{\sigma(1)}^\circ \, \widehat{\cdot} \ (H(-,t) \circ f_{\sigma(1)} ) \bigr)  \, \circ \,
\bigl(L_{\sigma(2)}^\circ \, \widehat{\cdot} \ ( H(-,t) \circ f_{\sigma(2)} ) \bigr)
\Bigr)
\end{equation}
and for $t=1$ by the limit of this map as $t$ approaches 1.  
This limit will exist essentially because we have post-composed with $(\quot\circ ev_n)$ to take us from $\K^{fr}$ to $\tAM_n^{fr}$.  
Indeed, by adjunction, we are considering a map $\Ov^1(2) \x (\K^{fr})^2 \x \Delta^n \x [0,1) \to \tC_n^{fr} \la \R^d\ra$, where the limit as $t\to 1$ exists, for the same reasons as indicated in the paragraph immediately after formula \eqref{Eq:shrinking-H}.
Here we again indicate a proof by describing the limit map in terms of its effect on points $(t_1,\dots,t_n)\in \Delta^n$.  

The $t_i$ in $L_{\sigma(1)}^\circ[-2,2]$ produce configuration points in an infinitesimal configuration at $L_{\sigma(1)}(0)$.
The $t_i$ in $L_{\sigma(1)}^\circ(-2,2)$ yield points at finite distance from each other in the infinitesimal configuration, while the $t_i$ at $L_{\sigma(1)}^\circ(\pm 2)$ yield points at $(\pm \infty, 0^{d-1})$ in it.
Any $t_i$ which approach $L_{\sigma(1)}^\circ(\pm 2)$ from the right or left respectively yield points which approach these two points at infinity from outside of the infinitesimal configuration.   
Similarly, those $t_i$ in or near $L_{\sigma(2)}^\circ[-2,2]$ are mapped to points in or near an infinitesimal configuration at 
$f_{\sigma(1)}(L_{\sigma(2)}(0),0^{d-1})$.  
This configuration lies within the first infinitesimal configuration if $L_{\sigma(1)}^\circ[-2,2] \cap L_{\sigma(2)}^\circ[-2,2] \neq \varnothing$, in which case it is rotated by the orthonormal frame of $f_{\sigma(1)}$ at $(L_{\sigma(2)}(0),0^{d-1})$.
The $t_i$ at $L_{\sigma(2)}^\circ(\pm 2)$ yield points at $(\pm \infty, 0^{d-1})$ in this second infinitesimal configuration, along the line tangent to $f_{\sigma(1)}|_{\R \x \{0^{d-1}\}}$ at $f_{\sigma(1)}(L_{\sigma(2)}(0),0^{d-1})$.  

At the end of Step 3, we have the map
\begin{equation}
\label{eq:end-step-3}
\boxed{
\bigl( [(L_1,L_2),\sigma] , (f_1, f_2) \bigr) \longmapsto
\lim_{t\to 1}(\quot \circ ev_n) 
\Bigl(
\bigl(L_{\sigma(1)}^\circ \, \widehat{\cdot} \ (H(-,t) \circ f_{\sigma(1)}) \bigr)  \ \circ \
\bigl(L_{\sigma(2)}^\circ \, \widehat{\cdot} \ (H(-,t) \circ f_{\sigma(2)}) \bigr)
\Bigr)
}
\end{equation}
where $H$ is the family of smooth shrinking maps defined in formula \eqref{Eq:shrinking-H}.

It is worth noting that if $L_{\sigma(2)} \cap \d \I \neq \varnothing$, then after this step, $f_{\sigma(2)}$ appears infinitely far from $f_{\sigma(1)}$, even if the interiors of the intervals overlap, i.e., even if $|L_{\sigma(1)}|>1$.  Indeed, $f_{\sigma(2)}$ is shrunk to a point $\pm 1/2$ on the long axis.  This point lies outside of $L^\circ_{\sigma(1)}((-2,2))$, which by the definition of $H$ is the interval corresponding to configuration points that appear to be a finite distance from $f_{\sigma(1)}$.

{\bf Step 4:}  
We have now obtained the desired output configurations up to reparametrization, as we will see in Step 5.  However, the frames still need to be modified.  
Let $F$ denote the composition $\alpha_2 \circ (p_2 \x (ev_n)^2)$ through the lower-left corner in \eqref{Eq:browder-square}, and let $G$ denote the map $\Ov^1(2) \x (\K^{fr})^2$ obtained at the end of Step 3.
Fix $t\in \I$.  Let $A:=Df_{\sigma(1)}( (L^\circ_{\sigma(1)})^{-1}(t), 0^{p-1})$, and let $B:=Df_{\sigma(2)}( (L^\circ_{\sigma(2)})^{-1}(t),0^{p-1})$. 
Recall that the Gram--Schmidt map $GS$ appears in the definition of ${ev}_n$ and that it is essentially applied before shrinking a framed knot to infinitesimal size.
An aligned map in the image of $G$ thus sends $t$ to $GS(AB)$, whereas an aligned map in the image of $F$ sends a time corresponding to $t$ (under an appropriate reparametrization) to $GS(A)GS(B)$.
Although these two matrices are not equal, they are homotopic in $GL(d)$.  Indeed, $GS: GL(d) \to GL(d)$ is homotopic to the identity, and thus each of $GS(AB)$ and $GS(A)GS(B)$ is homotopic to $AB$.  
In this step we perform a homotopy $\Ov^1(2) \x (\K^{fr})^2 \x \Delta^n \x [0,1] \to \tC_n^{fr} \la\R^d\ra = \tC_n\la\R^d\ra \x GL(d)^n$ which implements the homotopy from $GS(AB)$ to $GS(A)GS(B)$ at each of the $n$ frames, over all overlapping intervals elements and all framed knots.

{\bf Step 5:}  
Again let $F$ denote the composition $\alpha_2 \circ (p_2 \x (ev_n)^2)$ through the lower-left corner in \eqref{Eq:browder-square}, but now let $G$ denote the map $\Ov^1(2) \x (\K^{fr})^2$ obtained at the end of Step 4.
In this last step, we will connect $G$ to $F$ by essentially reparametrizing $\I$.
Let $[\bL, \sigma] \in \Ov^1(2)$ and $f_1, f_2 \in \K^{fr}$.
Suppose for now that $\sigma=\mathrm{id}$.  
Let $\phi, \phi' \in \tAM_n^{fr}$ be the images of $([\bL, \sigma] , f_1, f_2)$ under $F$ and $G$ respectively.  
A reparametrization will suffice because the output configuration points and frames from both $\phi$ and $\phi'$ roughly speaking trace out the shape of $f_1$ with $f_2$ inserted at infinitesimal scale at a certain point.

We now more carefully justify the claim that $\phi$ and $\phi'$ agree up to reparametrization.
Fix $\bt = (t_1, \dots, t_n)$ in $\Delta^n$, and partition the configuration points in $\phi(\bt)$ and $\phi'(\bt)$ into two parts: 
\begin{enumerate}[(I)]
\item those corresponding to $t_i$ which lie outside of $L_2$ and 
\item those corresponding to $t_i$ which lie in $L_2$.  
\end{enumerate}
Suppose that $k$ of the $n$ points lie in part (I).

For the following key observations, recall that the output configuration of an element of $\tAM_n^{fr}$ is considered only up to translation and scaling.  
For $\phi$, the collection of direction vectors between either a pair of points in part (I) or a point in part (I) and a point in part (II) can be obtained by applying $\quot \circ ev_k$ to $f_1 \circ (\rho_1 \x \mathrm{id}_{D^{d-1}})$ for a nondecreasing piecewise-smooth surjection  $\rho_1: \R \to \R$ that is the identity outside of $\I$.
The frame at a point in part (I) corresponding to time $t$ is $GS(Df_1 \circ (\rho_1(t), 0^{d-1}))$.
Similar statements are true for $\phi'$, with the role of $\rho_1$ played by a possibly different function $\rho_1':\R\to \R$ which satisfies the same properties as $\rho_1$.

Likewise, for $\phi$, the collection of direction vectors between points in part (II) can be obtained by applying $\quot \circ ev_{n-k}$ to $f_2 \circ (\rho_2 \x \mathrm{id}_{D^{d-1}})$ for a nondecreasing piecewise-smooth surjection $\rho_2: \R \to \R$ which is the identity outside of $\I$ and such that $\I$ is partitioned into closed intervals on which one of $\rho_1$ and $\rho_2$ is constant and one is strictly increasing.  
The frame at a point in part (II) corresponding to time $t$ is $GS(Df_1 \circ (\rho_1(t), 0^{d-1})) GS(Df_2 \circ (\rho_2(t), 0^{d-1}))$.
Similar statements are true for $\phi'$, with the role of $\rho_2$ played by a possibly different map $\rho_2': \R \to \R$.

Here the functions $\rho_1$ and $\rho_2$ are those defined in Section \ref{S:lobe-param}.  The claimed effect of $\phi'$ on configuration points is guaranteed by Steps 2 and 3, while the claimed effect of $\phi'$ on frames is guaranteed by Step 4.

The last point in checking that $\phi$ and $\phi'$ agree up to reparametrization is that for a fixed $\bt \in \Delta^n$, the conglomerations of points in part (II) in $\phi(\bt)$ and $\phi'(\bt)$ have the same location.  Indeed, suppose these conglomerations for $\phi(\bt)$ and $\phi'(\bt)$ are located at $f_1(t,0^{d-1})$ and $f_1(t',0^{d-1})$ respectively.  Let $s_0$ (respectively $s_1$) be the distance between the left (respectively right) endpoints of $\I$ and $L_2$.  That is, $s_0$ and $s_1$ are the gaps outside $L_2$.  Then the ratios $(t+1):(1-t)$ and $s_0:s_1$ are equal by Definition \ref{D:AM-quot}, while $(t'+1):(1-t')$ is the ratio of the two subintervals of $L^\circ_1([-2,2]) = L_1([-1/2, 1/2])$ on either side of the midpoint of $L_2$, which is also $s_0:s_1$.  So $t=t'$, as claimed.

Finally, we describe the overall reparametrization of $\I$ that will complete the homotopy.
Let $\rho: \R \to \R$ be the piecewise-smooth homeomorphism which is the identity outside of $\I$ and whose derivative agrees with that of $\frac{1}{2}(\rho_1+ \rho_2)$.
Similarly define a homeomorphism $\rho': \R \to \R$ with the roles of $\rho_1$ and $\rho_2$
played by $\rho_1'$ and $\rho_2'$.
By the constructions of $F$ and $G$, the functions $\rho_i$ and $\rho_i'$ do not depend on $f_1$, $f_2$, or $\bt$.  They depend continuously on $\bL \in \Ov^1(2, \mathrm{id})$ because $\phi$ and $\phi'$ do.  
Thus so do $\rho$ and $\rho'$.  
If $\sigma \neq \id$, an analogous description applies, but with the roles of the indices $1$ and $2$ reversed, so we obtain families of homeomorphisms $\rho([\bL, \sigma])$ and $\rho'([\bL, \sigma])$ parametrized by $\Ov^1(2) \cong S^1$.
Since the space $\mathrm{Homeo}^+(\I)$ of orientation-preserving homeomorphisms $\I \to \I$ is contractible, there is a homotopy $\Ov^1(2) \x [0,1] \to \mathrm{Homeo}^+(\I)$ from $\rho'([\bL, \sigma])$ to $\rho([\bL, \sigma])$.  
Pre-composing the output of $G$ by the induced map 
$\Delta^n \x [0,1] \to \Delta^n$  
gives the required homotopy from $G$ to $F$. 
We thus obtain a homotopy from the composite through the upper-right corner to the composite $F$ through the lower-left corner in \eqref{Eq:browder-square}.
\end{proof}

\begin{remark}
A similar statement \cite[Proposition 4.12]{BCKS:2017} appeared for the compatibility of $\CC_1$-actions on $AM_n^{fr}$ and $\tAM_n^{fr}$.  The proof there is essentially very similar to the one above.  The latter is longer not so much because it is conceptually more difficult but rather because here we have thoroughly covered some details that were omitted in the former.
\end{remark}

\begin{corollary}
The multiplications and Browder brackets on $\K^{fr}$ from the $\CC_2$-action and on $\tAM_n^{fr}$ from the $\PCact$-action are compatible via the evaluation map.
\end{corollary}
\begin{proof}
The homotopy commutativity of the square \eqref{Eq:browder-square} established in Theorem \ref{T:browder-compatible-space-level} implies the homotopy commutativity of the rectangle \eqref{Eq:compatibility-diagram}.  
Although we do not have the full structure of an operad map between  $\CC_2$ and $\PCact$, it is easily verified from the definition of the product and Browder bracket in homology that the latter commutativity implies the desired compatibilities.
\end{proof}

The proof of the assertions about the multiplication and Browder operation in Theorem \ref{T:A} is now complete.  
It remains to address the compatibility of the Dyer--Lashof operations.

\subsection{The Dyer--Lashof operations}

\begin{theorem}
\label{T:dyer-lashof-compatible-space-level}
For any $m\geq 1$, the square 
\begin{equation}
\label{Eq:dyer-lashof-square}
\xymatrix@R1.5pc{
\Ov^1(m) \x_{\SS_m} \left(\K^{fr}\right)^m   \ar[r]^-{A_m} \ar[d]
\ar_-{p_m \x (\quot \circ ev_n)^m}[d]  
& \K^{fr} \ar^-{\quot \circ ev_n}[d] \\
\Cact^1(m) \x_{\SS_m} \left(\tAM_n^{fr}\right)^m \ar[r]^-{\alpha_m}  & \tAM_n^{fr} 
}
\end{equation}
commutes up to homotopy.  
\end{theorem}

\begin{proof}
First, $p_m$ is $\SS_m$-equivariant, so the left-hand vertical map in \eqref{Eq:dyer-lashof-square} is well defined.  It suffices to show that the square 
\[
\xymatrix@R1.5pc{
\Ov^1(m) \x \left(\K^{fr}\right)^m   \ar[r]^-{A_m} \ar[d]
\ar_-{p_m \x (\quot \circ ev_n)^m}[d]  
& \K^{fr} \ar^-{\quot \circ ev_n}[d] \\
\Cact^1(m) \x \left(\tAM_n^{fr}\right)^m \ar[r]^-{\alpha_m}  & \tAM_n^{fr} 
}
\]
commutes up to a homotopy $\Ov^1(m) \x (\K^{fr})^m \x[0,1] \to \tAM_n^{fr}$ which at each time is $\SS_m$-invariant, i.e., respects the relation given by the quotient $\Ov^1(m) \x (\K^{fr})^m \to \Ov^1(m) \x_{\SS_m} (\K^{fr})^m$.

We proceed as in the proof of Theorem \ref{T:browder-compatible-space-level}, so our exposition will be more concise.  
Again, we build a homotopy in four steps.  The starting point is the composition in \ref{Eq:dyer-lashof-square} through the upper-right corner: 
\begin{equation}
\bigl( [(L_1,\dots, L_m),\sigma] , (f_1,\dots, f_m) \bigr) \mapsto 
(\quot \circ ev_n) \bigl( L_{\sigma(1)} \cdot f_{\sigma(1)} \circ \dots \circ
L_{\sigma(m)} \cdot f_{\sigma(m)} \bigr)
\end{equation}

In {\bf Step 1}, we perform a preliminary shrinking to facilitate a shrinking to infinitesimal size by the map $H$ from formula \eqref{Eq:shrinking-H}.  By a homotopy, we will replace the intervals $L_i$ by their cores $L_i^\circ$, where we now define $L^\circ$ as the composition of the map $t \mapsto t/(2m)$ followed by $L$.  That is, we generalize Step 1 from $m=2$ to arbitrary $m$ by replacing $1/4$ by $1/(2m)$.  
This modification is not strictly necessary, but it has the feature that 
for any $[(L_1, \dots, L_m), \sigma]$ which maps to a corolla cactus, 
the $L_i^\circ$ are pairwise disjoint.  
(If instead we used $1/4$ for all $m$, the location of the resulting aligned maps would be corrected in Step 5.)
We achieve this step by a straight-line homotopy in $\Ov(m)$ from the $L_i$ to the $L_i^\circ$, and we end at the map 
\begin{equation}
\bigl( [(L_1,\dots, L_m),\sigma] , (f_1, \dots, f_m) \bigr) \longmapsto
(\quot \circ ev_n) \bigl( L_{\sigma(1)}^\circ \cdot f_{\sigma(1)} \circ \dots \circ
L_{\sigma(m)}^\circ \cdot f_{\sigma(m)} \bigr).
\end{equation}

In {\bf Step 2}, we arrange for knots and hence the aligned maps to be scaled by the same factor in all $d$ directions.  We perform a homotopy that replaces each factor of $L^\circ_{\sigma(i)} \cdot f_{\sigma(i)}$ above by a smooth embedding  satisfying 
\[
\bx \mapsto \left\{
\begin{array}{ll}
(L_{\sigma(i)}^\circ \x  S_{\sigma(i)}^\circ \dots \x S_{\sigma(i)}^\circ) \circ f_{\sigma(i)} \circ (L_{\sigma(i)}^\circ \x \mathrm{id}_{D^{d-1}})^{-1}(\bx) &  \text{ if } x_1 \in L_{\sigma(i)}^\circ(\I) \\
(L_{\sigma(i)}^\circ \x \mathrm{id}_{D^{d-1}}) \circ f_{\sigma(i)} \circ (L_{\sigma(i)}^\circ \x \mathrm{id}_{D^{d-1}})^{-1}(\bx) & \text{ if } x_1 \notin L_{\sigma(i)}(\I)
\end{array}
\right.
\]
where $S_{\sigma(i)}^\circ$ is the linear map that scales by the same factor as $L_{\sigma(i)}^\circ$ but fixes 0.  We denote this result $L_{\sigma(i)}^\circ \, \widehat{\cdot} \, f_{\sigma(i)}$.  This homotopy can be achieved by straight-line homotopies from $L_{\sigma(i)}^\circ \x \mathrm{id}_{D^{d-1}}$ to a smooth map with the same behavior for $\bx$ with $x_1 \in \I$ but which agrees with $L_{\sigma(i)}^\circ \x  S_{\sigma(i)}^\circ \dots \x S_{\sigma(i)}^\circ$ 
for $\bx$ with $x_1 \notin [-2m, 2m]$. 
Its endpoint is the map
\begin{equation}
\bigl( [(L_1,\dots, L_m),\sigma] , (f_1, \dots, f_m) \bigr) \longmapsto
(\quot\circ ev_n) \bigl(
L_{\sigma(1)}^\circ \, \widehat{\cdot} \, f_{\sigma(1)} \circ \dots \circ
L_{\sigma(m)}^\circ \, \widehat{\cdot} \, f_{\sigma(m)}
\bigr).
\end{equation}

In {\bf Step 3}, we shrink each aligned map to infinitesimal size.  
We use the map $H$ defined by the formula \eqref{Eq:shrinking-H} but with domain and codomain restricted so that it is a smooth self-embedding of $\R \x D^{d-1}$.  
Each $f_{\sigma(i)}$ incurs $i$ applications of $H$.  This ensures that configurations at higher lobes in the cactus appear at a smaller scale.  On the other hand, configurations at the same height appear at the same scale infinitely far apart, even though one will incur more applications of $H$ than the other. 
We perform the homotopy 
\begin{align*}  
\bigl( [(L_1,\dots, L_m),\sigma] , \ (f_1, \dots, f_m), \ t \bigr) \longmapsto \\
(\quot \circ ev_n) 
\Bigl(
\bigl(L_{\sigma(1)}^\circ \, \widehat{\cdot} \ (H(-,t) \circ f_{\sigma(1)}) \bigr)  \ \circ \ \dots \ \circ \
\bigl(L_{\sigma(m)}^\circ \, \widehat{\cdot} \ (H(-,t) \circ f_{\sigma(m)}) \bigr)
\Bigr)
\end{align*}
for $t\in [0,1)$ and extend to $t=1$ by taking the limit.  We end at the map 
\begin{align}  
\begin{split}
\bigl( [(L_1,\dots, L_m),\sigma] , \ (f_1, \dots, f_m) \bigr) \longmapsto \\
\lim_{t\to 1} \ 
(\quot \circ ev_n) 
\Bigl(
\bigl(L_{\sigma(1)}^\circ \, \widehat{\cdot} \ (H(-,t) \circ f_{\sigma(1)}) \bigr)  \ \circ \ \dots \ \circ \
\bigl(L_{\sigma(m)}^\circ \, \widehat{\cdot} \ (H(-,t) \circ f_{\sigma(m)}) \bigr)
\Bigr).
\end{split}
\end{align}

In {\bf Step 4}, we correct the frames by performing a homotopy in $GL(d)$ from $GS(A_1 \dots A_m)$ to the product $GS(A_1) \cdots GS(A_m)$, where $GS$ is the Gram--Schmidt map and where $A_i = Df_{\sigma(i)}( (L^\circ_{\sigma(i)})^{-1}(t), 0^{p-1})$.  
This homotopy is a map $\Ov^1(m) \x (\K^{fr})^m \x \Delta^n \x[0,1] \to \tC_n\la\R^d\ra \x GL(d)^n$ built out of the straight-line homotopy from $GS$ to the identity map.

In {\bf Step 5}, we complete the homotopy via reparametrizations of $\I$.  Let $F$ denote the composition in \eqref{Eq:dyer-lashof-square} through the lower-left corner, and let $G$ denote the map obtained at the end of Step 4.  Let $\phi, \phi' \in \tAM_n^{fr}$ be the images of a fixed element $([\bL, \sigma], f_1, \dots, f_m)$ under the maps $F$ and $G$ respectively.  

Fix $\bt \in \Delta^n$.
As in Definition \ref{D:action}, we can describe $\phi(\bt)$ by the formula 
\begin{equation}
\label{Eq:phi(t)}
\phi(\bt)=\bigcirc_{T} \left(
\mathbf{e}^{|b|}, 
\phi^1_{S_1}\left(\rho_1 \left(\bt^1\right)\right),\dots,
\phi^m_{S_m}\left(\rho_m \left(\bt^m\right)\right)
\right)
\end{equation}
where now $\phi^i = \quot \circ ev_n (f_i)$ and thus $\phi^i_{S_i}  = \quot \circ ev_{|S_i|-1} (f_i)$.  Recall that the $\bt^i$ are roughly sublists of $\bt$, but they may also include coordinates of local roots determined by $[\bL, \sigma]$ and the map $p_m:\Ov^1(m) \to \Cact^1(m)$.  As in Section \ref{S:lobe-param}, the symbols $\rho_1, \dots, \rho_m$ stand for both maps $\I \to \I$ and the induced maps $\Delta^n \to \Delta^n$ on non-decreasing $n$-tuples in $\I$.

The configuration $\phi'(\bt)$ is obtained via operadic insertion, given by the same tree $T$, of framed configurations in the images of $(\quot \circ ev_{|S_i|-1})(f_i \circ (\rho_i' \x \mathrm{id}_{D^{d-1}}))$, for $i\in \{1,\dots, m\}$, where $\rho_i': \R \to \R$ is a nondecreasing piecewise-smooth surjection that is the identity outside $\I$.
This fact is guaranteed at the level of configuration points by Steps 2 and 3, while it is guaranteed at the level of frames by Step 4.
Hence 
\begin{equation}
\label{Eq:phi'(t)}
\phi'(\bt)=\bigcirc_{T} \left(
\mathbf{e}^{|b|}, 
\phi^1_{S_1}\left(\rho_1' \left(\bt^1\right)\right),\dots,
\phi^m_{S_m}\left(\rho_m' \left(\bt^m\right)\right)
\right)
\end{equation}
where by abuse of notation,  $\rho_1', \dots, \rho_m'$ above denote the maps $\Delta^n \to \Delta^n$ induced by restricting domain and range to $\I$ and considering the effect on the compactified space of $n$-point configurations.
Note that both the $\rho_i$ in \eqref{Eq:phi(t)} and the $\rho_i'$ in \eqref{Eq:phi'(t)} depend continuously on $[\bL, \sigma] \in \Ov^1(m)$.

For $m=2$, we were able to arrange the homotopy so that only one overall reparametrization between maps $\rho$ and $\rho'$ was needed.  For $m>2$, the discrepancy between $G$ and $F$ can become more complicated, so we instead use homotopies between $\rho_i'$ and $\rho_i$ for each $i$, viewing all of them as maps $\I \to \I$.
Because $G$, and thus $\phi'$, is obtained as a limit of evaluation maps, each $\rho_i'$ is independent of $\bt$.  
Since $G$ does not depend on $f_1, \dots, f_m$, the same is true of $\rho_i'$.
Both $\rho_i$ and $\rho_i'$ depend continuously on $[\bL, \sigma]$.  
Finally,
the space $\Map(\I, \d \I)$ of continuous maps $\I \to \I$ which fix $\d \I$ pointwise is contractible, 
by the same elementary proof that $\Homeo^+(\I)$ is.
Thus we can find a homotopy 
$\Ov^1(m) \x [0,1] \to \Map(\I, \d I)$ from $[\bL, \sigma] \mapsto \rho_i'([\bL, \sigma])$ to $[\bL, \sigma] \mapsto \rho_i([\bL, \sigma])$ for each $i\in \{1,\dots,m\}$.  
(One could replace $\Map(\I, \d \I)$ by the subspace of non-decreasing such maps.)
Implementing these $m$ homotopies yields the necessary homotopy from $G$ to $F$.

Having completed the homotopy of maps $\Ov^1(m) \x (\K^{fr})^m \to \tAM_n^{fr}$, we just have to check the $\SS_m$-invariance throughout each of the four steps.  For this first three steps, the formulas are manifestly $\SS_m$-invariant.  In Step 4, the $\SS_m$-invariance of the required reparametrizations follows from that of the first three steps and of the action of $\PCact$ on $\tAM_n^{fr}$.
\end{proof}

\begin{corollary}
\label{C:dyer-lashof}
The Dyer--Lashof operations on $\K^{fr}$ from the $\CC_2$-action and on $\tAM_n^{fr}$ from the $\PCact$-action are compatible via the evaluation map.
\end{corollary}

\begin{proof}
In the diagram \eqref{Eq:compatibility-diagram}, the horizontal maps are $\SS_m$-invariant, and the vertical maps on the left-hand side are $\SS_m$-equivariant.  After taking quotients by $\SS_m$ on the left-hand side, 
the resulting rectangle still commutes up to homotopy by Theorem \ref{T:dyer-lashof-compatible-space-level}, 
yielding the following homotopy commutative square:
\begin{equation}
\label{Eq:dyer-lashof-square-2}
\xymatrix@R1.5pc{
\CC_2(m) \x_{\SS_m} \left(\K^{fr}\right)^m   \ar[r]^-{A_m} \ar_-{p_m \x (\quot \circ ev_n)^m}[d]  
& \K^{fr} \ar^-{\quot \circ ev_n}[d] \\
\PCact(m) \x_{\SS_m} \left(\tAM_n^{fr}\right)^m \ar[r]^-{\alpha_m}  
& \tAM_n^{fr} 
}
\end{equation}
Although we do not have an operad map between $\CC_2$ and $\PCact$, the Dyer--Lashof operations are defined by applying $H_*(-;\,\Z/p)$ to a row of square \eqref{Eq:dyer-lashof-square-2} with $m=p$.  
See work of F.~Cohen \cite[III.5, Definition 5.6]{FCohen} or Turchin \cite{Turchin:DLC} for more details.
The commutativity of this square thus gives the desired result.
\end{proof}

The proof of Theorem \ref{T:A} is now complete.

\begin{remark}
Theorem \ref{T:A} and its proof apply equally well when $d=1$ or $2$.  The spaces of framed long knots in $\R^1$ and $\R^2$ are contractible, so these settings are less interesting than those where $d\geq 3$.
\end{remark}

\section{Applications}
\label{S:applications}

We will now compute the action of cacti on certain classes in the Taylor tower and prove Theorem \ref{T:C}, thus recovering a result of Sakai \cite{Sakai:Nontrivalent}.  He used configuration space integrals to show the nontriviality of the Browder bracket $[e,f]$ of two specific cycles $e$ and $f$, which we define in Section \ref{S:cycles-of-knots}, after reviewing the space of long knots modulo immersions.  We will instead show its nontriviality using its class in the homology spectral sequence for the space of knots, described in Section \ref{S:spectral-sequence}.  There we relate $e$ and $f$ to corresponding classes in the spectral sequence; our proof of Theorem \ref{T:C} follows a similar argument.  The nontriviality of a certain class in the spectral sequence was first established by Turchin \cite{Turchin:sur-l'homologie}.  In Section \ref{S:geom-vs-alg-bracket}, we prove Theorem \ref{T:C} by showing that $[e,f]$ gives rise to this class, using Theorem \ref{T:browder-compatible-space-level}.  That is, the latter result helps us identify the bracket of cycles of knots with a certain bracket of classes in the spectral sequence.

\subsection{Knots modulo immersions and cycles of framed knots}
\label{S:cycles-of-knots}

\subsubsection{The space of knots modulo immersions}
Let $d \geq 3$.  A convenient auxiliary object will be the space of long knots modulo immersions, i.e., $\overline{\K}_d := \operatorname{hofiber}(\K_d \to \Imm_d)$.  A point in $\overline{\K}_d $ is a knot $f \in \K_d$ together with a path $\gamma$ through immersions from $f$ to the standard long unknot.  The Hirsch--Smale theorem \cite{Hirsch:1959, Smale:1959} gives an equivalence $\Imm_d \simeq \Omega S^{d-1}$ by taking the unit tangent vector at each point of the knot.  The forgetful map $\K_d \to \Imm_d$ is nullhomotopic \cite[Proposition 5.17]{Sinha:Operads}, so there is an equivalence $\overline{\K}_d \xrightarrow{\simeq} \K_d \x \Omega^2 S^{d-1}$.  The map to $\Omega^2 S^{d-1}$ is given by pre-concatenating the path $\gamma$ with a path $\gamma_0$ induced by a nullhomotopy of the forgetful map, thus yielding a loop $\gamma_0 \cdot \gamma$ of immersions based at that unknot.  

There is an equivalence $\K_d^{fr} \xrightarrow{\simeq} \K_d \x \Omega SO(d-1)$.  For $d \geq 4$, in which case $\K_d$ is path-connected, it is due to the fact that $\K_d \to \Imm_d$ is nullhomotopic \cite[\S1]{Salvatore:Knots}.  The map to $\Omega SO(d-1)$ is given by using such a nullhomotopy to standardize the first vector in the frame and thus obtain element of $SO(d-1)$ at each point; equivalently, it is given by lifting the path induced by a nullhomotopy from $\Imm_d$ to the space $\Imm_d^{fr}$ of framed long immersions.  For $d=3$, the equivalence $\K_d^{fr} \xrightarrow{\simeq} \K_d \x \Omega SO(d-1)$ comes from the framing number \cite[Proposition 9]{Budney:Cubes}.

There is a map $h: \overline{\K}_d \to \K_d^{fr}$ corresponding to the map $\Omega^2 S^{d-1} \to \Omega SO(d-1)$ induced by the fibration sequence $SO(d-1) \to SO(d) \to S^{d-1}$.  Alternatively, $h$ is given by constructing from $(f, \gamma) \in \overline{\K}_d$ a based loop $\gamma_0 \cdot \gamma$ of immersions (as above), lifting it to a path of framed immersions, and taking the framing at the midpoint of the path.  We thus have a commutative diagram, where the right-hand horizontal maps are forgetful maps:
\begin{equation}
\label{eq:3-knot-spaces}
\xymatrix{
\overline{\K}_d \ar[r]^-h\ar[d]_-\simeq & \K^{fr}_d \ar[r]^-p\ar[d]_-\simeq & \K_d \ar@{=}[d] \\
\K_d \x \Omega^2 S^{d-1} \ar[r] & \K_d \x \Omega SO(d-1)  \ar[r] & \K_d \\
}
\end{equation}

\subsubsection{Cycles of framed long knots}
We now construct cycles of framed long knots.
We start by constructing a $(d-3)$-cycle $e$ in $\K_d^{fr}$ for odd $d\geq 3$.
Consider this part of the long exact sequence in homotopy groups of the fibration $SO(d) \to S^{d-1}$:
\[
\dots \to \pi_{d-1} S^{d-1} \to \pi_{d-2} SO(d-1) \to \pi_{d-2} SO(d) \to
\dots
\]
The first map is injective because post-composing by the map $\pi_{d-2} SO(d-1) \to \pi_{d-2} S^{d-2}$ is essentially multiplication by the Euler characteristic of $S^{d-1}$ \cite[\S1,B)]{Levine:1985}.  
Fix the standard orientation on $S^{d-1}$, and identify the image of the positive generator under $\pi_{d-1} S^{d-1} \to \pi_{d-2} SO(d-1)$ with an element in $\pi_{d-3} \Omega SO(d-1)$.  
This element is also given by the clutching function for the tangent bundle of $S^{d-1}$, as described by Sakai \cite[\S 2.2.2]{Sakai:Nontrivalent}.  
Via the equivalence $\K_d^{fr} \simeq \K_d \x \Omega SO(d-1)$, we obtain a class in $\K_d^{fr}$.  
Let $e$ denote its image under the Hurewicz map in $H_{d-3} \K_d^{fr}$.

For our second cycle, we only require $d\geq 3$.  We take a resolution of a singular long immersion $g:\R \to \R^d$ with two transverse double-points.  Let $t_1^*, \dots, t_4^*$ be the preimages of the double-points of $g$, where $-1< t_1^* < \dots < t_4^* < 1$.  We require that $g(t_1^*)=g(t_3^*)$ and $g(t_2^*)=g(t_4^*)$.  Define a map $S^{d-3} \x S^{d-3} \to \K_d$ where each factor of $S^{d-3}$ parametrizes a resolution of a double-point.  Each $S^{d-3}$ corresponds to the unit sphere in the orthogonal complement of the two tangent vectors at a double-point.  As two configuration points pass through a resolved double-point, one on each strand, the resulting $(d-1)$-parameter family defines a degree-1 map of $S^{d-1}$, which we may take to be smooth.  This class represents the first nontrivial element of the homotopy groups of $\K_d$ \cite{Budney:Family}, so it factors through the quotient $S^{d-3} \x S^{d-3} \to S^{2d-6}$.  For more details, see the work of various authors \cite[\S 2.2.1]{Sakai:Nontrivalent}, \cite[\S 2]{CCRL:AGT}, \cite[Figure 3]{Budney:Family}.  We map this cycle to a cycle $f$ in $\K_d^{fr}$ by fixing the constant loop at $I_{d-1} \in SO(d-1)$ and using the equivalence  $\K_d \x \Omega SO(d-1) \simeq  \K_d^{fr}$.  We have thus defined classes $e \in H_{d-3} \K_d^{fr}$ and $f \in H_{2d-6} \K_d^{fr}$.  

\subsection{The homology spectral sequence of the cosimplicial model}
\label{S:spectral-sequence}
Recall the cosimplicial spaces for the three spaces of embeddings $\overline{\K}_d$, $\K_d$, and $\K_d^{fr}$ from Section \ref{S:cosimplicial-models}.  We consider the induced homology spectral sequences, which arise from a general construction for cosimplicial spaces.

Any cosimplicial space $X^\bullet$ gives rise to a second quadrant spectral sequence via the double complex 
$C_*(X^\bullet)$, where the vertical differential is given by the boundary map $\d$ on chains, and the horizontal differential is given by $\delta=\sum_{i=0}^{n+1} (-1)^i d^i_*$.  
Thus $E^0_{-p,q} = C_q(X^p)$, and $E^1_{-p,q}= H_q(X^p)$.
This spectral sequence converges \cite[\S5.6, p.~142]{Weibel} to the homology of the (direct product) total complex $\Tot C_*(X^\bullet)$, where $\Tot C_*(X^\bullet)_n = \prod_{q-p=n} C_q(X^p)$ and the differential on each factor is given by $\d + (-1)^p \delta$.  

The spectral sequence for $\K_d^{fr}$ is of most interest to us, but we will use those for $\overline{\K}_d$ and $\K_d$ as auxiliary objects.  So we now briefly describe and fix notation for all three.  We have the following for all $r \geq 1$:
\begin{itemize}
\item
In the spectral sequence for $\barK_d$, each term $\barE^r_{-p,q}$ is a subquotient of $H_q (C_p (\R^d))$.
\item
In the spectral sequence for $\K_d^{fr}$, each term $(E^r_{-p,q})^{fr}$ is a subquotient of $H_q (C_p (\R^d) \x O(d)^p)$.  
\item
In the spectral sequence for  $\K_d$, each term $E^r_{-p,q}$ is a subquotient of $H_q (C_p (\R^d) \x (S^{d-1})^p)$.  
\end{itemize}
More details about these spectral sequences can be found in various references \cite[\S7]{Sinha:Top}, \cite[\S2, \S7]{Sinha:Operads}, \cite[\S4]{Salvatore:Knots}, \cite[\S2]{Pelatt-Sinha}, \cite[\S9.6.4, \S10.4.4]{Munson-Volic:Book}.

\subsubsection{Evaluation and the spectral sequence}
Evaluating a framed knot on at most $p$ points 
induces a map
\begin{equation}
\label{eq:map-to-total-complex}
C_k(\K_d^{fr})
\to 
\prod_{i=0}^p C_{i+k} \left( \tC_i^{fr} \la \R^d \ra \right)
\end{equation}
from chains on $\K_d^{fr}$ to the total complex of chains on the cosimplicial model, 
by applying
the adjunctions
\[
\Delta^k \to \Map \left(\Delta^i, \tC^{fr}_i \la \R^d \ra\right)
\quad \rightsquigarrow \quad
\Delta^k \x \Delta^i \to \tC_i^{fr} \la \R^d \ra.
\]
The map \eqref{eq:map-to-total-complex} is a chain map because in the total complex, the boundary faces of $\Delta^i$ appearing in the vertical boundary map $\d$ are canceled by the images of $\Delta^{i-1}$ under the horizontal boundary map $\delta$.  
The map $C_*(\K_d^{fr}) \to \Tot C_*(X^\bullet)$ can also be described as the composition of two maps.  The first is the map on chains induced by an evaluation map $\K_d^{fr} \to \Tot^p X^\bullet$ to the (honest) partial totalization of the cosimplicial model $X^\bullet = \tC^{fr}_\bullet \la \R^d \ra$ for $\K_d^{fr}$, while the second is the map $C_*(\Tot^p X^\bullet) \to  \Tot_p C_*(X^\bullet)$ described by  Bousfield \cite[Lemma 2.2]{Bousfield:1987}.
The map \eqref{eq:map-to-total-complex} then induces a map 
\begin{equation}
\label{eq:map-to-E-infinity}
H_k \K_d^{fr} \to \prod_{i=0}^p (E^\infty_{-i, i+k})^{fr}
\end{equation}
 since the latter term is part of the associated graded group of the homology of $\Tot C_*(X^\bullet)$.  
Unlike elements of $\Tot X^\bullet$, elements of $\tAM_p^{fr}$ may have different restrictions to different faces of the same dimension.  However, the map \eqref{eq:map-to-E-infinity} factors through the map $(ev_p)_*: H_k \K_d^{fr}$ in $H_k \tAM_p^{fr}$, so we may study the former map via the map induced by $ev_p$ on chains.
This also means that we can use the action of cacti on that image to understand its image in $(E^\infty_{*,*})^{fr}$.  

On cosimplicial models, the maps $h: \overline{\K}_d \to \K_d^{fr}$ and $p:\K_d^{fr} \to \K_d$ in diagram \eqref{eq:3-knot-spaces} correspond to maps 
\[
C_p \la \I^d, \d \ra \incl C_p \la \I^d, \d \ra \x O(d)^p \to C_p \la \I^d, \d \ra \x (S^{d-1})^p
\]
where the first map is given by fixing  $(I_d, \dots, I_d) \in O(d)^p$, while the second is induced by the fibration $O(d) \to S^{d-1}$.  These induce maps that we denote by the same letters as the corresponding maps on embedding spaces, 
and there is a commutative diagram with the vertical arrows induced by evaluation maps similar to \eqref{eq:map-to-E-infinity}:
 \begin{equation}
 \label{eq:evaluation-h-p}
 \xymatrix{
 H_*\overline{\K}_d \ar[r]^-h\ar[d] &  H_*\K^{fr}_d \ar[r]^-p\ar[d] &  H_*\K_d \ar[d] \\
\barE^\infty_{*,*} \ar[r]^-{h_*} & (E^\infty_{*,*})^{fr} \ar[r]^-{p_*} & E^\infty_{*,*}
 }
 \end{equation}
The commutativity of the first square on the factors of $O(d)$ comes from the fact that the composition 
$\Omega^2 S^{d-1} \to \Omega SO(d-1) \to \Omega SO(d)$ is nullhomotopic.

\subsubsection{Bracket expressions in the spectral sequence}

Crucial to our calculations is a geometric interpretation of classes in the spectral sequence.  They will come from the spectral sequence for $\overline{\K_d}$, so it suffices to understand $H_q(C_p (\R^d))$.  Algebraically, this group is spanned by products of brackets of elements $x_1, \dots, x_p$ in which no $x_i$ is repeated \cite{FCohen, Sinha:Disks}.  Geometrically, a product $A \cdot B$ is represented by the juxtaposition in disjoint balls in $\R^d$ of representative chains for the expressions $A$ and $B$.  A bracket $[x_i, x_j]$ is represented by the $i$-th configuration point tracing a $(d-1)$-sphere around the $j$-th point (or vice-versa).  More generally, a bracket $[A,B]$ of expressions $A$ and $B$ is represented by taking representative chains for $A$ and $B$ to be contained in small disjoint balls in $\R^d$, and then having $A$ orbit $B$ in a $(d-1)$-sphere.  
In the compactification $C_p \la \R^d \ra$, instead of putting a configuration in a small ball, one can insert it into a configuration point $x_i$; that is, one can represent iterated brackets via the insertion maps $\circ_i$.
This gives a geometric interpretation of any class in $\overline{E}^r_{-p,q}$ for any $r \geq 1$.

Specifically, we start with the bracket expressions $\alpha:=[x_1, x_2]$ and $\beta:=[x_1,x_3][x_2,x_4]$,  which represent classes in $\barE^*_{*,*}$.  We also consider their images under $h_*$ in $(E^*_{*,*})^{fr}$, which we sometimes denote by the same bracket expressions, since $h_*$ is induced by inclusions $C_p \la \R^d \ra \incl C_p \la \R^d \ra \x O(d)^p$.
Both $\alpha$ and $\beta$ can be defined at $\barE^1_{*,*}$, and they lie in degrees low enough for one to verify that they support no nonzero differentials and hence survive to $\barE^\infty_{*,*}$, as indicated by Salvatore \cite[Proof of Theorem 2]{Salvatore:Knots}.  
More important to us than that verification is the following geometric interpretation of these bracket expressions.
Salvatore's work mentions an analogue of it for knots modulo immersions, which by diagram \eqref{eq:evaluation-h-p} implies this result for framed knots.  Nonetheless, we prove it in detail because our detection of Sakai's Browder bracket class will be proven similarly.

\begin{proposition} 
\label{P:eval-brackets}
Let $d \geq 3$.
\begin{enumerate}
\item[(a)] If $d$ is odd, the image in $({E}^\infty_{-2,d-1})^{fr}$ of the class $e$ defined in Section \ref{S:cycles-of-knots} 
is represented by  $h_*\alpha$.
\item[(b)] The image in $({E}^\infty_{-4,2(d-1)})^{fr}$ of the class $f$ defined in Section \ref{S:cycles-of-knots} 
is represented by $h_*\beta$.
\end{enumerate}
\end{proposition}

\begin{proof}
A priori, evaluating either $k$-cycle (where $k=d-3$ or $2d-6$) on at most $p$ points gives an element of
$\prod_{i=0}^p C_{i+k} \left( \tC_i \la \R^d \ra \x O(d)^i \right)$ that is a cycle in the total complex, 
rather than a cycle in $\tC_p \la \R^d \ra \x O(d)^p$.  
We will find a homotopy of the map 
\[
\Map (\Delta^k, \K_d^{fr}) \to 
\prod_{i=0}^p \Map (\Delta^i \x \Delta^k, \tC_i \la \R^d \ra \x O(d)^i)
\]
that stays within the subspace determined by compatibility with the coface maps.  
It will modify the resulting maps from $\Delta^p \x \Delta^k$
near the boundary 
to yield a $(p+k)$-cycle in $\tC_p \la \R^d \ra \x O(d)^p$.  That is, our argument will be to essentially 
find lifts to the $E^1$-page of the images of $e$ and $f$ under the map \eqref{eq:map-to-E-infinity}.

{\bf Proof of part (a):}  Here $p=2$.  We start by unwinding the data from evaluation.
The chain $ev_2(e) \in C_{d-3}(\widetilde{AM}_2^{fr})$ induces a map $\Delta^{d-3} \x \Delta^2 \to \tC_2\la \R^d \ra \x O(d)^2$.
We get a map $\Delta^{d-3} \x \Delta^2 \to S^{d-1} \x O(d)^2$ by taking the direction between the two configuration points, and its projection to $S^{d-1}$ is constant at the at the basepoint $(1,0,\dots,0)$.  
In the factors of $O(d)$, this map is given by evaluating at $(t_1, t_2) \in \Delta^2$ the image in $C_{d-3}(\Omega SO(d))$ of $e \in C_{d-3}(\Omega SO(d-1))$.
In particular, for any fixed $s \in \Delta^{d-3}$, it is given on each edge of $\Delta^2$ by the evaluation $t \mapsto e(s)(t)$, followed by one of the three maps $x \mapsto (\ast, x)$, $x \mapsto (x, \ast)$, and $x \mapsto (x,x)$.

Since $e$ is (the image under the Hurewicz map of) an element in $\ker (\pi_{d-3} \Omega SO(d-1) \to \pi_{d-3} \Omega SO(d))$, there is a nullhomotopy in $O(d)$ of these maps from each edge of $\Delta^2$.
Performing this nullhomotopy on a collar going from each edge of $\Delta^2$ to the corresponding edge of a larger 2-simplex $D = \Delta^2 \cup (\I \x \d \Delta^2)$ (and applying a reparametrization $\Delta^2 \xrightarrow{\cong} D$) yields a map $\Delta^{d-3} \x \Delta^2 \to O(d)^2$ that is constant on $\Delta^{d-3} \x \d \Delta^2$.
To ensure compatibility with the coface maps throughout this homotopy, the map to $S^{d-1}$ along the collar of the edge where $t_1=t_2$ 
must be given by the first vector of either frame in $O(d)$.  Over $\Delta^2$, we thus obtain for each $s \in \Delta^{d-3}$ an element of $\Omega^2 S^{d-1} = \operatorname{hofiber}(\Omega SO(d-1) \to \Omega SO(d))$.  Allowing $s \in \Delta^{d-3}$ to vary gives precisely a lift of $e \in \ker (\pi_{d-3} \Omega SO(d-1) \to \pi_{d-3} \Omega SO(d))$ which by the definition of $e$ is a generator of $\pi_{d-1}S^{d-1}$.  

In summary, we have performed a homotopy of the map $ev_2(e):\Delta^{d-3} \to \widetilde{AM}_2^{fr}$ such that the final endpoint is a family of aligned maps that send all of $\d \Delta^2$ to the basepoint.  The resulting homology class in $H_{d-1}( \tC_2^{fr}\la \R^d \ra)$ corresponds to the fundamental class of $S^{d-1} \simeq \tC_2\la \R^d \ra$ and hence to  $[x_1, x_2]$, i.e., to $h_*\alpha$.


{\bf Proof of part (b):} Here $p=4$.  By construction, $f$ is represented by a $(2d-6)$-parameter family that is constant at the standard long unknot outside small neighborhoods $U_1, \dots, U_4$ of the preimages in $t_1^*, \dots, t_4^*\in \I$ of the double-points that are resolved.  

For $1 \leq i < j \leq 4$, let $V_{i,j}$ be the set of $(t_1, \dots, t_4) \in \Delta^4$ such that either ${U_i}$ or ${U_j}$ contains none of $t_1, \dots, t_4$.  We need to consider only $V_{1,3}$ and $V_{2,4}$.  On each of these $V_{i,j}$, we can perform a homotopy of the output configurations in $\tC_4^{fr}\la \R^d \ra$ resulting from a nullhomotopy of the resolution of the double-point where $g(t_i^*)=g(t_j^*)$.  On $V_{1,3} \cap V_{2,4}$, we may perform these homotopies simultaneously.  
Since any point $(t_1, \dots, t_4) \in \d\Delta^4$ has at most three distinct coordinates in the interior of $\I$, we have $\d \Delta^4 \subset V_{1,3} \cup V_{2,4}$.  We thus implement a homotopy of the restriction of $ev_4(f)$ to $\Delta^{2d-6} \x \d \Delta^4$.  In more detail, this homotopy gives a homotopy of an extension of $ev_4(f)$ to a larger domain $\Delta^{2d-6} \x D$, where $D=\Delta^4 \cup (\I \x \d \Delta^4)$, which we pre-compose with a reparametrization $\Delta^4 \xrightarrow{\cong} D$, to obtain a homotopy of maps $\Delta^{2d-6} \x \Delta^4 \to \tC_4^{fr} \la \R^d \ra$ compatible with the coface maps.

Evaluation of $f$ on $4$ points gives a chain in $\tC_4^{fr}\la \R^d \ra$ of dimension $2(d-3)+4 = 2d-2$.  
At the final endpoint of the homotopy, there is at most one resolution at each point in $\d \Delta^4$, specified by $d-3$ parameters.  That is, at each point, the family factors through one of the two projections $S^{d-3} \x S^{d-3} \to S^{d-3}$.
Hence evaluation over all of $\d \Delta^4$ is given by $(d-3)+3=d$ parameters.  
Thus we ultimately obtain a class in $H_{2d-2}(\tC_4^{fr}\la \R^d \ra, \, A)$ where $A$ is a subspace with no homology above degree $d$.

If $d\geq 4$, then $(2d-2) - d \geq 2$, and hence $H_{2d-2}(\tC_4^{fr}\la \R^d \ra, \, A) \cong H_{2d-2}(\tC_4^{fr}\la \R^d \ra)$.
If $d=3$, we can represent $f \in H_0\K_3^{fr}$ by an alternating sum of four knots, where one is a trefoil and the other three are unknots.  For the trefoil summand, we can perform a homotopy on each of the two $V_{i,j}$ that implements the effect of a crossing change on output configurations in $\tC_4^{fr}\la \R^d \ra$.  Changing either or both crossings yields an unknot, so by a further homotopy, we can achieve a constant map on $\d \Delta^4$.  The three unknot summands yield aligned maps that are homotopic to maps constant on all of $\Delta^4$.  All of these homotopies remain within the space of aligned maps because they are induced by embeddings (away from double-points where crossing changes occur).

In summary, for any $d\geq 3$, we have constructed a homotopy of $ev_4(f): \Delta^{2d-6} \to \tAM^{fr}_4$ such that the final endpoint is a family that is either constant on $\d \Delta^4$ or maps it to a subspace of codimension at least 2.  We thus obtain a class in $H_{2d-2}(\tC_4^{fr}\la \R^d \ra)$.  The class $f$ is constructed so that $x_1$ orbits $x_3$ in a degree-1 spherical trajectory, and so that $x_2$ similarly orbits $x_4$.  Thus the class in $H_{2d-2}(\tC_4^{fr}\la \R^d \ra)$ corresponds to $[x_1, x_3][x_2,x_4]$, i.e., to $h_*\beta$.
\end{proof}

\begin{remark}
\label{R:framed-unknot}
Here is another construction of the class $e \in H_{d-3} \K_d^{fr}$, which is similar to that of $f$ and further indicates its connection to the spherical class corresponding to the bracket expression $[x_1, x_2]$.  
We do not need this construction for our results, so we only sketch the details.

Take an immersion $\R \to \R^d$ with one double-point, and resolve the double-point in all the normal directions to obtain a family of embeddings parametrized by $S^{d-3}$.  
This family is homotopic to the constant family at the standard long knot.  By taking the unit tangent vector, we get from such a homotopy a $(d-3)$-parameter family of paths in $\Omega S^{d-1}$ from the element $\omega$ induced by the immersion with a double-point (or more precisely, a slight perturbation of it) to the constant based loop.  One can choose this family of paths so as to induce a degree-$1$ map of $S^{d-1}$.  If $d=3$, one can do so by using the two Reidemeister I moves, corresponding to the two embeddings parametrized by $S^0$.  For $d>3$, one can generalize this construction.  
We thus get an element of $H_{d-3}\overline{\K}_d$.  
By the choice of family of paths, this element corresponds to a generator of $\pi_{d-3}(\Omega^2 S^{d-1})$ (via a nullhomotopy of $\omega$, which is homotopic to the image of an embedding).  We then take its image under the map $\overline{\K}_d \to \K_d^{fr}$, which is thus $e$.
\end{remark}

\subsection{Compatibility of the knot and configuration bracket operations}
\label{S:geom-vs-alg-bracket}

The brackets appearing in expressions in the spectral sequence essentially come from the bracket operation on the homology of configuration spaces, while the Browder bracket we consider is an operation on the homology of the space of knots.  To show their compatibility, we use an avatar of the latter operation in the spectral sequence.  Namely, Turchin \cite[\S 2]{Turchin:NATO} identified $\barE^2_{*,*}$ with the Hochschild homology of the Poisson algebras operad and showed that as such, it has product and bracket operations.  
We write $\llbracket \cdot, \cdot \rrbracket$ for this bracket operation to distinguish it from the brackets on configuration points that generate the spectral sequence entries.
Since the classes $\alpha=[x_1,x_2]$ and $\beta=[x_1,x_3][x_2,x_4]$ survive to $\barE^2_{*,*}$, it makes sense to consider $\llbracket \alpha, \beta \rrbracket$.
The compatibility established in the following theorem will readily lead to the proof of Theorem \ref{T:C}, our last main result.

\begin{theorem}
\label{T:bracket-example} 
Let $d\geq 3$ be odd.  
Let $e \in H_{d-3} \K_d^{fr}$ and $f \in H_{2(d-3)} \K_d^{fr}$ be the cycles defined in Section \ref{S:cycles-of-knots}.
The Browder bracket $[e,f]$ and $h_*\llbracket \alpha, \beta \rrbracket \in (E^2_{-5, 3(d-1)})^{fr}$ have the same image in  $(E^\infty_{-5, 3(d-1)})^{fr}$.
%
\end{theorem}

\begin{proof}
The image of $[e,f]$ in $(E^\infty_{-5, 3(d-1)})^{fr}$ is determined by 
its image under $ev_5 : \K^{fr}_d \to \tAM_5^{fr}$.  
The evaluation map has the form $K \mapsto (ev_5(K): \Delta^5 \to \tC^{fr}_5 \la \R^d \ra)$.   
The class $(ev_5)_*[e,f]$ is a $(3d-8)$-cycle in $\tAM_5^{fr}$,
 and by an adjunction involving $\Delta^5$, it gives rise to a $(3d-3)$-chain in $\tC^{fr}_5 \la \R^d \ra$.  
Our first goal is to obtain a cycle from this chain, essentially lifting it from the $E^\infty$-page to the $E^1$-page.
Our remaining goal will then be to show that it agrees with $h_*\llbracket \alpha, \beta \rrbracket$.

To understand $(ev_5)_*[e,f]$, we use Theorem \ref{T:browder-compatible-space-level}, which implies that $(ev_5)_*[e,f]$ and $[(ev_5)_*(e), (ev_5)_*(f)]$ are homologous, where the bracket in the latter expression comes from the action of cacti.  
Thus it suffices to the image of the latter chain in $(E^\infty_{-5, 3(d-1)})^{fr}$.  Since $e$ and $f$ are parametrized by $S^{d-3}$ and $(S^{d-3})^2$ respectively, $[(ev_5)_*(e), (ev_5)_*(f)]$ corresponds to a map 
\begin{equation}
\label{eq:bracket-of-ev5}
(S^{d-3})^3 \x \Cact^1(2) \x \Delta^5 \to \tC^{fr}_5 \la \R^d \ra
\end{equation}
where $\Cact^1(2) \cong S^1$.  
As $C \in \Cact^1(2)$ varies, the local root of the higher of the two lobes moves, and this local root will essentially serve as a 6th point on the perimeter of $C$ along with the images of $t_1, \dots, t_5$.

As in the proof of Proposition \ref{P:eval-brackets}, we will show that up to homotopy of aligned maps, the restriction of the map \eqref{eq:bracket-of-ev5} to $(S^{d-3})^3 \x S^1 \x \d \Delta^5$ is either constant or lies in a subspace of  codimension at least 2.

First, let $X_1$ be the subspace of elements $(v_1, v_2, v_3; C; t_1, \dots, t_5) \in (S^{d-3})^3 \x \Cact^1(2) \x \Delta^5$ such that among the images of $t_1, \dots, t_5$ on $C$ and the local root of lobe 2, fewer than 2 distinct points lie in the interior of lobe 1.  The map \eqref{eq:bracket-of-ev5} is homotopic through aligned maps to one whose restriction to $X_1$ factors through the projection that kills the first factor of $S^{d-3}$; this is because, as in the proof of Proposition \ref{P:eval-brackets}, $e$ comes from $\ker(\Omega SO(d-1) \to \Omega SO(d))$, and with at most one configuration point, only the data in $O(d)$ is recorded.  

Next, we turn a subspace where lobe 2 of $C$ lacks the needed configuration points.  As before, let $U_1, \dots, U_4$ be small neighborhoods of the preimages of the two double-points resolved to obtain $f$ that are large enough to contain the support of the resolutions.  
Recall from Section \ref{S:lobe-param} that $\rho_i = \rho_i(C): \I \to \I$ is roughly a proper map that parametrizes the $i$-th lobe of  $C$ while being constant outside of the $i$-th lobe.  
Let $X_2$ be the subspace of elements $(v_1, v_2, v_3; C; t_1, \dots, t_5) \in (S^{d-3})^3 \x \Cact^1(2) \x \Delta^5$ where at least one $\rho_2(C)^{-1}(U_i)$ contains neither a preimage of the local root of lobe 1 of $C$ nor any of $t_1, \dots, t_5$.  
(Since $U_i$ is open, each such set contains either both or neither of the preimages of this local root.)  
There is a homotopy of the map \eqref{eq:bracket-of-ev5} through aligned maps to one whose restriction to $X_2$ factors through the projection that kills at least one of the second and third factors of $S^{d-3}$; this is because, as in the proof of Proposition \ref{P:eval-brackets}, one can implement a nullhomotopy of the resolution of the double-point to which the unoccupied interval $U_i$ corresponds.  

If $d=3$, then as before, three of the four components indexed by the second and third factors of $S^0$ are induced by unknots and hence yield aligned maps that are homotopic to constant maps.  The fourth is induced by a trefoil, and for that component, we can implement one or two crossing changes at each point where some $U_i$ is unoccupied.  These yield unknots, so through a homotopy of aligned maps, we obtain a map that is constant on $X_2$ when $d=3$.

Since we need 6 distinct points, including the local root of the higher lobe of $C$, for an element to lie outside $X_1 \cup X_2$, we have $(S^{d-3})^3 \x \Cact^1(2) \x \d \Delta^5 \subset X_1 \cup X_2$.  
The homotopies of aligned maps defined for $X_1$ and $X_2$ may be performed simultaneously.
The final endpoint is a map that is either constant on $(S^{d-3})^3 \x \Cact^1(2) \x \d \Delta^5$ (if $d=3$) or takes it into a subspace $A \subset \tC^{fr}_5 \la \R^d \ra$ of codimension $(d-3)+1>1$ (if $d >3$).  Thus we obtain a homology class in $H_{3d-3} (\tC^{fr}_5 \la \R^d \ra)$, using in the case where $d>3$ the fact that $H_{3d-3} (\tC^{fr}_5 \la \R^d \ra, \, A) \cong H_{3d-3} (\tC^{fr}_5 \la \R^d \ra)$.

It remains to check that this homology class corresponds to $h_* \llbracket \alpha, \beta \rrbracket$.  
Let $\mathcal{S}$ be the complement of $X_1 \cup X_2$ in $(S^{d-3})^3 \x \Cact^1(2) \x \d \Delta^5$.
It consists of multiple path components determined by the position of the local root of the higher lobe of the cactus $C$: either the local root of lobe 1 lies in one of $\rho_2^{-1}(U_1), \dots, \rho_2^{-1}(U_4)$, or the local root of lobe 2 lies in $\rho_1^{-1}((-1,1))$, in which case it lies either before or after the single $t_i$ that must lie on lobe 1.  See Figure \ref{F:cactus-bracket}.
We will show that these components correspond to bracket expressions.  

Letting a number $i$ stand for $x_i$, we write $h_*\alpha=[1,2]$ and $h_*\beta=[1,3]\cdot [2,4]$.  Turchin's bracket formula on the spectral sequence \cite[equations (2.1), (2.3), (2.5)]{Turchin:NATO} roughly says that $\llbracket \alpha, \beta \rrbracket$ is a signed sum of substituting $\alpha$ for each symbol in $\beta$ and vice-versa, increasingly labels as necessary to obtain distinct symbols.  Since $h_*$ is induced by inclusions $C_p \la \R^d \ra \incl C_p \la \R^d \ra \x O(d)^p$, the class $h_*\llbracket \alpha, \beta \rrbracket$ is given by the same symbols as $\llbracket \alpha, \beta \rrbracket$, and the precise formula yields 
\begin{align}
\begin{split}
\label{eq:bracket-formula}
h_*\llbracket \alpha, \beta \rrbracket = 
&- [[1,3]\cdot[2,4], 5] 
 + [1, [2,4]\cdot[3,5]] 
 - [[1,2],4]\cdot [3,5]
+ [1,4] \cdot [[2,3],5] \\
&- [1,[3,4]] \cdot [2,5]
+[1,3] \cdot [2,[4,5]].
\end{split}
\end{align}

Let  $B_1, \dots, B_6$ be the bracket expressions above, in the order given and without any minus signs.  We claim that each of the homology classes in $H_{3d-3}(\tC^{fr}_5\la \R^d \ra)$ indexed by the $B_i$ is parametrized by exactly one path component of $\mathcal{S}$, up to sign.
Indeed, having modified our aligned map on $X_1 \cup X_2$ and identified $H_* (\tC^{fr}_5 \la \R^d \ra, \, A)$ with $H_* (\tC^{fr}_5 \la \R^d \ra)$, we see that on any component of $\mathcal{S}$, each factor of $S^{d-3}$ together with two of the parameters $(s, t_1, \dots, t_5) \in S^1 \x \Delta^5$ parametrize a homology class $[j,k]$ in $\tC^{fr}_5\la \R^d \ra$.
The action of $\Cact^1(2)$ is defined so that configurations in the cycle parametrized by the higher lobe are inserted via some $\circ_i$ operation into configurations in the cycle parametrized by the lower lobe.
Thus each component of $\mathcal{S}$ parametrizes a class indexed by three bracket pairs, two of which are nested.  

\begin{figure}[h!]
\includegraphics[scale=0.9]{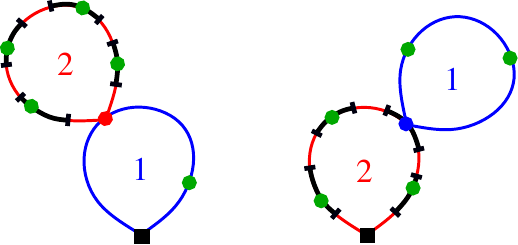}
\caption{Above are two elements of $\Cact^1(2) \x \Delta^5$.  The black intervals on lobe 2 correspond to $U_1, \dots, U_4$.  In either element, the two points on lobe 1, together with one factor of $S^{d-3}$, parametrize $\alpha=[1,2]$, while the four points on lobe 2, together with $(S^{d-3})^2$, parametrize $\beta=[1,3] \cdot [2,4]$.  
Hence the element on the left, together with $(S^{d-3})^3$, parametrizes a cycle where one substitutes $\beta$ for 1 in $\alpha$ and accordingly relabels 2 by 5. The result is $B_1=[[1,3]\cdot [2,4], 5]$, up to sign.  
Similarly, the element on the right, together with $(S^{d-3})^3$, parametrizes a cycle where one substitutes $\alpha$ for 3 in $\beta$, relabeling the inserted $[1,2]$ by $[3,4]$ and the original 4 by 5.  The result is $B_5=[1,[3,4]] \cdot [2,5]$, up to sign.}
\label{F:cactus-bracket}
\end{figure}

To fully understand these components, we consider the possible types of configurations of 6 points on the two types of cacti.
For example, there is a component $\mathcal{S}_1$ where the local root of lobe 2 lies in the interior of lobe 1
before $t_5$ (and sufficiently far from both $t_5$ and the global root). 
Similarly, there is a component $\mathcal{S}_5$ is where the local root of lobe 1 lies in $\rho_2^{-1}(U_3)$.  
Both are illustrated in Figure \ref{F:cactus-bracket}, where it is explained how $\mathcal{S}_1$ parametrizes $\pm B_1$ and $\mathcal{S}_5$ parametrizes $\pm B_5$.  After a similar analysis and an appropriate labeling for the remaining four components, we see that for each $i\in \{1,\dots, 6\}$, the component $\mathcal{S}_i$ parametrizes the cycle $B_i$, up to sign. 

Finally, to calculate the signs, we compare the parametrization of $\pm B_i$ by $\mathcal{S}_i \subset (S^{d-3})^3 \x S^1 \x \Delta^5$ to $(S^{d-1})^3$ to that by $(S^{d-1})^3$.  Since $d$ is odd, $d-3$ is even, these parametrization differ by the sign of the permutation of coordinates $(s, t_1, \dots, t_5) \in S^1 \x \Delta^5$ that groups them with the appropriate factors of $S^{d-3}$, where the order of these factors does not matter.  For example, for $B_1$ and $\mathcal{S}_1$, it is the odd permutation $(t_1, t_3, t_2, t_4, s, t_5)$; for $B_5$ and $\mathcal{S}_5$, it is the odd permutation $(t_1, s, t_3, t_4, t_2, t_5)$.  Thus $\mathcal{S}_1$ and $\mathcal{S}_5$ parametrize $-B_1$ and $-B_5$ respectively.  A similar verification shows that $\mathcal{S}_i$ parametrizes $\pm B_i$ with exactly the sign on $B_i$ in formula \eqref{eq:bracket-formula}.  
Thus $(ev_5)_*[e,f]$ maps to $h_*\llbracket \alpha, \beta \rrbracket \in (E^2_{-5, 3(d-1)})^{fr}$.
\end{proof}

\begin{remark}
The right-hand side of Figure \ref{F:cactus-knots-action} illustrates for $d=3$ the cycle detected in Theorem \ref{T:bracket-example}.  More accurately, it depicts $[f, \frac{1}{2}e]$, where $f$ is a trefoil with framing number 0, with two of its three crossings arising as resolved double-points, and $\frac{1}{2}e$ is an unknot with framing number 1.  The factor of $\frac{1}{2}$ arises because the image of $\pi_2S^2 \to \pi_1 SO(2)$ is spanned by twice a generator of the target group, which corresponds to an unknot with framing number 2.  This interpretation of $e$ aligns with Remark \ref{R:framed-unknot}.

This figure also shows evaluation at 5 points, albeit not quite at the configurations relevant to the homology calculations in the proof of Theorem \ref{T:bracket-example}.  The terms $B_1$ and $B_2$ correspond to two knots almost as in the SW and SE corners respectively.  The corresponding 5-point configurations would have one point on a red strand above (respectively below) the trefoil, while the other 4 points would be on the blue part, with 2 points at each of the 2 crossings arising from resolution of double-points.  The terms $B_3$ and $B_4$ correspond to knots  almost as in the NW corner, and $B_5$ and $B_6$ to knots almost as in the NE corner.  The corresponding 5-point configuration would have one point on a blue strand above (respectively below) the red portion, 2 others at the crossing in the red knot, and the remaining 2 at the other resolved double-point in the blue trefoil.
\end{remark}

We can now finish the proof of our last main result.

\begin{proof}[Proof of Theorem \ref{T:C}]
It suffices to show that $h_*\llbracket \alpha, \beta \rrbracket$ yields a nontrivial class in $(E^\infty_{-5, 3(d-1)})^{fr}$, for then the cycle $[e,f]$ must be nontrivial in homology.
We will consider the subgroups generated by $\llbracket \alpha, \beta \rrbracket$ and its images under $h_*$ and $p_*h_*$, and we will use the maps $\barE^r_{-p,q} \xrightarrow{h_*} (E^r_{-p,q})^{fr} \xrightarrow{p_*} E^r_{-p,q}$.
  
By calculations of Turchin \cite[(2.9.21)]{Turchin:sur-l'homologie}, $\barE^2_{-5,3(d-1)}$ is generated over $\Q$ by $\llbracket \alpha, \beta \rrbracket$ and $\alpha \llbracket \alpha, \alpha \rrbracket$; see also work of Salvatore \cite[Proof of Theorem 2]{Salvatore:Knots}.  In particular, $\llbracket \alpha, \beta \rrbracket$ is nontrivial in $\barE^2_{-5,3(d-1)}$.  
(Alternatively, we verified that $\llbracket \alpha, \beta \rrbracket$ is not in the image of $d^1$, using a formula from Pelatt's Ph.D.~thesis \cite[Theorem 3.1, Example 3.9]{Pelatt:PhD}.  It is a $d^1$-cycle because $\llbracket \cdot, \cdot \rrbracket$ is compatible with the differential.  It was insufficient to compute the image of $d^1$ modulo 2, consistent with Turchin's calculation that $\llbracket \alpha, \beta \rrbracket$ is twice a generator over $\Z$.)
From the same references, we know that $E^2_{-5,3(d-1)}$ is generated over $\Q$ by $p_*h_*\llbracket \alpha,\beta \rrbracket$, which is hence nontrivial.  Thus $h_*\llbracket \alpha,\beta \rrbracket \in (E^2_{-5,3(d-1)})^{fr}$ is also nontrivial.  

We now use $\barE^r_{-p,q}$ and $E^r_{-p,q}$ to rule out nonzero differentials involving $h_*\llbracket \alpha,\beta \rrbracket$.    
The fact that any outgoing differential from $\barE^r_{-5,3(d-1)}$ lands below the vanishing line rules out nonzero differentials from $h_*\llbracket \alpha,\beta \rrbracket$.
(On the other hand, $(E^r_{-p,q})^{fr}$ has a lower vanishing line, due to low-degree homology in $O(d)$.)
Because of the low bidegree, only a handful of incoming differentials need be considered.  
One can verify that all but one come from trivial groups in $E^r_{-p,q}$, which implies that those differentials are zero in 
$(E^r_{-p,q})^{fr}$.  The remaining one is $d^2: (E^2_{-3,5})^{fr} \to (E^2_{-5,6})^{fr}$ for $d=3$, which actually has a trivial source in $(E^2_{-3,5})^{fr}$.
(Alternatively, one could use the collapse of $\barE^r_{p,q}$ and $E^r_{p,q}$ at the second page for both $d>3$ \cite{LTV} and $d=3$ \cite{Songhafouo:2013, Moriya:2015}.)
Since $h_* \llbracket \alpha, \beta \rrbracket$ is nontrivial at $(E^2)^{fr}$ and supports no nonzero differentials, it is nontrivial in $(E^\infty_{-5, 3(d-1)})^{fr}$.  
\end{proof}

\section{Future directions}
\label{S:future}

We now briefly describe a possible future direction (in addition to the purely operadic ones indicated in Conjecture \ref{conj:action-on-Omega-G} and Question \ref{Q:infinity-operad} at the end of Section \ref{S:projection-onto-Cact^1}).
It concerns extensions of Theorem \ref{T:C} to classes of higher dimensions.
More specifically, one could use Theorem \ref{T:browder-compatible-space-level} to detect the nontriviality of Browder brackets of further classes in $H_*\K_d^{fr}$.  Classes $a$ and $b$ in dimensions that are multiples of $d-3$ can be obtained by resolving singular knots with $k$ double points, as done by Cattaneo, Cotta-Ramusino, and Longoni \cite{CCRL:AGT}.  These would correspond to expressions $\alpha$ and $\beta$ of the form $[x_{i_1}, x_{j_1}] \dots [x_{i_k}, x_{j_k}]$.  (In more recent work, Kosanovi\'c \cite{Kosanovic:graspers} also realizes classes in these dimensions via graspers.)  One could use Theorem \ref{T:browder-compatible-space-level} to show that the Browder bracket $[a,b]$ maps to $\llbracket \alpha, \beta \rrbracket$ in the homology spectral sequence, as in Theorem \ref{T:bracket-example}.  The least immediate step would be showing that $\llbracket \alpha, \beta \rrbracket$ is nontrivial 
in the spectral sequence.  It could be completed by applying the formula of Pelatt \cite{Pelatt:PhD} or emulating calculations of Turchin.

Such a class $[a,b]$ lies in a dimension $k(d-3)+1$, where few explicit constructions of nontrivial classes are known.
Nonetheless, calculations of Turchin \cite[Appendix B, Table 5]{Turchin:Bialgebra} and of Scannell and Sinha \cite[Table 2]{Scannell-Sinha} show that for even $d \geq 4$, $H_{5d-14}(\K^{fr}_d; \Q)$ is 1-dimensional, hence possibly generated by the bracket of known classes in dimensions $2(d-3)$ and $3(d-3)$.  For odd $d\geq 5$, Turchin's calculations \cite[Appendix B, Table 6]{Turchin:Bialgebra} show that $H_{4d-11}(\K^{fr}_d; \Q)$ is 2-dimensional (respectively, $H_{5d-14}(\K^{fr}_d; \Q)$ is 3-dimensional).  In this case, each of $H_{k(d-3)}(\K^{fr}_d; \Q)$ is nontrivial for all $k=1,\dots,4$ and the bracket of any two classes for which the values of $k$ add up to 4 (respectively, 5) lies in this group.

\bibliographystyle{alpha}   
\bibliography{refs}

\end{document}